\documentclass[a4paper,12pt]{article}
\usepackage[latin1]{inputenc}
\usepackage{subfigure}
\usepackage{amssymb,amsmath,latexsym,times,amsthm}
\usepackage{tikz,mathrsfs,lscape}
\usetikzlibrary{trees}
\usetikzlibrary{shapes.multipart,trees}
\usetikzlibrary{patterns,shapes.misc,mindmap}
\usepackage{pgflibraryshapes}
\usetikzlibrary{automata,arrows}
\usepackage[pdftex]{hyperref}
\hypersetup{plainpages=True, pdfstartview=FitV,
colorlinks=true,linkcolor=blue,citecolor=blue}
\def\dd#1{{\,\rm d}#1}
\def\ve{\varepsilon}
\def\JS{\mathscr{J\!\!S}}

\newtheorem{thm}{Theorem}[section]

\newtheorem{lmm}[thm]{Lemma}
\newtheorem{prop}[thm]{Proposition}

\newtheorem{definition}{Definition}

\author{Hsien-Kuei Hwang
\\Institute of Statistical Science
\\Academia Sinica
\\ Taipei 115
\\ Taiwan
\and Michael Fuchs
\\Department of Applied Mathematics
\\ National Chiao Tung
University
\\
Hsinchu 300
\\ Taiwan
\and Vytas Zacharovas
\\ Institute of
Statistical Science\\ Academia Sinica\\ Taipei 115\\ Taiwan}
\title{Asymptotic variance of random symmetric digital search trees}

\begin{document}
\maketitle

\begin{abstract}

Asymptotics of the variances of many cost measures in random digital
search trees are often notoriously messy and involved to obtain. A
new approach is proposed to facilitate such an analysis for several
shape parameters on random symmetric digital search trees. Our
approach starts from a more careful normalization at the level of
Poisson generating functions, which then provides an asymptotically
equivalent approximation to the variance in question. Several new
ingredients are also introduced such as a combined use of the Laplace
and Mellin transforms and a simple, mechanical technique for
justifying the analytic de-Poissonization procedures involved. The
methodology we develop can be easily adapted to many other problems
with an underlying binomial distribution. In particular, the less
expected and somewhat surprising $n(\log n)^2$-variance for certain
notions of total path-length is also clarified.

\end{abstract}

\medskip\centerline{\emph{Dedicated to the 60th birthday of Philippe
Flajolet}}

\section{Introduction}

The variance of a distribution provides an important measure of
dispersion of the distribution and plays a crucial and, in many
cases, a determinantal r\^ole in the limit law\footnote{The first
formal use of the term ``variance" in its statistical sense is
generally attributed to R. A. Fisher in his 1918 paper (see
\cite{fisher18a} or Wikipedia's webpage on variance), although its
practical use in diverse scientific disciplines predated this by a
few centuries (including closely-defined terms such as mean-squared
errors and standard deviations).}. Thus finding more effective means
of computing the variance is often of considerable significance in
theory and in practice. However, the calculation of the variance can
be computationally or intrinsically difficult, either because of the
messy procedures or cancellations involved, or because the
dependence structure is too strong or simply because no simple
manageable forms or reductions are available. We are concerned in
this paper with random digital trees for which asymptotic
approximations to the variance are often marked by heavy
calculations and long, messy expressions. This paper proposes a
general approach to simplify not only the analysis but also the
resulting expressions, providing new insight into the methodology;
furthermore, it is applicable to many other concrete situations and
leads readily to discover several new results, shedding new light on
the stochastic behaviors of the random splitting structures.

\paragraph{A binomial splitting process.}
The analysis of many splitting procedures in computer algorithms
leads naturally to a structural decomposition (in terms of the
cardinalities) of the form
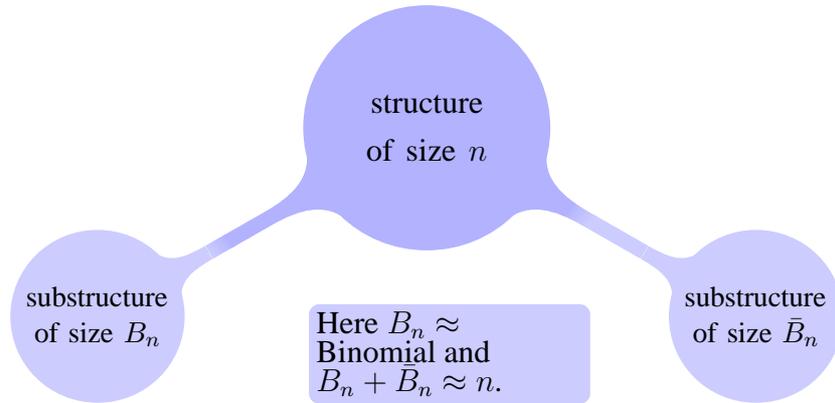
\begin{center}
\begin{tikzpicture}
\path[mindmap,concept color=blue!30,text=black] node[concept,minimum
size=2cm,text width=3cm]{\normalsize structure of size $n$}
child[concept color=blue!20, grow=210] {
    node[concept] (left) {substructure of size $B_n$} }
child[concept color=blue!20, grow=330] {
    node[concept] {substructure of size $\bar{B}_n$}};
\node [annotation,fill=blue!20] at (0.3,-3)
{\normalsize Here $B_n\approx$ Binomial and
$B_n+\bar{B}_n\approx n$.};
\end{tikzpicture}
\end{center}
where $B_n$ is essentially a binomial distribution (up to truncation
or small perturbations) and the sum of $B_n+\bar{B}_n$ is
essentially $n$.

Concrete examples in the literature include (see the books
\cite{drmota2009a,flajolet2009a,knuth1998a,
mahmoud1992a,szpankowski2001a} and below for more detailed
references)
\begin{itemize}
\item tries, contention-resolution tree algorithms, initialization
problem in distributed networks, and radix sort:
$B_n=\text{Binomial} (n;p)$ and $\bar{B}_n = n-B_n$, namely,
$\mathbb{P}(B_n=k)=\binom{n}{k} p^k q^{n-k}$ (here and throughout
this paper, $q:=1-p$);

\item bucket digital search trees (DSTs), directed diffusion-limited
aggregation on Bethe lattice, and Eden model: $B_n=\text{Binomial}
(n-b;p)$ and $\bar{B}_n = n-b-B_n$;

\item Patricia tries and suffix trees: $\mathbb{P}(B_n=k)
=\binom{n}{k} p^k q^{n-k}/(1-p^n-q^n)$ and $\bar{B}_n = n-B_n$.
\end{itemize}
Yet another general form arises in the analysis of multi-access
broadcast channel where
\[
    \left\{\begin{array}{l}
        B_n=\text{Binomial}(n;p)+\text{Poisson}(\lambda),\\
        \bar{B}_n = n-\text{Binomial}(n;p)+\text{Poisson}(\lambda),
    \end{array}\right.
\]
see \cite{fayolle86a,jacquet90a}. For some other variants, see
\cite{bai01a,chen03b,flajolet86e}. One reason of such a ubiquity of
binomial distribution is simply due to the binary outcomes (either
zero or one, either on or off, either positive or negative, etc.) of
many practical situations, resulting in the natural adaptation of
the Bernoulli distribution in the modeling.

\paragraph{Poisson generating function and the Poisson heuristic.}
A very useful, standard tool for the analysis of these binomial
splitting processes is the Poisson generating function
\[
    \tilde{f}(z)
    = e^{-z} \sum_{k\ge0} \frac{a_k}{k!}\, z^k,
\]
where $\{a_k\}$ is a given sequence, one distinctive feature being
the \emph{Poisson heuristic}, which predicts that
\begin{center}
    \framebox{
    \emph{If $a_n$ is smooth enough, then $a_n \sim \tilde{f}(n)$.}}
\end{center}
In more precise words, if the sequence $\{a_k\}$ does not grow too
fast (usually at most of polynomial growth) or does not fluctuate
too violently, then $a_n$ is well approximated by $\tilde{f}(n)$ for
large $n$. For example, if $\tilde{f}(z) = z^m$, $m=0,1,\dots$, then
$a_n \sim n^m$; indeed, in such a simple case, $a_n = n(n-1)\cdots
(n-m+1)$.

Note that the Poisson heuristic is itself a Tauberian theorem for the
Borel mean in essence; an Abelian type theorem can be found in
Ramanujan's Notebooks (see \cite[p.\ 58]{berndt1985a}).

From an elementary viewpoint, such a heuristic is based on the local
limit theorem of the Poisson distribution (or essentially Stirling's
formula for $n!$)
\[
    \frac{n^k}{k!}e^{-n}
    \sim \frac{e^{-x^2/2}}{\sqrt{2\pi n}}\left(1+\frac{x^3-3x}
    {6\sqrt{n}}+\cdots\right)\qquad(k=n+x\sqrt{n}),
\]
whenever $x=o(n^{1/6})$. Since $a_n$ is smooth, we then expect that
\begin{align*}
    \tilde{f}(n)
    \approx \sum_{\substack{k=n+x\sqrt{n}\\
    x=O(n^{\ve})}} a_k \frac{e^{-x^2/2}}{\sqrt{2\pi n}}
    \approx a_n \int_{-\infty}^\infty \frac{e^{-x^2/2}}
    {\sqrt{2\pi}} \dd x
    = a_n.
\end{align*}

On the other hand, by Cauchy's integral representation, we also have
\begin{align*}
    a_n
    &= \frac{n!}{2\pi i} \oint_{|z|=n}
    z^{-n-1} e^z\tilde{f}(z) \dd z \\
    &\approx \tilde{f}(n)
    \frac{n!}{2\pi i} \oint_{|z|=n} z^{-n-1} e^z \dd z\\
    &= \tilde{f}(n),
\end{align*}
since the saddle-point $z=n$ of the factor $z^{-n}e^z$ is unaltered
by the comparatively more smooth function $\tilde{f}(z)$.

\paragraph{The Poisson-Charlier expansion.}
The latter analytic viewpoint provides an additional advantage of
obtaining an expansion by using the Taylor expansion of $\tilde{f}$
at $z=n$, yielding
\begin{align}\label{PC}
    a_n
    =\sum_{j\ge0} \frac{\tilde{f}^{(j)}(n)}{j!}\tau_j(n),
\end{align}
where
\begin{align*}
    \tau_j(n)
    := n![z^n] (z-n)^j e^z
    = \sum_{0\le \ell \le j}
    \binom{j}{\ell} (-1)^{j-\ell} \frac{n!n^{j-\ell}}{(n-\ell)!}
    \qquad(j=0,1,\dots),
\end{align*}
and $[z^n]\phi(z)$ denotes the coefficient of $z^n$ in the Taylor
expansion of $\phi(z)$. We call such an expansion \emph{the
Poisson-Charlier expansion} since the $\tau_j$'s are essentially the
Charlier polynomials $C_j(\lambda,n)$ defined by
\[
    C_j(\lambda,n)
    := \lambda^{-n}n![z^n](z-1)^je^{\lambda z},
\]
so that $\tau_j(n) = n^j C_j(n,n)$. For other terms used in the
literature, see \cite{flajolet2009a,hald2002a}.

The first few terms of $\tau_j(n)$ are given as follows.
\begin{center}
\begin{tabular}{|c|c|c|c|c|c|c|}\hline
$\tau_0(n)$ & $\tau_1(n)$ & $\tau_2(n)$ & $\tau_3(n)$ & $\tau_4(n)$
&$\tau_5(n)$ & $\tau_6(n)$ \\ \hline $1$ &  $0$ & $-n$  & $2n$ &
$3n(n-2)$ & $-4n(5n-6)$ & $-5n(3n^2-26n+24)$ \\ \hline
\end{tabular}
\end{center}
It is easily seen that $\tau_j(n)$ is a polynomial in $n$ of degree
$\lfloor j/2\rfloor$.

The meaning of such a Poisson-Charlier expansion becomes readily
clear by the following simple but extremely useful lemma.
\begin{lmm} \label{lm-PC}
Let $\tilde{f}(z) := e^{-z}\sum_{k\ge0} a_k z^k/k!$. If $\tilde{f}$
is an entire function, then the Poisson-Charlier expansion
(\ref{PC}) provides an identity for $a_n$.
\end{lmm}
\begin{proof}
Since $\tilde{f}$ is entire, we have
\[
    \sum_{n\ge0} \frac{a_n}{n!}z^n
    = e^z \tilde{f}(z)
    = e^z\sum_{j\ge0} \frac{\tilde{f}^{(j)}(n)}{j!}(z-n)^j,
\]
and the lemma follows by absolute convergence.
\end{proof}

Two specific examples are worthy of mention here as they speak
volume of the difference between identity and asymptotic
equivalence. Take first $a_n = (-1)^n$. Then the Poisson heuristic
fails since $(-1)^n \not\sim e^{-2n}$, but, by Lemma~\ref{lm-PC}, we
have the identity
\[
    (-1)^n
    = e^{-2n} \sum_{j\ge0} \frac{(-2)^j}{j!}\,\tau_j(n).
\]
See Figure~\ref{fig-cv} for a plot of the convergence of the series
to $(-1)^n$.
\begin{figure}[h!]
\begin{center}
\begin{tikzpicture}[xscale=.05,yscale=3]
\foreach \x/\xtext in { 20 , 40, ..., 100} \draw (\x,.5pt) -- (\x
,-.5pt) node[anchor=north] {$\xtext$}; \foreach \y/\ytext in
{0.2,0.4, 0.6,0.8,1.0} \draw (1,\y ) -- (-1,\y ) node[anchor=east]
{$\ytext$};
\draw[-,line width=.5pt] (-1,0) -- (110,0) node[right] {};%
\draw[-,line width=.5pt] (0,-.05) -- (0,1.2) node[above] {};%
\draw[line width=.7pt,smooth,color=blue] plot coordinates{(0,
0.0000)(1, 0.0000)(2, 0.0000)(3, 0.0000)(4, 0.0000) (5, 0.0000)(6,
0.0000)(7, 0.0000)(8, 0.0000)(9, 0.0000) (10, 0.0001)(11,
0.0001)(12, 0.0000)(13, 0.0003)(14, 0.0007) (15, 0.0005)(16,
0.0009)(17, 0.0029)(18, 0.0037)(19, 0.0010) (20, 0.0058)(21,
0.0136)(22, 0.0161)(23, 0.0074)(24, 0.0135) (25, 0.0398)(26,
0.0581)(27, 0.0544)(28, 0.0212)(29, 0.0379) (30, 0.1078)(31,
0.1669)(32, 0.1940)(33, 0.1749)(34, 0.1060) (35, 0.0061)(36,
0.1474)(37, 0.3013)(38, 0.4522)(39, 0.5885) (40, 0.7032)(41,
0.7942)(42, 0.8626)(43, 0.9115)(44, 0.9449) (45, 0.9669)(46,
0.9807)(47, 0.9891)(48, 0.9940)(49, 0.9968)}; \draw[line
width=.7pt,color=blue] plot coordinates{(49, 0.9968)
(50, 0.9984)(100, 1.0000)};%
\end{tikzpicture}
\begin{tikzpicture}[xscale=.05,yscale=3]
\foreach \x/\xtext in { 20 , 40, ..., 100} \draw (\x,.5pt) -- (\x
,-.5pt) node[anchor=north] {$\xtext$}; \foreach \y/\ytext in
{-1,-0.8, -0.6,-0.4,-0.2,0, 0.2} \draw (.5,\y ) -- (-.5,\y )
node[anchor=east] {$\ytext$};
\draw[-,line width=.5pt] (-1,0) -- (105,0) node[right] {};%
\draw[-,line width=.5pt] (0,-1.05) -- (0,.22) node[above] {};%
\draw[line width=.7pt,smooth,color=red] plot coordinates{(0,
0.0000)(1, 0.0000)(2, 0.0000)(3, 0.0000)(4, 0.0000)(5, 0.0000) (6,
0.0000)(7, 0.0000)(8, 0.0000)(9, 0.0000)(10, 0.0000)(11, 0.0000)
(12, 0.0000)(13, -0.0001)(14, -0.0002)(15, 0.0000)(16, 0.0004) (17,
0.0009)(18, 0.0007)(19, -0.0007)(20, -0.0030)(21, -0.0043) (22,
-0.0025)(23, 0.0037)(24, 0.0122)(25, 0.0175)(26, 0.0135) (27,
-0.0030)(28, -0.0285)(29, -0.0527)(30, -0.0620)(31, -0.0457) (32,
-0.0013)(33, 0.0629)(34, 0.1302)(35, 0.1803)(36, 0.1954) (37,
0.1652)(38, 0.0888)(39, -0.0258)(40, -0.1653)(41, -0.3146) (42,
-0.4599)(43, -0.5909)(44, -0.7018)(45, -0.7904)(46, -0.8578) (47,
-0.9067)(48, -0.9408)(49, -0.9636)(50, -0.9783)(51, -0.9874)};
\draw[line width=.7pt,smooth,color=red] plot coordinates{
(51, -0.9874)(100, -1)};%
\end{tikzpicture}
\end{center}
\caption{\emph{Convergence of $e^{-2n}\sum_{j\le k}
(-2)^j\tau_j(n)/j!$ to $(-1)^n$ for $n=10$ (left) and $n=11$ (right)
for increasing $k$.}} \label{fig-cv}
\end{figure}
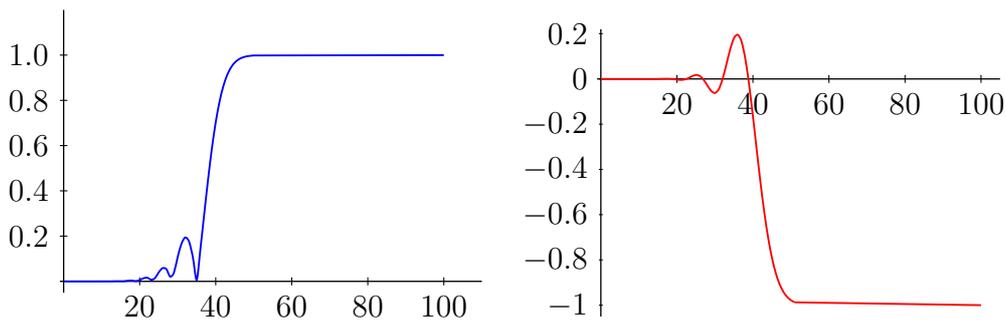

Now if $a_n = 2^n$, then $2^n \not\sim e^n$, but we still have
\[
    2^n
    = e^n \sum_{j\ge0} \frac{\tau_j(n)}{j!}.
\]

So when is the Poisson-Charlier expansion also an asymptotic
expansion for $a_n$, in the sense that dropping all terms with $j\ge
2\ell$ introduces an error of order $\tilde{f}^{(2\ell)}n^\ell$
(which in typical cases is of order $\tilde{f}(n) n^{-\ell}$)? Many
sufficient conditions are thoroughly discussed in \cite{jacquet98a},
although the terms in their expansions are expressed differently;
see also \cite{szpankowski2001a}.

\paragraph{Poissonized mean and variance.} The majority of random
variables analyzed in the algorithmic literature are at most of
polynomial or sub-exponential (such as $e^{c(\log n)^2}$ or $e^{c
n^{1/2}}$) orders, and are smooth enough. Thus the Poisson
generating functions of the moments are often entire functions. The
use of the Poisson-Charlier expansion is then straightforward, and
in many situations it remains to justify the asymptotic nature of
the expansion.

For convenience of discussion, let $\tilde{f}_m(z)$ denote the
Poisson generating function of the $m$-th moment of the random
variable in question, say $X_n$. Then by Lemma~\ref{lm-PC}, we have
the identity
\[
    \mathbb{E}(X_n)
    = \sum_{j\ge0}\frac{\tilde{f}_1^{(j)}(n)}{j!}\tau_j(n),
\]
and for the second moment
\begin{align}\label{sm-PC}
    \mathbb{E}(X_n^2)
    = \sum_{j\ge0}\frac{\tilde{f}_2^{(j)}(n)}{j!}\tau_j(n),
\end{align}
provided only that the two Poisson generating functions
$\tilde{f}_1$ and $\tilde{f}_2$ are entire functions.

These identities suggest that a good approximation to the variance
of $X_n$ be given by
\[
    \mathbb{V}(X_n)
    =  \mathbb{E}(X_n^2) - (\mathbb{E}(X_n))^2
    \approx \tilde{f}_2(n) - \tilde{f}_1(n)^2,
\]
which holds true for many cost measures, where we can indeed replace
the imprecise, approximately equal symbol ``$\approx$" by the more
precise, asymptotically equivalent symbol ``$\sim$". However, for a
large class of problems for which the variance is essentially
linear, meaning roughly that
\begin{align} \label{ess-linear}
    \lim_{n\to\infty} \frac{\log \mathbb{V}(X_n)}{\log n}
    = 1,
\end{align}
the Poissonized variance $\tilde{f}_2(n) - \tilde{f}_1(n)^2$ is not
asymptotically equivalent to the variance. This is the case for the
total cost of constructing random digital search trees, for example.
One technical reason is that there are additional cancellations
produced by dominant terms. The next question is then: can we find a
better normalized function so that the variance is asymptotically
equivalent to its value at $n$?

\paragraph{Poissonized variance with correction.} The crucial step of
our approach that is needed when the variance is essentially linear
is to consider
\begin{align}\label{Vz}
    \tilde{V}(z)
    :=  \tilde{f}_2(z) - \tilde{f}_1(z)^2 - z\tilde{f}_1'(z)^2,
\end{align}
and it then turns out that
\[
    \mathbb{V}(X_n)
    = \tilde{V}(n) + O((\log n)^c),
\]
in all cases we consider for some $c\ge0$. The asymptotics of the
variance is then reduced to that of $\tilde{V}(z)$ for large $z$,
which satisfies, up to non-homogeneous terms, the same type of
equation as $\tilde{f}_1(z)$. Thus the same tools used for analyzing
the mean can be applied to $\tilde{V}(z)$.

To see how the last correction term $z\tilde{f}_1'(z)^2$ appears, we
write $\tilde{D}(z) :=\tilde{f}_2(z)- \tilde{f}_1(z)^2$, so that
$\tilde{f}_2(z)=\tilde{D}(z) + \tilde{f}_1(z)^2$, and we obtain, by
substituting this into (\ref{sm-PC}),
\begin{align*}
    \mathbb{V}(X_n)
    &= \mathbb{E}(X_n^2) - (\mathbb{E}(X_n))^2  \\
    &=\sum_{j\ge0} \frac{\tilde{f_2}^{(j)}(n)}{j!}
    \tau_j(n)- \left(\sum_{j\ge0} \frac{\tilde{f_1}^{(j)}(n)}{j!}
    \tau_j(n)\right)^2\\
    &= \tilde{D}(n) - n\tilde{f}_1'(n)^2
    -\frac{n}2\tilde{D}''(n)+ \text{smaller-order terms}.
\end{align*}

Now take $\tilde{f}_1(n) \asymp n\log n$. Then the first term
following $\tilde{D}(n)$ is generally not smaller than
$\tilde{D}(n)$ because
\[
    n\tilde{f}_1'(n)^2
    \asymp n(\log n)^2,
\]
while $\tilde{D}(n)\asymp n(\log n)^2$, at least for the examples we
discuss in this paper. Note that the variance is in such a case
either of order $n\log n$ or of order $n$. Thus to get an
asymptotically equivalent approximation to the variance, we need at
least an additional correction term, which is exactly
$n\tilde{f}_1'(n)^2$.

The correction term $n\tilde{f}_1'(n)^2$ already appeared in
many early papers by Jacquet and R\'egnier (see \cite{jacquet88a}).

\paragraph{A viewpoint from the asymptotics of the characteristic
function.} Most binomial recurrences of the form
\begin{align}\label{bino-rr}
    X_n
    \stackrel{d}{=} X_{B_n} + X_{\bar{B}_n}^* + T_n,
\end{align}
as arising from the binomial splitting processes discussed above are
asymptotically normally distributed, a property partly ascribable to
the highly regular behavior of the binomial distribution. Here the
$(X_n^*)$ are independent copies of the $(X_n)$ and the random or
deterministic non-homogeneous part $T_n$ is often called the
``toll-function," measuring the cost used to ``conquer" the two
subproblems. Such recurrences have been extensively studied in
numerous papers; see \cite{jacquet98a,neininger04a,schachinger95b,
schachinger01a} and the references therein.

The correction term we introduced in (\ref{Vz}) for Poissonized
variance also appears naturally in the following heuristic, formal
analysis, which can be justified when more properties are available.
By definition and formal expansion
\begin{align*}
    e^{-z} \sum_{n\ge0} \mathbb{E}\left(e^{X_n i\theta}\right)
    \frac{z^n}{n!}
    &= \sum_{m\ge0} \frac{\tilde{f}_m(z)}{m!} (i\theta)^m \\
    &= \exp\left(\tilde{f}_1(z)i\theta - \frac{\tilde{D}(z)}{2}
    \theta^2 + \cdots\right),
\end{align*}
where $\tilde{D}(z) :=\tilde{f}_2(z)- \tilde{f}_1(z)^2$, we have
\begin{align*}
    \mathbb{E}\left(e^{(X_n-\tilde{f}_1(n))i\theta}\right)
    \approx \frac{n!}{2\pi i} \oint_{|z|=n} z^{-n-1} \exp\left(z+
    \left(\tilde{f}_1(z)-\tilde{f}_1(n)\right)i\theta
    - \frac{\tilde{D}(z)}{2} \theta^2 + \cdots\right) \dd z.
\end{align*}
Observe that with $z=ne^{it}$, we have the local expansion
\[
    ne^{it}-nit+\left(\tilde{f}_1(ne^{it})
    -\tilde{f}_1(n)\right)i\theta
    - \frac{\tilde{D}(ne^{it})}{2} \theta^2
    = n-\frac{nt^2}{2} - n\tilde{f}_1'(n) t \theta
    - \frac{\tilde{D}(n)}{2} \theta^2 + \cdots,
\]
for small $t$. It follows that
\begin{align*}
    \mathbb{E}\left(e^{(X_n-\tilde{f}_1(n))i\theta}\right)
    &\approx \frac{n!n^{-n}e^n}{2\pi} \exp\left(-
    \frac{\tilde{D}(n)}{2} \theta^2 \right)\int_{-\ve}^{\ve}
    \exp\left(-\frac{nt^2}{2} - n\tilde{f}_1'(n) t \theta \right)
    \dd t\\
    &\sim \exp\left(- \frac{\theta^2}{2}\left(\tilde{D}(n)
    -n\tilde{f}_1'(n)^2\right)\right),
\end{align*}
by extending the integral to $\pm\infty$ and by completing the
square. This again shows that $n\tilde{f}_1'(n)^2$ is the right
correction term for the variance. For more precise analysis of
this type, see \cite{jacquet98a}.

\paragraph{A comparison of different approaches to the asymptotic
variance.} What are the advantages of the Poissonized variance with
correction? In the literature, a few different approaches have been
adopted for computing the asymptotics of the variance of the
binomial splitting processes.
\begin{itemize}
\item Second moment approach: this is the most straightforward means
and consists of first deriving asymptotic expansions of sufficient
length for the expected value and for the second moment, then
considering the difference $\mathbb{E}(X_n^2) -
(\mathbb{E}(X_n))^2$, and identifying the lead terms after
cancellations of dominant terms in both expansions. This approach is
often computationally heavy as many terms have to be cancelled;
additional complication arises from fluctuating terms, rendering the
resulting expressions more messy. See below for more references.

\item Poissonized variance: the asymptotics of the variance is
carried out through that of $\tilde{D}(n) = \tilde{f}_2(n) -
\tilde{f}_1(n)^2$. The difference between this approach and the
previous one is that no asymptotics of $\tilde{f}_2(n)$ is derived
or needed, and one always focuses directly on considering the
equation (functional or differential) satisfied by $\tilde{D}(z)$.
As we discussed above, this does not give in many cases an
asymptotically equivalent estimate for the variance, because
additional cancellations have to be further taken into account; see
for instance \cite{jacquet88a,jacquet95a,jacquet98a}.

\item Characteristic function approach: similar to the formal
calculations we carried out above, this approach tries to derive a
more precise asymptotic approximation to the characteristic function
using, say complex-analytic tools, and then to identify the right
normalizing term as the variance; see the survey \cite{jacquet98a}
and the papers cited there.

\item Schachinger's differencing approach: a delicate, mostly
elementary approach based on the recurrence satisfied by the
variance was proposed in \cite{schachinger95b} (see also
\cite{schachinger01a}). His approach is applicable to very general
``toll-functions" $T_n$ in (\ref{bino-rr}) but at the price of less
precise expressions.

\end{itemize}
The approach we use is similar to the Poissonized variance one but
the difference is that the passage through $\tilde{D}(z)$ is
completely avoided and we focus directly on equations satisfied
by $\tilde{V}(z)$ (defined in (\ref{Vz})).

In contrast to Schachinger's approach, our approach, after starting
from defining $\tilde{V}(z)$, is mostly analytic. It yields then
more precise expansions, but more properties of $T_n$ have to be
known. The contrast here between elementary and analytic approaches
is thus typical; see, for example, \cite{chern07a,chern02a}. See
also Appendix for a brief sketch of the asymptotic linearity of the
variance by elementary arguments.

Additional advantages that our approach offer include comparatively
simpler forms for the resulting expressions, including Fourier
series expansions, and general applicability (coupling with the
introduction of several new techniques).

\paragraph{Organization of this paper.}
This paper is organized as follows. We start with the variance of
the total path-length of random digital search trees in the next
section, which was our motivating example. We then extend the
consideration to bucket DSTs for which two different notions of
total path-length are distinguished, which result in very different
asymptotic behaviors. The application of our approach to several
other shape parameters are discussed in Section~\ref{DST-II}.
Table~\ref{tb-all-pl} summarizes the diverse behaviors exhibited
by the means and the variances of the shape parameters we consider
in this paper.

\begin{table}[!h]
\begin{center}
\begin{tabular}{|c|c|c|} \hline
Shape parameters & mean & variance \\ \hline\hline
Internal PL & $n\log n$ & $n$ \\ \hline
Key-wise PL$^*$ & $n\log n$ & $n$ \\ \hline
Node-wise PL$^*$ & $n\log n$ & $n(\log n)^2$ \\ \hline
Peripheral PL & $n$ & $n$ \\ \hline
$\#$(leaves) & $n$ & $n$ \\ \hline
Differential PL & $n$ & $n\log n$ \\ \hline
Weighted PL & $n(\log n)^{m+1}$ & $n$ \\ \hline
\end{tabular}
\end{center}
\caption{\emph{Orders of the means and the variances of all shape
parameters in this paper; those marked with an $^*$ are for $b$-DSTs
with $b\ge2$. Here PL denotes path-length and $m\ge0$.}}
\label{tb-all-pl}
\end{table}

Applications of the approach we develop here to other classes of
trees and structures, including tries, Patricia tries, bucket sort,
contention resolution algorithms, etc., will be investigated in a
future paper.

\section{Digital Search Trees}\label{dst}

We start in this section with a brief description of digital search
trees (DSTs), list major shape parameters studied in the literature,
and then focus on the total path-length. The approach we develop is
also very useful for other linear shape measures, which is
discussed in a more systematic form in the following sections.

\subsection{DSTs}
DSTs were first introduced by Coffman and Eve in \cite{coffman70a}
in the early 1970's under the name of sequence hash trees. They can
be regarded as the bit-version of binary search trees (thus the
name); see \cite[p.\ 496 \emph{et seq.}]{knuth1998a}. Given a
sequence of binary strings, we place the first in the root node;
those starting with ``$0$" (``$1$") are directed to the left (right)
subtree of the root, and are constructed recursively by the same
procedure but with the removal of their first bits when comparisons
are made. See Figure~\ref{fg-dst} for an illustration.

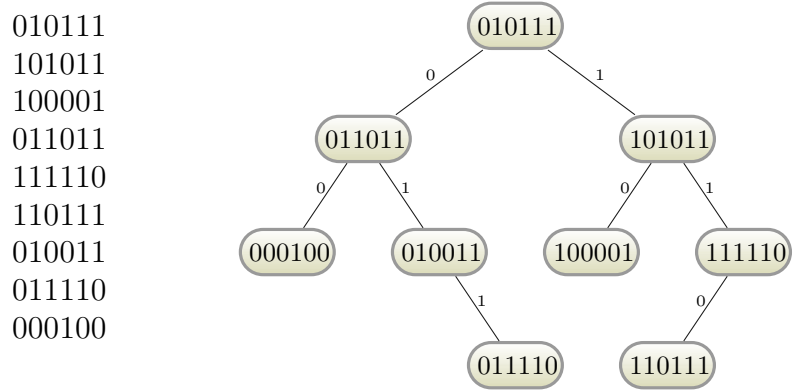
\begin{figure}[h!]
\begin{center}
\begin{tikzpicture}[
    s1/.style={
    rectangle,
    rounded corners=3mm,
    minimum size=6mm,
    very thick,
    draw=white!50!black!80,
    top color=white,
    bottom color=yellow!50!black!30
}]
\path node at (1,0) {$010111$};%
\path node at (1,-.5) {$101011$};%
\path node at (1,-1) {$100001$};%
\path node at (1,-1.5) {$011011$};%
\path node at (1,-2) {$111110$};%
\path node at (1,-2.5) {$110111$};%
\path node at (1,-3) {$010011$};%
\path node at (1,-3.5) {$011110$};%
\path node at (1,-4) {$000100$};%
\footnotesize %
\path node(a) at (7,0) [s1, text width=1cm, text centered]{$010111$};%
\path node(b) at (9,-1.5) [s1, text width=1cm, text centered]{$101011$};%
\path node(c) at (8,-3) [s1, text width=1cm, text centered]{$100001$};%
\path node(d) at (5,-1.5) [s1, text width=1cm, text centered]{$011011$};%
\path node(e) at (10,-3) [s1, text width=1cm, text centered]{$111110$};%
\path node(f) at (9,-4.5) [s1, text width=1cm, text centered]{$110111$};%
\path node(g) at (6,-3) [s1, text width=1cm, text centered]{$010011$};%
\path node(h) at (7,-4.5) [s1, text width=1cm, text centered]{$011110$};%
\path node(i) at (4,-3) [s1, text width=1cm, text centered]{$000100$};%
\path [draw,-,black!90] (a) -- (b) node[above,pos=.6,black]{{\tiny $1$}};%
\path [draw,-,black!90] (b) -- (c) node[above,pos=.6,black]{{\tiny $0$}};%
\path [draw,-,black!90] (a) -- (d) node[above,pos=.6,black]{{\tiny $0$}};%
\path [draw,-,black!90] (b) -- (e) node[above,pos=.6,black]{{\tiny $1$}};%
\path [draw,-,black!90] (e) -- (f) node[above,pos=.6,black]{{\tiny $0$}};%
\path [draw,-,black!90] (d) -- (g) node[above,pos=.6,black]{{\tiny $1$}};%
\path [draw,-,black!90] (g) -- (h) node[above,pos=.6,black]{{\tiny $1$}};%
\path [draw,-,black!90] (d) -- (i) node[above,pos=.6,black]{{\tiny $0$}};
\end{tikzpicture}
\end{center}
\caption{\emph{A digital search tree of nine binary strings.}}
\label{fg-dst}
\end{figure}

While the practical usefulness of digital search trees is limited,
they represent one of the simplest, fundamental, prototype models
for divide-and-conquer algorithms using coin-tossing or similar
random devices. Of notable interest is its close connection to the
analysis of Lempel-Ziv compression scheme that has found
widespread incorporation into numerous softwares. Furthermore, the
mathematical analysis is often challenging and leads to intriguing
phenomena. Also the splitting mechanism of DSTs appeared naturally
in a few problems in other areas; some of these are mentioned in
the last section.

\paragraph{Random digital search trees.} The simplest random model
we discuss in this paper is the independent, Bernoulli model. In
this model, we are given a sequence of $n$ independent and
identically distributed random variables, each comprising an
infinity sequence of Bernoulli random variables with mean $p$,
$0<p<1$. The DST constructed from the given random sequence of
binary strings is called a \emph{random DST}. If $p=1/2$, the DST is
said to be \emph{symmetric}; otherwise, it is \emph{asymmetric}. We
focus on symmetric DSTs in this paper for simplicity; extension to
asymmetric DSTs is possible but much harder.

Stochastic properties of many shape characteristics of random DSTs
are known. Almost all of them fall into one of the two categories,
according to their growth order being logarithmic or essentially
linear (in the sense of (\ref{ess-linear})), which we simply refer
to as ``log shape measures" and ``linear shape measures".

\paragraph{Log shape measures.}
The two major parameters studied in this category are \emph{depth},
which is the distance of the root to a randomly chosen node in the
tree (each with the same probability), and \emph{height}, which
counts the number of nodes from the root to one of the longest
paths. Both are of logarithmic order in mean. Depth provides a good
indication of the typical cost needed when inserting a new key in
the tree, while height measures the worst possible cost that may be
needed.

Depth was first studied in \cite{konheim73a} in connection with the
\emph{profile}, which is the sequence of numbers, each enumerating
the number of nodes with the same distance to the root. For example,
the tree \tikz[level distance=10pt,sibling distance=2pt, every
node/.style={fill=black,circle,inner sep=.5pt},grow=0] \node[circle]
{}
  child { node{}
      child { node{} child{ node{} }}
      child { node{}
          child { node{}
              child { node{} }
              child { node{} }
              child { node{} }
          }
      }
  }
  child { node{}
      child[missing]
      child{ node{}}
}; has the profile $\{1,2,3,2,3\}$. For other papers on the depth of
random DSTs, see
\cite{dennert07a,devroye92a,devroye99a,jacquet01a,janson06a,
kirschenhofer88d,knuth1998a,louchard87a,louchard94a,mahmoud1992a,
pittel86a,szpankowski88b,szpankowski91a}. The height of random
DSTs is addressed in \cite{devroye99a,drmota02a,knessl00b,
mahmoud1992a,pittel86a}.

\paragraph{Linear shape measures.} These include the total internal
path-length, which sums the distance between the root and every
node, and the occurrences of a given pattern (leaves or nodes
satisfying certain properties); see
\cite{flajolet92a,flajolet86d,hubalek00a,hubalek02a,jacquet95a,
kirschenhofer88c,kirschenhofer94a,knuth1998a}.

The profile contains generally much more information than most other
shape measures, and it can to some extent be regarded as a good
bridge connecting log and linear measures; see
\cite{drmota2009a,drmota09b,konheim73a,louchard87a} for known
properties concerning expected profile of random DSTs.

Nodes of random DSTs with $p=1/2$ are distributed in an extremely
regular way, as shown in Figures~\ref{fg-rnd-dst} and
\ref{fg-rnd-dst2}.

\subsection{Known and new results for the total internal path-length}

Throughout this section, we focus on $X_n$, the total path length of
a random digital search tree built from $n$ binary strings. By
definition and by our random assumption, $X_n$ can be computed
recursively by
\begin{equation}\label{rec-xn-dst}
    X_{n+1}
    \stackrel{d}{=}X_{B_n}+X_{n-B_n}^{*}+n,\qquad (n\ge 0)
\end{equation}
with the initial condition $X_0=0$, since removing the root results
in a decrease of $n$ for the total path length (each internal node
below the root contributes $1$). Here
$B_n\sim\text{Binomial}(n;1/2), X_n\stackrel{d}{=}X_n^*$, and
$X_n,X_n^{*},B_n$ are independent.

\include{DST_500_1100_colored}
\include{DST-1000-Circle-Spiral}

\paragraph{Known results.}
It is known that (see \cite{flajolet86d,hubalek00a,prodinger92b})
\begin{align}\label{mean-asymp}
\begin{split}
    \mathbb{E}(X_n)
    &= (n+1)\log_2n + n\left(
    \frac{\gamma-1}{\log 2} + \frac12-c_1
    +\varpi_1(\log_2n)\right) \\
    &\qquad +\frac{\gamma-1/2}{ \log 2} + \frac52-c_1
    + \varpi_2(\log_2n) +O\left(n^{-1}\log n\right),
\end{split}
\end{align}
where $\gamma$ denotes Euler's constant, $c_1 :=
\sum_{k\ge1} (2^k-1)^{-1}$, and $\varpi_1(t), \varpi_2(t)$ are
$1$-periodic functions with zero mean whose Fourier expansions are
given by ($\chi_k := 2k\pi i/L$, $L := \log 2$)
\begin{align}
    \varpi_1(t)
    &= \frac1{L}\sum_{k\not=0} \Gamma\left(-1-\chi_k\right)
    e^{2k\pi i t}, \label{F1u}\\
    \varpi_2(t)
    &= -\frac1{L} \sum_{k\not=0}\left(1-\frac{\chi_k}2\right)
    \Gamma(-\chi_k) e^{2k\pi i t},\nonumber
\end{align}
respectively. Here $\Gamma$ denotes the Gamma function. Thus we see
roughly that \emph{random digital search trees under the unbiased
Bernoulli model are highly balanced in shape}. An important feature
of the periodic functions is that they are marked by very small
amplitudes of fluctuation: $|\varpi_1(t)|\le 3.4\times 10^{-8}$ and
$|\varpi_2(t)|\le 3.4\times 10^{-6}$. Such a quasi-flat (or smooth)
behavior may in practice be very likely to lead to wrong conclusions
as they are hardly visible from simulations of moderate sample sizes.

\begin{figure}[!h]
\begin{center}
\begin{tikzpicture}[xscale=1,yscale=15]
\draw[color=red, line width=1.5pt] plot
coordinates{(1.000000,0)(1.584963,0.083333)(2,0.121094)
(2,0.121094)(3,0.182472)(3.584963,0.207897)(4,0.221225)
(4.584963,0.235443)(5,0.24275)(5.321928,0.247253)
(5.584963,0.250313)(5.807355,0.2525)(6,0.254142)
(6.169925,0.25543)(6.321928,0.256473)(6.459432,0.257334)
(6.584963,0.258052)(6.70044,0.258658)(6.807355,0.259175)
(6.906891,0.259621)(7,0.260011)(7.087463,0.260357)
(7.169925,0.260666)(7.247928,0.260945)(7.321928,0.261197)
(7.392317,0.261427)(7.459432,0.261636)(7.523562,0.261827)
(7.584963,0.262002)(7.643856,0.262163)(7.70044,0.26231)
(7.754888,0.262446)(7.807355,0.262571)(7.857981,0.262687)
(7.906891,0.262795)(7.954196,0.262896)(8,0.262991)
(8.129283,0.263244)(8.247928,0.26346)(8.357552,0.263648)
(8.459432,0.263812)(8.554589,0.263955)(8.643856,0.26408)
(8.72792,0.264189)(8.807355,0.264284)(8.882643,0.264369)
(8.954196,0.264446)(9.022368,0.264515)(9.087463,0.264578)
(9.149747,0.264637)(9.209453,0.264692)(9.266787,0.264743)
(9.321928,0.264791)(9.375039,0.264837)(9.426265,0.264879)
(9.475733,0.264919)(9.523562,0.264956)(9.569856,0.26499)
(9.61471,0.265023)(9.658211,0.265052)(9.70044,0.26508)
(9.741467,0.265106)(9.78136,0.26513)(9.820179,0.265152)
(9.857981,0.265173)(9.894818,0.265193)(9.930737,0.265212)
(9.965784,0.265229)(10,0.265246)};
\foreach \x/\xtext in {2,3,...,9}
 \draw[shift={(\x,0)}, line width=.5pt]
 (0,-.005) -- (0,.005) node[below=2pt] {\tiny$\xtext$};%
\foreach \x in {1.5,2.5,...,9.5}
 \draw[shift={(\x,0)}, line width=.5pt]
 (0,-.005) -- (0,.005) node[below] {};%
\foreach \y/\ytext in {0.1,0.2}
 \draw[shift={(1,\y)}, line width=.5pt]
 (.1,0) -- (-.1,0) node[left]{\tiny$\ytext$};%
\foreach \y in {0.05,0.15,0.25}
 \draw[shift={(1,\y)}, line width=.5pt]
 (.1,0) -- (-.1,0) node[left]{};%
\foreach \y in {0.2,0.1}
 \draw[shift={(10,\y)}, line width=.5pt]
 (.1,0) -- (-.1,0) ;%
\foreach \y in {0.25,0.15,0.05}
 \draw[shift={(10,\y)}, line width=.5pt]
 (.1,0) -- (-.1,0) node[left]{};%
\foreach \y/\ytext in {0.25/-0.1,0.15/-0.2,0.05/-0.3}
 \draw[shift={(10,\y)}, line width=.5pt]
     node[right]{\tiny$\ytext$};%
\draw (0.8,-0.01) node{\tiny$1$};%
\draw (9.8,-0.01) node{\tiny$10$};%
\draw (8,0.24) node{$\mathbb{V}(X_n)/n$};%
\draw (8,0.03) node{$\mathbb{E}(X_n)/(n+1)-\log_2n$};%
\draw[-latex, line width=.5pt] (0.5,0) -- (10.5,0) node[right] {};%
\draw[-latex, line width=.5pt] (1,-.02) -- (1,0.3) node[above] {};%
\draw[-latex, line width=.5pt] (10,-.02) -- (10,0.3) node[above] {};%
\draw[line width=.5pt] (9.9,0.0072024) -- (10.1,0.0072024)
node[right]{} ;%
\draw[color=blue!100, line width=1.5pt] plot
coordinates{(1,0.2166666)(1.584963,0.1580074)(2.807355,0.0782712)
(3.459432,0.0535818)(3.906891,0.0415548)(4.247928,0.0344368)
(4.523562,0.029731)(4.754888,0.0263892)(4.954196,0.0238934)
(5.285402,0.0204146)(5.554589,0.0181048)(5.78136,0.0164598)
(5.97728,0.0152286)(6.149747,0.0142724)(6.321928,0.0134238)
(6.459432,0.012814)(6.584963,0.0123054)(6.70044,0.0118744)
(6.807355,0.0115046)(6.906891,0.011184)(7,0.0109034)
(7.044394,0.0107758)(7.129283,0.0105422)(7.209453,0.0103338)
(7.285402,0.0101468)(7.357552,0.0099778)(7.400879,0.0098804)
(7.459432,0.0097532)(7.523562,0.0096196)(7.584963,0.009497)
(7.643856,0.0093844)(7.70044,0.0092802)(7.754888,0.009184)
(7.807355,0.0090944)(7.857981,0.009011)(7.906891,0.0089332)
(7.954196,0.0088604)(8,0.0087922)(8.129283,0.0086108)
(8.247928,0.008458)(8.357552,0.0083276)(8.459432,0.0082148)
(8.554589,0.0081164)(8.643856,0.0080298)(8.72792,0.007953)
(8.807355,0.0078844)(8.882643,0.0078228)(8.954196,0.0077672)
(9.022368,0.0077166)(9.087463,0.0076706)(9.149747,0.0076284)
(9.209453,0.0075896)(9.266787,0.007554)(9.321928,0.0075208)
(9.375039,0.0074902)(9.426265,0.0074616)(9.475733,0.007435)
(9.523562,0.0074102)(9.569856,0.0073868)(9.61471,0.0073648)
(9.658211,0.0073442)(9.70044,0.0073248)(9.741467,0.0073066)
(9.78136,0.0072892)(9.820179,0.0072728)(9.857981,0.0072572)
(9.894818,0.0072426)(9.930737,0.0072284)(9.965784,0.0072152)
(10,0.0072024)};
\end{tikzpicture}\end{center}
\caption{\emph{A plot of $\mathbb{E}(X_n)/(n+1) - \log_2n$ in
log-scale (the decreasing curve using the $y$-axis on the right-hand
side), and that of $\mathbb{V}(X_n)/n$ in log-scale (the increasing
curve using the $y$-axis on the left-hand side).}}
\end{figure}
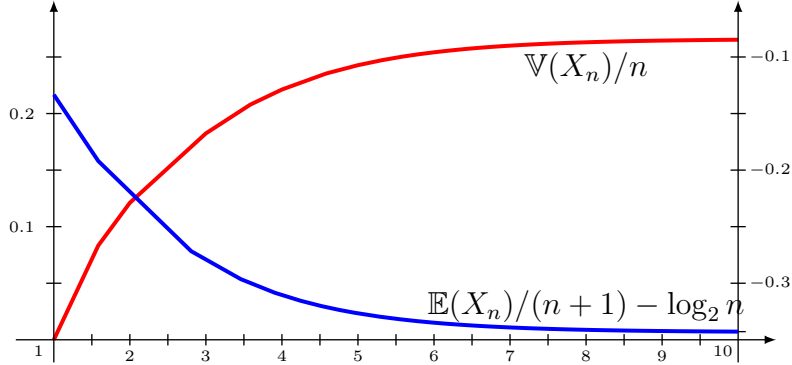

Let
\begin{align}\label{Q}
    Q_k
    := \prod_{1\le j\le k} \left(1-\frac1{2^j}\right),
    \quad \text{and}\quad
    Q(z)
    := \prod_{j\ge1} \left(1-\frac{z}{2^j}\right).
\end{align}
In particular, $Q(1)=Q_\infty$.
The variance was computed in \cite{kirschenhofer94a} by a direct
second-moment approach and the result is
\begin{align*}
    \mathbb{V}(X_n)
    = n(C_{\textit{kps}}+\varpi_{\textit{kps}}(\log_2n))
    + O(\log^2n),
\end{align*}
where $\varpi_{\textit{kps}}(t)$ is again a $1$-periodic, zero-mean
function and the mean value $C_{\textit{kps}}$ is given by ($L :=
\log 2$)
\[
\begin{split}
    C_{\textit{kps}}
    &=-\frac{28}{3L}-\frac{39}{4}
    +\frac{\pi^2}{2L^2}+\frac{2}{L^2}
    -\frac{2Q_\infty}{L}-2\sum_{\ell\ge 1}
    \frac{\ell2^\ell}{(2^\ell-1)^2}+\frac{2}{L}\sum_{\ell\ge
    1}\frac{1}{2^\ell-1}\\
    &-\frac{2}{L}\sum_{\ell\ge 3}\frac{(-1)^{\ell+1}(\ell-5)}
    {(\ell+1)\ell(\ell-1)(2^\ell-1)}\\
    &+\frac{2}{L}\sum_{\ell\ge1}(-1)^\ell2^{-\binom{\ell+1}{2}}
    \left(\frac{L(1-2^{-\ell+1})/2-1}{1-2^{-\ell}}
    -\sum_{r\ge 2}\frac{(-1)^{r+1}}{r(r-1)(2^{r+\ell}-1)}\right)\\
    &+\sum_{\ell\ge3}\sum_{2\le r<\ell}\binom{\ell+1}{r}
    \frac{Q_{r-2}Q_{\ell-r-1}}{2^{\ell}Q_\ell}\sum_{j\ge \ell+1}
    \frac{1}{2^{j}-1}-2\left[\varpi_1^{[1]}\varpi_2^{[2]}\right]_0-
    \left[(\varpi_1^{[1]})^2\right]_0\\
    &+2\sum_{\ell\ge 2}\frac{1}{2^\ell Q_\ell}
    \sum_{r\ge 0}\frac{(-1)^r2^{-\binom{r+1}{2}}}
    {Q_r}Q_{r+\ell-2}\times\\
    &\qquad\times\Biggl\{-\sum_{j\ge1}\frac{1}{2^{j+r+\ell+2}-1}
    \left(2^\ell-\ell-2+\sum_{2\le i<\ell}
    \binom{\ell+1}{i}\frac{1}{2^{r+i-1}-1}\right)\\
    &\quad\quad\quad+\frac{1}{(1-2^{-\ell-r})^2}
    +\frac{\ell+1}{(1-2^{1-\ell-r})^2}
    -\frac{1}{L(1-2^{1-\ell-r})}\\
    &\quad\quad\quad-\sum_{2\le j\le \ell+1}
    \binom{\ell+1}{j}\frac{1}{2^{r+j-1}-1}
    +\frac{1}{L}\sum_{1\le j\le \ell+1}
    \binom{\ell+1}{j}\frac{1}{2^{r+j}-1}\\
    &\quad\quad\quad+\frac{1}{L}
    \sum_{0\le j\le \ell+1}\binom{\ell+1}{j}
    \sum_{i\ge 1}\frac{(-1)^i}{(i+1)(2^{r+j+i}-1)}\Biggr\}.
\end{split}
\]
Here $[\varpi_1\varpi_2]_0$ denotes the mean value of the function
$\varpi_1(t)\varpi_0(t)$ over the unit interval. The long expression
obviously shows the complexity of the asymptotic problem.

We show that this long expression can be largely simplified. Before
stating our result, we mention that the asymptotic normality of
$X_n$ (in the sense of convergence in distribution) was first proved
in \cite{jacquet95a} by a complex-analytic approach; for other
approaches, see \cite{schachinger01a} (martingale difference),
\cite{hubalek02a} (method of moments), \cite{neininger04a}
(contraction method).

\paragraph{A new asymptotic approximation to $\mathbb{V}(X_n)$.}
Define
\begin{align}\label{G2w}
    G_2(\omega)
    =Q_\infty\sum_{j,h,\ell\ge 0}
    \frac{(-1)^j2^{-\binom{j+1}{2}+j(\omega-2)}}
    {Q_jQ_hQ_\ell2^{h+\ell}}
    \varphi(\omega;2^{-j-h}+2^{-j-\ell}) ,
\end{align}
where for $0<\Re(\omega)<3$ and $x>0$
\begin{align*}
    \varphi(\omega;x)
    := \int_{0}^{\infty}\frac{s^{\omega-1}}{(s+1)(s+x)^2}\dd{s},
\end{align*}
which, by the relation
\[
    \int_0^\infty \frac{s^{\omega-1}}{s+1}\dd s
    = \frac{\pi}{\sin(\pi \omega)}
    = \Gamma(\omega)\Gamma(1-\omega)\qquad(0<\Re(\omega)<1),
\]
can be represented as
\begin{align*}
    \varphi(\omega;x)
    &= \begin{cases}
        \displaystyle
        \frac{\pi\left(1+x^{\omega-2}((\omega-2)\xi
        +1-\omega\right)}{(x-1)^2\sin(\pi \omega)}
        ,&\text{if}\ x\ne1;\\
        \displaystyle
        \frac{\pi(\omega-1)(\omega-2)}{2\sin(\pi\omega)}
        ,&\text{if}\ x=1.
    \end{cases}
\end{align*}
The last expression provides indeed a meromorphic continuation of
$\varphi(\omega;x)$ into the whole complex $\omega$-plane whenever
$x>0$. In particular,
\[
    \varphi(2;x)
    :=\begin{cases}
        \displaystyle \frac{x-\log x-1}{(x-1)^2},
        &\text{if}\ x\ne1;\\
        \displaystyle \frac12,&\text{if}\ x=1.
    \end{cases}
\]

\begin{thm} \label{thm-KPS}
The variance of the total path-length of random DSTs of $n$ nodes
satisfies
\begin{align}\label{VXn-KPS}
    \mathbb{V}(X_n)
    = n(C_{\textit{kps}}+\varpi_{\textit{kps}}(\log_2n)) + O(1),
\end{align}
where
\[
    C_{\textit{kps}}
    = \frac{G_2(2)}{\log 2}
    =\frac{Q_\infty}{\log 2}\sum_{j,h,\ell\ge 0}
    \frac{(-1)^j2^{-\binom{j+1}{2}}}{Q_jQ_hQ_\ell 2^{h+\ell}}\,
    \varphi(2;2^{-j-h}+2^{-j-\ell}),
\]
and $\varpi_{\textit{kps}}$ has the Fourier series expansion
\[
    \varpi_{\textit{kps}}(t)
    = \frac1{\log2}\sum_{k\in\mathbb{Z}\setminus \{0\}}
    \frac{G_2(2+\chi_k)}{\Gamma(2+\chi_k)}\,e^{2k\pi i t},
\]
which is absolutely convergent.
\end{thm}
One can derive more precise asymptotic expansions for
$\mathbb{V}(X_n)$ by the same approach we use. We content ourselves
with (\ref{VXn-KPS}) for convenience of presentation.

Note that
\[
    \frac{G_2(2+\chi_k)}{\Gamma(2+\chi_k)}
    = \Gamma(-1-\chi_k)Q_\infty\sum_{j,h,\ell\ge 0}
    \frac{(-1)^j2^{-\binom{j+1}{2}}}{Q_jQ_hQ_\ell2^{h+\ell}}
    \lambda_k(2^{-j-h}+2^{-j-\ell}),
\]
where
\[
    \lambda_k(t)
    := \begin{cases}
        \displaystyle
        \frac{1-t^{\chi_k}(1+\chi_k(1-t))}{(1-t)^2}
        ,&\text{if}\ t\ne1;\\
        \displaystyle \frac{\chi_k(\chi_k-1)}{2}
        ,&\text{if}\ t=1.
    \end{cases}
\]
Thus the Fourier series is absolutely convergent by the order
estimate (see \cite{erdelyi1953a})
\begin{align}\label{Gamma-large-t}
    |\Gamma(c+it)|
    = O\left(|t|^{c-1/2} e^{-\pi|t|/2}\right) \qquad(|t|\to\infty).
\end{align}

Numerically, $C_{\textit{kps}} \approx 0.26600\,36454\,05936\dots$,
in accordance with that given in \cite{kirschenhofer94a}. Also
$|\varpi_{\textit{kps}}(t)|\le 1.9\times 10^{-5}$.

\paragraph{Sketch of our approach.} Following the discussions in
Introduction, we first prove that the Poisson-Charlier expansion for
the mean and that for the second moment are not only identities but
also asymptotic expansions. For that purpose, it proves very useful
to introduce the following notion, which we term \emph{JS-admissible
functions} (following the survey paper \cite{jacquet95a}). This is
reminiscent of the classical H-admissible (due to Hayman) or
HS-admissible (due to Harris-Schoenfeld) functions; see \cite[\S
VIII.5]{flajolet2009a}.

Once we prove the asymptotic nature of the Poisson-Charlier
expansions for the mean and the second moment, it remains, according
again to the discussions in Introduction, to derive more precise
asymptotics for the function $\tilde{V}$ (as defined in (\ref{Vz})),
for which we will use first the Laplace transforms, normalize the
Laplace transform properly, and then apply the Mellin transform.
Such an approach will turn out to be very effective and readily
applicable to more general cases such as bucket DSTs, which is
discussed in details in the next section. The approach parallels
closely in essence that introduced by Flajolet and Richmond in
\cite{flajolet92a}, which starts from the ordinary generating function,
followed by an Euler transform, a proper normalization and the Mellin
transform, and then conclude by singularity analysis; see also
\cite{dean06a}. The path we take, however, offers additional
operational advantages, as will be clear later. See
Figure~\ref{fig-egf-ogf} for a diagrammatic illustration of the two
analytic approaches.

\subsection{Analytic de-Poissonization and JS-admissibility}

The fundamental differential-functional equations for the analysis
of random DSTs is of the form
\[
    \tilde{f}(z)+\tilde{f}'(z)
    = 2\tilde{f}(z/2) + \tilde{g}(z),
\]
with suitably given initial value $f(0)$ and $\tilde{g}$. For such
functions, it turns out that the asymptotic nature of the
Poisson-Charlier expansions for the coefficients (or
\emph{de-Poissonization}) can be justified in a rather systematic
way by the introduction of the notion of JS-admissible functions.

\emph{Here and throughout this paper}, the generic symbol
$\ve\in(0,1)$ always represents an arbitrarily small constant whose
value is immaterial and may differ from one occurrence to another.

\begin{definition} An entire function $\tilde{f}$ is said to be
JS-admissible, denoted by $\tilde{f}\in\JS$, if the following two
conditions hold for $|z|\ge1$.
\begin{itemize}

\item[\textbf{(I)}] There exist $\alpha,\beta\in\mathbb{R}$ such that
uniformly for $|\arg(z)|\le\ve$,
\[
    \tilde{f}(z)
    = O\left(|z|^{\alpha} (\log_+|z|)^\beta\right),
\]
where $\log_+x := \log (1+x)$.
\item[\textbf{(O)}] Uniformly for $\ve\le|\arg(z)|\le\pi$,
\[
    f(z)
    := e^{z}\tilde{f}(z)
    = O\left(e^{(1-\ve)|z|}\right).
\]
\end{itemize}
\end{definition}
For convenience, we also write $\tilde{f}\in\JS_{\!\!\alpha,\beta}$
to indicate the growth order of $\tilde{f}$ inside the sector
$|\arg(z)|\le\ve$.

Note that if $\tilde{f}$ satisfies condition \textbf{(I)}, then, by
Cauchy's integral representation for derivatives (or by Ritt's
theorem; see \cite[Ch.~1, \S~4.3]{olver1974a}), we have,
\begin{align*}
    \tilde{f}^{(k)}(z)
    &= O\left(\oint_{|w-z|=\ve |z|}
    \frac{|w|^\alpha|
    (\log_+ |w|)^\beta}{|w-z|^{k+1}}|\dd w|\right)\\
    &= O\left(|z|^{\alpha-k}(\log_+|z|)^\beta\right).
\end{align*}

\begin{prop}\label{prop-PC-asymp}
Assume $\tilde{f}\in\JS_{\!\!\alpha,\beta}$. Let $f(z) := e^z
\tilde{f}(z)$. Then the Poisson-Charlier expansion (\ref{PC}) of
$f^{(n)}(0)$ is also an asymptotic expansion in the sense that
\begin{align*}
    a_n
    &:= f^{(n)}(0)
    = n![z^n]f(z)
    = n![z^n]e^z \tilde{f}(z) \\
    &= \sum_{0\le j<2k}
    \frac{\tilde{f}^{(j)}(n)}{j!}\,\tau_j(n) + O\left(n^{\alpha-k}
    \left(\log n\right)^\beta\right),
\end{align*}
for $k=1,2,\dots$.
\end{prop}
\begin{proof} (Sketch)
Starting from Cauchy's integral formula for the coefficients, the
lemma follows from a standard application of the saddle-point
method. Roughly, condition \textbf{(O)} guarantees that the integral
over the circle with radius $n$ and argument satisfying $\ve\le
|\arg(z)|\le\pi$ is negligible, while condition \textbf{(I)} implies
smooth estimates for all derivatives (and thus error terms).
\end{proof}

The polynomial growth of condition \textbf{(I)} is sufficient for
all our uses; see \cite{jacquet98a} for more general versions.

The real advantage of introducing admissibility is that it opens the
possibility of developing closure properties as we now discuss.

\begin{lmm} \label{lm-closure} Let $m$ be a nonnegative integer and
$\alpha\in(0,1)$.
\begin{itemize}

\item[(i)] $z^m, e^{-\alpha z}\in\JS$.

\item[(ii)] If $\tilde{f}\in\JS$, then $\tilde{f}(\alpha
z),z^m\tilde{f}\in\JS$.

\item[(iii)] If $\tilde{f}, \tilde{g} \in\JS$,
then $\tilde{f}+\tilde{g}\in\JS$.

\item[(iv)] If $\tilde{f}\in\JS$, then the product
$\tilde{P}\tilde{f}\in\JS$, where $\tilde{P}$ is a polynomial
of $z$.

\item[(v)] If $\tilde{f}, \tilde{g}\in\JS$, then $\tilde{h} \in\JS$,
where $\tilde{h}(z) := \tilde{f}(\alpha z)\tilde{g}((1-\alpha)z)$.

\item[(vi)] If $\tilde{f}\in\JS$, then $\tilde{f}'\in\JS$, and thus
$\tilde{f}^{(m)}\in\JS$.

\end{itemize}
\end{lmm}
\begin{proof} Straightforward and omitted.
\end{proof}

Specific to our need for the analysis of DSTs is the following
transfer principle.
\begin{prop} \label{prop-dst-tr}
Let $\tilde{f}(z)$ and $\tilde{g}(z)$ be entire functions satisfying
\begin{align}\label{dfe-DST}
    \tilde{f}(z)+\tilde{f}'(z)
    =2\tilde{f}(z/2)+\tilde{g}(z),
\end{align}
with $f(0)=0$. Then
\[
    \tilde{g}\in\JS
    \quad\text{if and only if}\quad\tilde{f}\in\JS.
\]
\end{prop}
\begin{proof}
Assume $\tilde{g}\in\JS$. We check first the condition \textbf{(O)}
for $\tilde{f}$. Let $f(z):=e^z\tilde{f}(z)$ and
$g(z):=e^{z}\tilde{g}(z)$. By (\ref{dfe-DST}),
\[
    f'(z)
    =2e^{z/2}f(z/2)+g(z).
\]
Consequently, since $f(0)=0$,
\begin{align}\label{dst-ir-f}
    f(z)
    =\int_{0}^{z}\left(2e^{t/2}f(t/2)+g(t)\right)\dd t
    =z\int_{0}^{1}\left(2e^{tz/2}f(tz/2)+g(tz)\right)\dd t.
\end{align}
Now define
\[
    B(r)
    :=\max_{z\in\mathcal{C}_{r,\ve}}| f(z)|,
\]
where
\[
    \mathcal{C}_{r,\ve}
    :=\{z\ :\ |z|\le r,
    \ve\le|\arg(z)|\le\pi\},\qquad(r\ge 0;0<\ve<\pi/2).
\]
Then, by (\ref{dst-ir-f}), we have
\begin{align*}
    B(r)
    &\le r\int_0^1\left(2e^{tr\cos(\ve)/2}B(tr/2)
    +|g(tr)|\right)\dd t \\
    &=  \int_0^r\left(2e^{t\cos(\ve)/2}B(t/2)
    +O\left(e^{(1-\ve)t}\right)\right)\dd t\\
    &\le Ce^{r\cos(\ve)/2}B(r/2)+O\left(e^{(1-\ve)r}\right),
\end{align*}
where $C=4/\cos\ve>1$. This suggests that we define a majorant
function $K(r)$ of $B(r)$ by $K(r)=O(1)$ for $r\le1$ and for $r\ge1$
\[
    K(r)
    =Ce^{r\cos(\ve)/2}K(r/2)+h(r),
\]
where $h$ is an entire function satisfying $h(r)=O(1)$ for $r\le1$
and $h(r)=O\left(e^{(1-\ve)r}\right)$ for $r\ge1$. Let
$\tilde{K}(r):=e^{-r\cos(\ve)}K(r)$ and $\tilde{h}(r) :=
e^{-r\cos(\ve)}h(r)$. Then since $\cos\ve-1+\ve>0$ for
$\ve\in(0,1)$, we obtain
\[
    \tilde{K}(r)
    =C\tilde{K}(r/2)+\tilde{h}(r),\quad \tilde{h}(r)=O(1).
\]
Thus if we choose $m=\lceil \log_2r\rceil$ such that $2^m\ge r$
and iterate $m$ times the functional equation, then we obtain the
estimate
\begin{align*}
    \tilde{K}(r)
    &=\sum_{0\le k\le m} C^k \tilde{h}(r/2^k)+C^{m+1}
    \tilde{K}(r/2^{m+1})\\
    &= O\left(\sum_{r/2^k>1}C^k+C^m\right)\\
    &= O\left(r^{\log_2C}\right).
\end{align*}
Thus
\[
    B(r)
    = O\left(r^{\log_2C} e^{r\cos\ve}\right).
\]
which establishes condition \textbf{(O)}.

Our proof for $\tilde{f}$ satisfying \textbf{(I)} proceeds in a
similar manner and starts again from (\ref{dst-ir-f}) but of the
form
\[
    \tilde{f}(z)
    =z\int_{0}^{1}e^{-(1-t)z}
    \left(2\tilde{f}(tz/2)+\tilde{g}(tz)\right)\dd t.
\]
Now, define
\[
    \tilde{B}(r)
    :=\max_{z\in\mathcal{S}_{r,\ve}}|\tilde{f}(z)|,
\]
where
\[
    \mathcal{S}_{r,\ve}
    :=\{z\ :\ |z|\le r,|\arg(z)|\le\ve\},
    \qquad (r\ge 0;0<\ve<\pi/2).
\]
Then
\begin{align*}
    \tilde{B}(r)
    &\le r\int_0^1 e^{-(1-t)r\cos\ve}\left(2
    \tilde{B}(tr/2)+|\tilde{g}(tr)|\right)\dd t \\
    &= \int_1^r\left(2e^{-(r-t)\cos\ve}\tilde{B}(t/2)
    +O\left(e^{-(r-t)\cos\ve}
    t^\alpha(\log_+t)^\beta\right)\right)\dd t+O(1)\\
    &\le C\tilde{B}(r/2) + O\left(r^{\alpha}
    (\log_+r)^\beta+1\right),
\end{align*}
where $C=2/\cos\ve>2$. The same majorization argument used above for
\textbf{(O)} then leads to
\begin{align*}
    \tilde{B}(r)
    =\begin{cases}
        O(r^{\log_2C}),&\text{if } \alpha<\log_2C;\\
        O(r^{\log_2C}(\log_+r)^{\beta+1}),
        &\text{if }\alpha=\log_2C;\\
        O\left(r^{\alpha} (\log_+r)^\beta\right),&
        \text{if }\alpha>\log_2C.
    \end{cases}
\end{align*}
This proves \textbf{(I)} for $\tilde{f}$.

The necessity part follows trivially from Lemma~\ref{lm-closure}.
\end{proof}

The estimates we derived of asymptotic-transfer type are indeed
over-pessimistic when $1\le\alpha\le\log_2C$, but they are
sufficient for our use. The true orders are those with $\ve\to0$,
which can be proved by the Laplace-Mellin-de-Poissonization approach
we use later.

Lemma~\ref{lm-closure} and Proposition~\ref{prop-dst-tr} provide
very effective tools for justifying the de-Poissonization of
functions satisfying the equation (\ref{dfe-DST}), which is often
carried out through the use of the increasing-domain argument (see
\cite{jacquet98a}). The latter argument is also inductive in nature
and similar to the one we are developing here, although it is less
``mechanical" and less systematic.

\subsection{Generating functions and integral transforms}

Since our approach is purely analytic and relies heavily on
generating functions, we first derive in this subsection the
differential-functional equations we will be working on later. Then
we apply the de-Poissonization tools we developed to the Poisson
generating functions of the mean and the second moment and justify
the asymptotic nature of the corresponding Poisson-Charlier
expansions. Then we sketch the asymptotic tools we will follow based
on the Laplace and Mellin transforms.

\paragraph{Generating functions.}
In terms of the moment generating function $M_n(y):=
\mathbb{E}(e^{X_n y})$, the recurrence (\ref{rec-xn-dst}) translates
into
\begin{align}\label{rr-TPL}
    M_{n+1}(y)
    =e^{ny}2^{-n}\sum_{0\le j\le n}
    \binom{n}{j}M_j(y)M_{n-j}(y),\qquad (n\ge 0),
\end{align}
with $M_0(y)=1$.

Now consider the bivariate exponential generating function
\[
    F(z,y)
    :=\sum_{n\ge 0}\frac{M_{n}(y)}{n!}z^n.
\]
Then by (\ref{rr-TPL}),
\[
    \frac{\partial}{\partial z}F(z,y)
    =F\left(\frac{e^yz}{2},y\right)^2,
\]
and the Poisson generating function $\tilde{F}(z,y):=e^{-z}F(z,y)$
satisfies the differential-functional equation
\begin{equation}\label{equ-tildeP}
    \tilde{F}(z,y)+\frac{\partial}{\partial z}\tilde{F}(z,y)
    =e^{(e^y-1)z}\tilde{F}\left(\frac{e^yz}{2},y\right)^2,
\end{equation}
with $\tilde{F}(0,y)=1$. No exact solution of such a nonlinear
differential equation is available; see \cite{jacquet95a} for an
asymptotic approximation to $\tilde{F}$ for $y$ near unity.

\paragraph{Mean and second moment.}
Let now
\[
    \tilde{F}(z,y)
    :=\sum_{m\ge 0}\frac{\tilde{f}_m(z)}{m!}y^m,
\]
where $\tilde{f}_m(z)$ denotes the Poisson generating function of
$\mathbb{E}(X_n^{m})$. Then we deduce from (\ref{equ-tildeP}) that
\begin{align}
    \tilde{f}_1(z)+\tilde{f}'_1(z)
    &=2\tilde{f}_1(z/2)+z, \label{poi-mean}\\
    \tilde{f}_2(z)+\tilde{f}'_2(z)
    &=2\tilde{f}_2(z/2)+2\tilde{f}_1(z/2)^2
    +4z\tilde{f}_1(z/2)+2z\tilde{f}'_1(z/2)+z+z^2,
    \label{poi-2nd-mom}
\end{align}
with the initial conditions $\tilde{f}_1(0)=\tilde{f}_2(0)=0$.

\begin{prop} \label{prop-PC-mean-var}
The Poisson-Charlier expansion for the mean and that for the second
moment are both asymptotic expansions
\begin{align*}
    \mathbb{E}(X_n)
    &= \sum_{0\le j<2k}\frac{\tilde{f}_1^{(j)}(n)}{j!}\,\tau_j(n)
    + O\left(n^{-k+1}\right),\\
    \mathbb{E}(X_n^2)
    &= \sum_{0\le j<2k}\frac{\tilde{f}_2^{(j)}(n)}{j!}\,\tau_j(n)
    + O\left(n^{-k+2}(\log n)^2\right),
\end{align*}
for $k=1,2,\dots$.
\end{prop}
\begin{proof} (Sketch)
By Lemma~\ref{lm-closure} and Proposition~\ref{prop-dst-tr}, we see
that both $\tilde{f}_1, \tilde{f}_2\in\JS$, and thus we can apply
Proposition~\ref{prop-PC-asymp}. Indeed the proof of
Proposition~\ref{prop-dst-tr} provides already crude bounds for the
growth order of $\tilde{f}_1, \tilde{f}_2$. The more precise
estimates $\tilde{f}_1(z)\asymp |z||\log z|$ and
$\tilde{f}_2(z)\asymp |z|^2|\log z|^2$ for $z$ inside the sector
$\{z\,:\,|\arg(z)|\le\ve\}$ will be provided later in the next two
subsections.
\end{proof}

\paragraph{An asymptotic approach based on Laplace and Mellin
transforms.} Once the de-Poissonization steps are justified, all
that remains for the proof of Theorem~\ref{thm-KPS} is to derive
more precise asymptotic approximations to $\tilde{f}_1$ and
$\tilde{V}$ (as defined in (\ref{Vz})). The approach we use begins
with a more precise characterization of $\tilde{f}_1(z)$. Both
$\tilde{f}_1$ and $\tilde{V}$  satisfy a differential-functional
equation of the form
\[
    \tilde{f}(z) +\tilde{f}'(z)
    = 2\tilde{f}(z/2) + \tilde{g}(z),
\]
with the initial condition $\tilde{f}(0)=0$. To derive the
asymptotics of $\tilde{f}$ for large complex $z$, we proceed along
the following principal steps; see also \cite{dean06a}.

\begin{description}
\item [\emph{Laplace transform}:]
The Laplace transform of $\tilde{f}$ satisfies
\begin{align}\label{laplace}
    (s+1) \mathscr{L}[\tilde{f};s]
    = 4\mathscr{L}[\tilde{f};2s] + \mathscr{L}
    [\tilde{g};s],
\end{align}
which exists and defines an analytic function if $\tilde{g}$ grows
at most polynomially for large $|z|$.

\item [\emph{Normalizing factor}:] Dividing both sides of
(\ref{laplace}) by $Q(-2s)=\prod_{j\ge0}(1+s/2^j)$ gives a
functional equation of the form
\[
    \bar{\mathscr{L}}[\tilde{f};s]
    = 4\bar{\mathscr{L}}[\tilde{f};2s] +
    \frac{\mathscr{L}[\tilde{g};s]}{Q(-2s)},
\]
where $\bar{\mathscr{L}}[\tilde{f};s] :=
\mathscr{L}[\tilde{f};s]/Q(-s)$.

\item [\emph{Mellin transform}:] The Mellin transform of
$\bar{\mathscr{L}}$ then satisfies
\[
   \mathscr{M}[\bar{\mathscr{L}}[\tilde{f}_1;s];\omega]
   = \frac1{1-2^{2-\omega}}{\mathscr{M}
    \left[\frac{\mathscr{L}[\tilde{g};s]}{Q(-2s)};\omega\right]}.
\]

\item[\emph{Inverting the process.}] We first derive the local
behavior of $\tilde{\mathscr{L}}[\tilde{f};s]$ for small $s$ by
the Mellin inversion (often by calculus of residues after justification
of analytic properties), and then the asymptotic behavior of
$\tilde{f}(z)$ for large $z$ is derived by the Laplace inversion,
similar to singularity analysis.

\end{description}

\subsection{Expected internal path-length of random DSTs}
\label{mean-dst}

We consider in details in this subsection the expected value $\mu_n
:= \mathbb{E}(X_n)$ of the total internal path-length,
paving the way for the asymptotic analysis of the variance.
Starting from either the equation (\ref{poi-mean}) or the recurrence
\[
    \mu_{n+1}
    = 2^{1-n}\sum_{0\le j\le n}\binom{n}{j}\mu_j + n \qquad(n\ge0)
\]
with $\mu_0:=0$, there are several approaches to the asymptotics of
$\mu_n$. We will briefly describe the one using integral
representation of finite differences (or Rice's integrals) and then
present the Laplace and Mellin transforms we will use, which, as
will become clear, is essentially the Flajolet-Richmond approach
(see \cite{flajolet92a}).

\paragraph{Rice's integral representation.}

By (\ref{poi-mean}), we have, with $\tilde{\mu}_n :=
n![z^n]\tilde{f}_1(z)$,
\[
    \tilde{\mu}_{n+1}
    = -\left(1-2^{1-n}\right) \tilde{\mu}_n \qquad(n\ge0),
\]
with $\tilde{\mu}_0=0$, which by iteration yields
\begin{align}
    \tilde{\mu}_n
    = (-1)^n Q_{n-2},\quad
    Q_n
    := \prod_{1\le j\le n} \left(1-2^{-j}
    \right). \label{t-mun}
\end{align}
Thus by Rice's formula (\cite{flajolet95b})
\begin{align*}
    \mu_n
    &:= \mathbb{E}(X_n)
    = \sum_{2\le j\le n}\binom{n}{j} \tilde{\mu}_j \\
    &= \frac1{2\pi i} \int_{(\frac32)} \frac{\Gamma(n+1)
    \Gamma(-s)}{\Gamma(n+1-s)}\cdot
    \frac{Q(1)}{(1-2^{1-s})Q(2^{1-s})}\dd s,
\end{align*}
where the integration path $\int_{(c)}$ is along the vertical line
with real part equal to $c$ and $Q$ is defined in (\ref{Q}). We then
obtain (\ref{mean-asymp}) by standard arguments; see
\cite{flajolet86d} or \cite{mahmoud1992a} for details.

This approach readily gives the approximation (\ref{mean-asymp}) for
the mean and can be refined to obtain a full asymptotic expansion.
However, its extension to the variance becomes extremely messy, as
shown in \cite{kirschenhofer94a}.

\paragraph{Laplace transform.}
We first show that the asymptotics of $\tilde{f}_1(z)$ can be
derived through a direct use of the Laplace and Mellin transforms,
which relies on several ad hoc steps that are not easily extended.
A more general procedure will be developed below.

By (\ref{poi-mean}), we see that the Laplace transform of $f_1(z)$
satisfies the functional equation
\begin{align}\label{rel-laplace}
    (s+1)\mathscr{L}[\tilde{f}_1;s]
    = 4\mathscr{L}[\tilde{f}_1;2s] + s^{-2},
\end{align}
which exists and is analytic in $\mathbb{C}\setminus(-\infty,0]$.

By dividing both sides by $s+1$ and by iteration, we get
\begin{align}
    \mathscr{L}[\tilde{f}_1;s]
    = \frac1{s^2} \sum_{j\ge0}
    \frac{1}{(s+1)\cdots(2^js+1)}. \label{Lf1s}
\end{align}
On the other hand, from (\ref{t-mun}), we have
\begin{align*}
    \mathscr{L}[\tilde{f}_1;s]
    &= \int_0^\infty e^{-sz} \sum_{n\ge0}
    \frac{\tilde{\mu}_n}{n!}\, z^n dz \\
    &= \sum_{n\ge0} (-1)^n Q_n s^{-n-3}.
\end{align*}
This implies the identity
\[
    \sum_{n\ge0} \frac{(-1)^n Q_n}{s^{n+1}}
    = \sum_{j\ge0}\frac{1}{(s+1)\cdots(2^js+1)}.
\]
However, neither form is useful for our asymptotic purpose.

Now by partial fraction expansion, we obtain
\begin{align*}
    \frac{1}{(s+1)\cdots(2^js+1)}
    = \sum_{0\le \ell \le j}\frac{(-1)^{j-\ell}
    2^{-\binom{j-\ell+1}{2}-\ell}}{(s+2^{-\ell})
    Q_\ell Q_{j-\ell}}.
\end{align*}
Thus
\begin{align*}
    \mathscr{L}[\tilde{f}_1;s]
    &= \frac1{s^2}\sum_{j\ge0}\sum_{0\le \ell \le j}
    \frac{(-1)^{j-\ell}
    2^{-\binom{j-\ell+1}{2}-\ell}}{(s+2^{-\ell})
    Q_\ell Q_{j-\ell}} \\
    &= \frac1{s^2}\sum_{\ell\ge0} \frac{1}{Q_\ell(2^\ell s+1)}
    \sum_{j\ge0}\frac{(-1)^j 2^{-\binom{j+1}{2}}}{Q_j}.
\end{align*}
Note that
\[
    \sum_{j\ge0} \frac{2^js}{(s+1)\cdots (2^js+1)}
    = 1.
\]
By the Euler identity
\[
    \sum_{j\ge0}\frac{q^{\binom{j}{2}} z^j}{(1-q)\cdots (1-q^j)}
    = \prod_{k\ge0}\left(1+q^k z\right),
\]
we see that
\[
    \sum_{j\ge0} \frac{(-1)^j 2^{-\binom{j+1}{2}}}{Q_j}
    = Q(1)
    = Q_\infty\approx 0.28878809\dots
\]
This gives
\[
    \mathscr{L}[\tilde{f}_1;s]
    = \frac{Q_\infty}{s^2}\sum_{\ell\ge0}
    \frac{1}{Q_\ell(2^\ell s+1)},
\]
and then
\begin{align}
    \tilde{f}_1(z)
    = Q_\infty \sum_{\ell\ge0}
    \frac{2^\ell}{Q_\ell}\left(e^{-z/2^\ell}
    -1+\frac{z}{2^\ell}\right). \label{t-f1z}
\end{align}
Consequently,
\[
    \mu_n
    = Q_\infty \sum_{\ell\ge0}\frac{2^\ell}{Q_\ell}
    \left((1-2^{-\ell})^n-1+2^{-\ell}n\right).
\]

Asymptotically, we have, by (\ref{t-f1z}) and the identity
\begin{align}\label{inv-Qz}
    \frac1{Q(z)}
    = \prod_{j\ge1}\frac1{1-z/2^j}
    = \sum_{\ell\ge0} \frac{z^\ell}{Q_\ell 2^\ell}
    \qquad(|z|<2),
\end{align}
the Mellin integral representation
\[
    \tilde{f}_1(z)
    = \frac1{2\pi i} \int_{(-3/2)}
    \frac{Q(1) \Gamma(s) z^{-s}}{(1-2^{s+1})Q(2^{s+1})} \dd s,
\]
from which we derive the asymptotic approximation
\begin{align}\label{tf1z-asymp}
    \tilde{f}_1(z)
    = (z+1)\log_2 z + z\left(\frac{\gamma-1}{\log 2}
    + \frac12-c_1+\varpi_1(\log_2z)\right)+O(1),
\end{align}
uniformly for $|z|\to\infty$ and $|\arg(z)|\le \pi/2-\ve$,
where $\varpi_1$ is given in (\ref{F1u}). (As usual, we use the
asymptotic estimate (\ref{Gamma-large-t}) for the Gamma function.)

\paragraph{Laplace and Mellin transforms.}
We now re-do the analysis for $\tilde{f}_1(z)$ in a more general way
that can be easily extended to other cases.

We again start from (\ref{rel-laplace}) and consider
\[
    \bar{\mathscr{L}}[\tilde{f}_1;s]
    :=\frac{\mathscr{L}[\tilde{f}_1;s]}{Q(-s)},
\]
where $Q(z)$ is defined in (\ref{Q}). Dividing both sides of
(\ref{rel-laplace}) by $Q(-2s)$ yields
\begin{align}\label{L-bar}
    \bar{\mathscr{L}}[\tilde{f}_1;s]
    = 4\bar{\mathscr{L}}[\tilde{f}_1;2s]+\frac1{Q(-2s)s^2}.
\end{align}

We now apply the Mellin transform. Note that we have, by the fact
that $X_0=X_1=0$ and the proof of Proposition~\ref{prop-dst-tr},
\begin{equation*}
    \tilde{f}_1(z)
    =\begin{cases}
        O(z^2),&\text{if}\ z\rightarrow 0^+;\\
        O(z^{1+\ve}),&\text{if}\ z\rightarrow\infty.
    \end{cases}
\end{equation*}
Then
\[
    \mathscr{L}[\tilde{f}_1;s]
    =\begin{cases}
        O(s^{-2-\ve}),&\text{as}\ s\rightarrow 0^+;\\
        O(s^{-3}),&\text{as}\ s\rightarrow\infty.
    \end{cases}
\]
On the other hand, by the Mellin transform,
\begin{align}
    \log Q(-2s)
    &=\sum_{j\ge0}\log\left(1+\frac{s}{2^j}\right)\nonumber\\
    &=\frac{1}{2\pi i}\int_{(-\frac{1}{2})}
    \frac{\pi s^{-w}}{(1-2^w) w\sin \pi w}\dd w\nonumber\\
    &=\frac{(\log s)^2}{2\log 2} +\frac{\log s}{2}+
    \sum_{k\in\mathbb{Z}}q_ks^{-\chi_k}+O(|s|^{-1})\label{est-Q}
\end{align}
uniformly for $|s|\to\infty$ and $|\arg(s)|\le\pi-\ve$, where
$\chi_k:=2k\pi i/\log 2$,
\[
    q_0
    =\frac{\log 2}{12}+\frac{\pi^2}{6\log 2}
\]
and
\[
    q_k
    =\frac{1}{2k\sinh(2k\pi/\log 2)}\qquad(k\not=0).
\]
This asymptotic expansion, together with the Taylor expansion
\begin{align*}
    Q(-2s)
    =1+O(|s|),\qquad (|s|\to 0),
\end{align*}
gives rise to
\[
    \bar{\mathscr{L}}[\tilde{f}_1;s]
    =\begin{cases}
        O(s^{-2-\ve}),&\text{as}\ s\rightarrow 0^+;\\
        O(s^{-M}),&\text{as}\ s\rightarrow\infty,
    \end{cases}
\]
where $M>0$ is an arbitrary real number. Consequently, the Mellin
transform of $\bar{\mathscr{L}}[\tilde{f}_1;s]$, denoted by
$\mathscr{M}[\bar{\mathscr{L}};\omega]$, exists in the half-plane
$\Re(\omega)\ge 2+\ve$. Then by applying the Mellin transform to
(\ref{L-bar}), we obtain
\begin{equation*}
    \mathscr{M}[\bar{\mathscr{L}};\omega]
    =\frac{G_1(\omega)}{1-2^{2-\omega}},\qquad (\Re(\omega)>2),
\end{equation*}
where
\begin{align}\label{G1w}
    G_1(\omega)
    := \int_{0}^{\infty}
    \frac{s^{\omega-3}}{Q(-2s)}\dd{s}
    = \frac{\pi Q(2^{\omega-2})}{Q(1)\sin\pi\omega}
    = \frac{Q(2^{\omega-2})}{Q(1)}\,\Gamma(\omega)
    \Gamma(1-\omega),
\end{align}
for $\Re(\omega)>2$; see \cite{flajolet92a}.

\paragraph{Inverse Mellin and inverse Laplace transforms.} We
can now apply successively the inverse Mellin and then Laplace
transforms to derive the asymptotics of $\tilde{f}_1(z)$. Observe
that $G_1(\omega)$ has a simple pole at $\omega=2$. By (\ref{G1w})
or Proposition 5 in \cite{flajolet95a}, we obtain
\[
    | G_1(c+it)|
    =O\left(e^{-(\pi-\ve)| t|}\right),
\]
for large $| t|$ and $c\in\mathbb{R}$.
Then by the calculus of residues,
\[
    \bar{\mathscr{L}}[\tilde{f}_1;s]
    =\frac{1}{s^2}\log_2\frac{1}{s} + \frac{1}{s^2}\left(
    \frac12-c_1+\frac{1}{\log2}
    \sum_{k\in\mathbb{Z} \setminus\{0\}}G_1(2+\chi_k)
    s^{-\chi_k}\right)+O(|s|^{-1}),
\]
uniformly for $|s|\to 0$ and $|\arg(s)|\le\pi-\ve$. Using the
expansion
\[
    Q(-s)
    = 1+s +(|s|^2)\qquad(|s|\sim0),
\]
we see that
\[
    \mathscr{L}[\tilde{f}_1;s]
    =\frac{1+s}{s^2}\log_2\frac{1}{s} + \frac{1}{s^2}\left(
    \frac12-c_1+\frac{1}{\log2}
    \sum_{k\in\mathbb{Z} \setminus\{0\}}G_1(2+\chi_k)
    s^{-\chi_k}\right)+O(|s|^{-1}),
\]
uniformly for $|s|\to 0$ and $|\arg(s)|\le\pi-\ve$.

Finally, we consider the inverse Laplace transform. The following
simple result is very useful for our purposes.
\begin{prop} \label{prop-inv-laplace}
Let $\tilde{f}(z)$ be a function whose Laplace transform exists and
is analytic in $\mathbb{C}\setminus(-\infty,0]$. Assume that
\begin{equation}\label{est-laplace}
    \mathscr{L}[\tilde{f};s]
    = \begin{cases}
        O\left(|s|^{-\alpha}|\log|s+1||^{m}\right),\\
        cs^{-\omega}(-\log s)^m,\\
        o(|s|^{-\alpha}|\log|s+1||^{m}),
    \end{cases}
\end{equation}
uniformly for $|s|\to 0$ and $| \arg(s)|\le\pi-\ve$, where
$\alpha\in\mathbb{R}$, $\omega\in\mathbb{C}$ and $m=0,1,\dots$. If
$\mathscr{L}[\tilde{f};s]$ satisfies
\begin{align} \label{Lap-small}
    |\mathscr{L}[\tilde{f};s]|
    = O\left(|s|^{-1-\ve}\right),
\end{align}
as $|s|\to\infty$ in $| \arg(s)|\le\pi-\ve$, then
\[
    \tilde{f}(z)
    = \begin{cases}
        O\left(|z|^{\alpha-1}(\log |z|)^m\right),\\
        \displaystyle cz^{\omega-1}
        \sum_{0\le j\le m}\binom{m}{j}
        (\log z)^{m-j}\frac{\partial^j}
        {\partial \omega^j}\frac1{\Gamma(\omega)},\\
        o\left(|z|^{\alpha-1}(\log |z|)^m\right),
    \end{cases}
\]
respectively, where the $O$- and $o$-terms hold uniformly for
$|z|\to\infty$ and $|\arg(z)|\le\pi/2-\ve$.
\end{prop}
\begin{proof}
Let $\tilde{\mathscr{L}}(s)=\mathscr{L}[\tilde{f};s]$. Then by
the inverse Laplace transform,
\[
    \tilde{f}(z)
    = \frac{1}{2\pi i}\int_{(1)}e^{zs}\tilde{\mathscr{L}}(s)\dd{s}
    = \frac{1}{2\pi i}\int_{\mathcal{H}}e^{zs}
    \tilde{\mathscr{L}}(s)\dd{s},
\]
where $\mathcal{H}$ is the Hankel contour consisting of the two
rays $te^{\pm i\ve}\pm i/|z|,-\infty<t\le 0$ and the semicircle
$\exp(i\varphi)/|z|,-\pi/2\le\varphi\le\pi/2$; see
Figure~\ref{fg-calH}.

\begin{figure}[!h]
\begin{center}
\begin{tikzpicture}
\draw[thick] (0,0) -- (-5,0.5);
\draw[thick] (0,0) -- (-5,-0.5);
\draw[thick] (0,0.2) -- (-5,0.7);
\draw[thick] (0,-0.2) -- (-5,-0.7);
\draw (0.23,-0.15) -- (0.23,-0.25);
\draw[-latex] (-5,0) -- (1.3,0) node[right] {\small $\Re(s)$};
\draw[-latex] (0,-1.3) -- (0,1.3) node[above] {\small $\Im(s)$};
\draw[-latex] (-0.2,-0.2) -- (0,-0.2);
\draw[-latex] (0.43,-0.2) -- (0.23,-0.2);
\draw[thick] (0mm,-2mm) arc (-90:90:2mm);
\draw (-19mm,0mm) arc (180:215:3mm);
\draw (-2.3,-0.1) node {\small $\ve$} ;
\draw (-3,0.8) node {\small $\mathcal{H}$} ;
\draw (0.2,-0.6) node {\small $\frac{1}{|z|}$} ;
\end{tikzpicture}
\end{center}
\caption{\emph{The contour $\mathcal{H}$.}}
\label{fg-calH}
\end{figure}
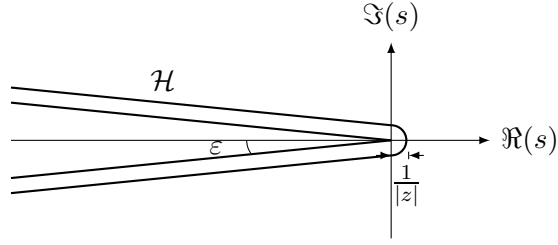

Assume from now on $|z|$ is sufficiently large and lies in the
sector with $|\arg(z)|\le\pi/2 -\ve$. We prove only the $O$-case,
the other two cases being similar. For simplicity, we consider
only the case $m=0$, the other cases being easily extended.

We split the above integral along $\mathcal{H}$ into two parts
\[
    \frac{1}{2\pi i}\int_{\mathcal{H}}e^{zs}
    \tilde{\mathscr{L}}(s)\dd s
    = \frac{1}{2\pi i}\int_{{\mathcal H}_{>}}
    e^{zs} \tilde{\mathscr{L}}(s)\dd{s}
    +\frac{1}{2\pi i}\int_{\mathcal{H}_\supset}e^{zs}
    \tilde{\mathscr{L}}(s)\dd{s},
\]
where $\mathcal{H}_>$ comprises the two rays $te^{i\ve}\pm i/|
z|,-\infty<t\le -T$ with $T>1$ a fixed constant
and $\mathcal{H}_\supset$ represents the remaining contour.

The integral along $\mathcal{H}_>$ is easily estimated
\begin{align*}
    \frac{1}{2\pi i}\int_{{\mathcal H}_>}
    e^{zs}\tilde{\mathscr{L}}(s)\dd{s}
    &= O\left(\int_{-\infty}^{-T}e^{\Re(|z|e^{i\arg(z)}
    (te^{i\ve}+i/|z|))}|t|^{-1-\ve}\dd t\right) \\
    &= O\left(\int_T^{\infty} t^{-1-\ve}
    e^{-|z|t\cos(\arg(z)+\ve)}\dd t\right)\\
    &= O\left(|z|^{\ve}e^{-c|z| T}\right),
\end{align*}
the $O$-term holding uniformly for $|z|\to\infty$ provided that
$|\arg(z)|+\ve<\pi/2$, where $c>0$ is a suitable constant.

For the second integral, we use (\ref{est-laplace}). Then the
integral along the semicircle is bounded as follows.
\[
    \frac{1}{2\pi|z|}\int_{-\pi/2}^{\pi/2}
    e^{ze^{i\theta}/|z|+i\theta}
    \tilde{\mathscr{L}}(e^{i\theta}/|z|)\dd\theta
    = O\left(|z|^{\alpha-1}\right),
\]
uniformly for $|z|\to\infty$. For the remaining part $t\pm i/|
z|,-T<t\le 0$, we have
\begin{align*}
    \frac{1}{2\pi i}\int_{-T}^{0} e^{z(t\pm i/|z|)}
    \tilde{\mathscr{L}}(t\pm i/|z|)\dd t
    &= O\left(|z|^{\alpha}\int_{-T}^{0}
    \frac{e^{c|z| t}}{(|z|^2t^2+1)^{\alpha/2}}
    \dd t\right)\\
    &= O\left(|z|^{\alpha-1}\int_{0}^{\infty}
    \frac{e^{-cu}}{(u^2+1)^{\alpha/2}}\dd u\right)\\
    &= O\left(|z|^{\alpha-1}\right),
\end{align*}
uniformly for $|z|\to\infty$, where $c>0$ is a suitable constant.
This completes the proof.
\end{proof}

Note that the inverse Laplace transform of $s^{-2}\log(1/s)$ is
$z\log z-(1-\gamma)z$. This, together with a combined use of
Proposition~\ref{prop-inv-laplace}, leads to (\ref{tf1z-asymp}).

The justification of the estimate (\ref{Lap-small}) is easily
performed by using the relation (\ref{LT-OGF}) below.

\paragraph{The Flajolet-Richmond approach \cite{flajolet92a}.}
Instead of the Poisson generating function, this approach starts
from the ordinary generating function $A(z) := \sum_n \mu_n z^n$.

\begin{itemize}

\item[--] Then the Euler transform\footnote{For a better comparison
with the approach we use, our $\hat{A}$ differs from the usual Euler
transform by a factor of $s$.}
\[
    \hat{A}(s)
    := \frac 1{s+1} A\left(\frac 1{s+1}\right)
\]
satisfies
\[
    (s+1) \hat{A}(s)
    = 4\hat{A}(2s)+s^{-2},
\]
identical to (\ref{rel-laplace}).

\item[--] The normalized function $\bar{A}(s) := \hat{A}(s)/Q(-s)$
satisfies
\[
    \bar{A}(s)
    = 4\bar{A}(2s)+\frac1{s^2Q(-2s)},
\]
again identical to (\ref{L-bar}).

\item[--] The Mellin transform of $\bar{A}$ satisfies
$(\Re(\omega)>2)$
\[
    \mathscr{M}[\bar{A};\omega]
    = \frac{G_1(\omega)}{1-2^{2-\omega}},
\]
where $G_1(\omega)$ is as defined in (\ref{G1w}).
\end{itemize}

Then invert the process by considering first the Mellin inversion,
deriving asymptotics of
\[
    \bar{A}(s) = \frac1{2\pi i} \int_{(5/2)} s^{-\omega}
    \frac{G_1(\omega)}{1-2^{2-\omega}} \dd \omega,
\]
as $s\to0$ in $\mathbb{C}$. Then deduce asymptotics of
\[
    A(z)= \frac1{z}\hat{A}\left(\frac1z-1\right),
\]
as $z\to1$. Finally, apply singularity analysis (see
\cite{flajolet90a}) to conclude the asymptotics of $\mu_n$.

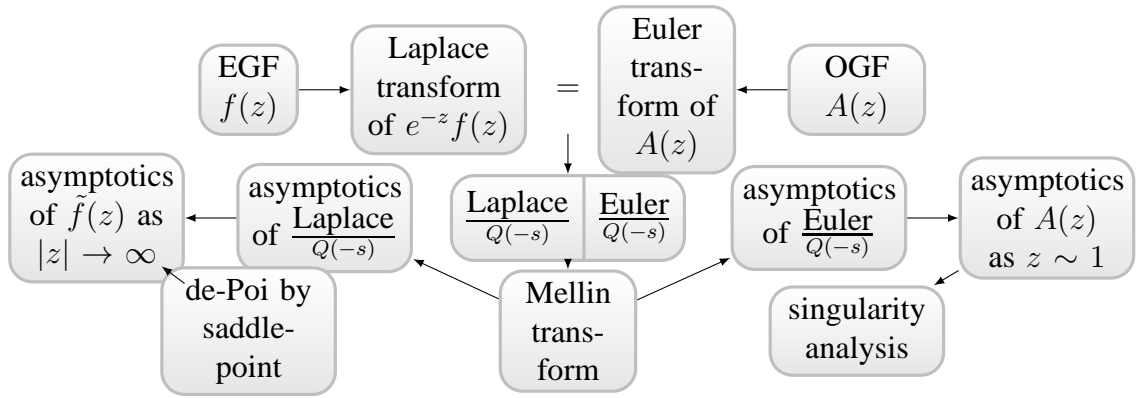
\begin{figure}[!h]
\begin{center}
\begin{tikzpicture}[
    s1/.style={
    rectangle,
    rounded corners=3mm,
    minimum size=6mm,
    very thick,
    draw=white!50!black!50,
    top color=white,
    bottom color=white!50!black!20,
}]%
\node [s1, text width=1cm, text centered] (n1) at (0,-.3) {EGF
$f(z)$};
\node [s1, text width=2cm, text centered] (n2) at (2.5,-.3)%
{Laplace transform of $e^{-z}f(z)$};
\path (n1) edge[-latex] (n2);
\node [inner sep=.5cm] (n) at (4.2,-.3) {$=$};
\node [s1, text width=1.5cm, text centered] (n3) at (5.5,-.3)
{Euler transform of $A(z)$};%
\node [s1, text width=1.5cm, text centered] (n4) at (8,-.3) {OGF
$A(z)$};
\path (n4) edge[-latex] (n3);%
\node [s1, text width=.8cm, text centered, rectangle split,
rectangle split parts=2,rotate=90] (n5) at (4.2,-2)
{\rotatebox{270}{$\frac{\mbox{Laplace}}{Q(-s)}$}
\nodepart{second}
\rotatebox{270}{$\frac{\mbox{Euler}}{Q(-s)}$}};
\path (n) edge[-latex] (n5);
\node [s1, text width=1.5cm, text centered] (n6) at (4.2,-3.5)
{Mellin transform};
\path (n5) edge[-latex] (n6);
\node [s1, text width=2cm, text centered] (n7) at (1,-2)
{asymptotics of $\frac{\mbox{Laplace}}{Q(-s)}$};
\path (n6) edge[-latex] (n7);
\node [s1, text width=2cm, text centered] (n8) at (-2,-2)
{asymptotics of $\tilde{f}(z)$ as $|z|\rightarrow \infty$};
\path (n7) edge[-latex] (n8);
\node [s1, text width=2cm, text centered] (n9) at (0,-3.5)
{de-Poi by saddle-point};
\path (n8) edge[-latex] (n9);
\node [s1, text width=2cm, text centered] (n10) at (7.5,-2)
{asymptotics of $\frac{\mbox{Euler}}{Q(-s)}$};
\path (n6) edge[-latex] (n10);
\node [s1, text width=2cm, text centered] (n11) at (10.5,-2)
{asymptotics of $A(z)$ as $z\sim 1$};
\path (n10) edge[-latex] (n11);
\node [s1, text width=2cm, text centered] (n12) at (8,-3.5)
{singularity analysis};
\path (n11) edge[-latex] (n12);
\end{tikzpicture}
\end{center}
\caption{\emph{A diagrammatic comparison of the major steps used in
the Laplace-Mellin (left-half) approach and the Flajolet-Richmond
(right-half) approach. Here EGF denotes ``exponential generating
function", OGF stands for ``ordinary generating function'' and
de-Poi is the abbreviation for de-Poissonization.
}}\label{fig-egf-ogf}
\end{figure}

The crucial reason why the two approaches are identical at certain
steps is that the Laplace transform of a Poisson generating function
is essentially equal to the Euler transform of an ordinary
generating function; or formally,
\begin{align}
    \int_0^\infty e^{-sz}\sum_{n\ge0} \frac{a_n}{n!}\,z^n \dd z
    &= \sum_{n\ge0} a_n (s+1)^{-n-1} \nonumber \\
    &= \frac1{s+1}A\left(\frac1{s+1}\right). \label{LT-OGF}
\end{align}
Thus the simple result in Proposition~\ref{prop-inv-laplace} closely
parallels that in singularity analysis. While identical at certain
steps, the two approaches diverge in their final treatment of the
coefficients, and the distinction here is typically that between the
saddle-point method and the singularity analysis, a situation
reminiscent of the use before and after Lagrange's inversion
formula; see for instance \cite{flajolet2009a}.

The relation (\ref{LT-OGF}) implies that the order estimate
(\ref{Lap-small}) for the Laplace transform at infinity can be
easily justified for all the generating functions we consider in
this paper since $A(0)=0$, implying that $A(z) = O(|z|)$ as
$|z|\to0$.

This comparison also suggests the possibility of developing
de-Poissonization tools by singularity analysis, which will be
investigated in details elsewhere.

\subsection{Variance of the internal path-length}

In this section, we apply the Laplace-Mellin-de-Poissonization
approach to the Poissonized variance with correction
\[
    \tilde{V}(z)
    := \tilde{f}_2(z)-\tilde{f}_1(z)^2-z\tilde{f}'_1(z)^2,
\]
aiming at proving Theorem~\ref{thm-KPS}. The starting point of
focusing on $\tilde{V}$ instead of on $\tilde{f}_2$ removes all
heavy cancellations involved when dealing with the variance, a
key step differing from all previous approaches.

\paragraph{Laplace and Mellin transform.} The following lemma will
be useful.
\begin{lmm} If \label{lmm-Vz-dfe}
\[
    \left\{\begin{array}{l}
        \tilde{f}_1(z)+\tilde{f}_1'(z)
        =2\tilde{f}_1(z/2)+\tilde{h}_1(z),\\
        \tilde{f}_2(z)+\tilde{f}_2'(z)
        =2\tilde{f}_2(z/2)+\tilde{h}_2(z),
    \end{array}\right.
\]
where all functions are entire with
$\tilde{f}_1(0)=\tilde{f}_2(0)=0$, then the function
$\tilde{V}(z) := \tilde{f}_2(z)
-\tilde{f}_1(z)^2-z\tilde{f}_1'(z)^2$ satisfies
\[
    \tilde{V}(z)+\tilde{V}'(z)
    =2\tilde{V}(z/2)+\tilde{g}(z),
\]
with $\tilde{V}(0)=0$, where
\[
    \tilde{g}(z)
    = z\tilde{f}_1''(z)^2 +\tilde{h}_2(z)-\tilde{h}_1(z)^2
    -z\tilde{h}_1'(z)^2 - 4\tilde{h}_1(z)\tilde{f}_1(z/2)
    -2z\tilde{h}_1'(z)\tilde{f}_1'(z/2)-2\tilde{f}_1(z/2)^2.
\]
\end{lmm}
\begin{proof}
Straightforward and omitted.
\end{proof}

By using the differential-functional equations (\ref{poi-mean}) and
(\ref{poi-2nd-mom}) for $\tilde{f}_1(z)$ and $\tilde{f}_2(z)$, we
see, by Lemma~\ref{lmm-Vz-dfe}, that
\begin{equation}\label{poi-var-V}
    \tilde{V}(z)+\tilde{V}'(z)
    = 2\tilde{V}(z/2)+z\tilde{f}''_1(z)^2,
\end{equation}
with $\tilde{V}(0)=0$.

Before applying the integral transforms, we need rough estimates of
$\tilde{V}(z)$ near $z=0$ and $z=\infty$. We have
\begin{equation}\label{asym-V}
    \tilde{V}(z)
    = \begin{cases}
        O\left(z^2\right),&\text{as}\ z\rightarrow 0^+;\\
        O(z^{1+\ve}),&\text{as}\ z\rightarrow\infty.
    \end{cases}
\end{equation}
These estimates follow from
\begin{align}\label{zf1d22}
    z\tilde{f}''_1(z)^2
    = \begin{cases}
        O(|z|),&\text{as}\ |z|\rightarrow 0;\\
        O(|z|^{-1}),&\text{as}\ |z|\rightarrow\infty,
    \end{cases}
\end{align}
which in turn result from $X_0=X_1=0$ and (\ref{tf1z-asymp}) (by the
proof of condition \textbf{(I)} of Proposition~\ref{prop-dst-tr}).
Indeed, the proof there shows that the same bounds hold uniformly
for $z\in\mathbb{C}$ with $|\arg(z)|\le\pi/2-\ve$.

We now apply the Laplace transform to both sides of (\ref{poi-var-V}).
First, observe that the Laplace transform of $\tilde{V}(z)$ exists
and is analytic in $\mathbb{C}\setminus(-\infty,0]$. Then,
by (\ref{poi-var-V}),
\[
   (s+1)\mathscr{L}[\tilde{V};s]
   = 4\mathscr{L}[\tilde{V};2s]+\tilde{g}^\star(s),
\]
where $\tilde{g}^\star(s):=\mathscr{L}[z\tilde{f}''_1;s]$.
Next the normalized Laplace transform
\[
    \bar{\mathscr{L}}[\tilde{V};s]
    := \frac{\mathscr{L}[\tilde{V};s]}{Q(-s)}
\]
satisfies
\[
    \bar{\mathscr{L}}[\tilde{V};s]
    = 4\bar{\mathscr{L}}[\tilde{V};2s]
    +\frac{\tilde{g}^\star(s)}{Q(-2s)}.
\]
By (\ref{asym-V}), we obtain
\[
    \mathscr{L}[\tilde{V};s]
    = \begin{cases}
        O(s^{-2-\ve}),&\text{as}\ s\rightarrow 0^+;\\
        O(s^{-3}),&\text{as}\ s\rightarrow\infty.
    \end{cases}
\]
From this and the asymptotic expansion (\ref{est-Q}) of $Q(-2s)$, it
follows that the Mellin transform of $\bar{\mathscr{L}}
[\tilde{V};s]$ exists in the half-plane $\Re(\omega)\ge 2+\ve$.
Consequently,
\[
    \mathscr{M}[\bar{\mathscr{L}}[\tilde{V};s];\omega]
    = \frac{G_2(\omega)}{1-2^{2-\omega}},\qquad(\Re(\omega)>2),
\]
where
\begin{align}\label{G2w-di}
    G_2(\omega)
    := \mathscr{M}\left[\frac{\tilde{g}^\star(s)}
    {Q(-2s)};\omega\right]
    = \int_{0}^{\infty}\frac{s^{\omega-1}}{Q(-2s)}
    \int_{0}^{\infty}e^{-zs}z\tilde{f}''_1(z)^2\dd z\dd{s}.
\end{align}
By (\ref{t-f1z}), we have
\[
    z\tilde{f}''_1(z)^2
    = Q_\infty^2\sum_{h,\ell\ge 0}
    \frac{1}{Q_hQ_\ell2^{h+\ell}}ze^{-z/2^h-z/2^\ell}.
\]
Substituting this and the partial fraction expansion
\[
    \frac1{Q(-2s)}
    = \frac1{Q_\infty}\sum_{j\ge0} \frac{(-1)^j2^{-\binom{j}{2}}}
    {Q_j(s+2^{-j})},
\]
into (\ref{G2w-di}), we obtain (\ref{G2w}).

\paragraph{Inverse Mellin and inverse Laplace transforms.}
For the Mellin inversion, we need more precise analytic properties
of $G_2(\omega)$. By (\ref{zf1d22}), we deduce that the Laplace
transform $\tilde{g}^\star(s)$ of $z\tilde{f}_1''(z)^2$ satisfies
\[
    \tilde{g}^\star(s)
    = \begin{cases}
        O(|\log s|),&\text{as}\ |s|\to 0;\\
        O(|s|^{-2}),&\text{as}\ |s|\to\infty
    \end{cases}
\]
uniformly in the cone $|\arg(s)|\le\pi-\ve$. Thus, by the asymptotic
expansion (\ref{est-Q}) for $Q(-2s)$ and Proposition 5 in
\cite{flajolet95a}, we have
\[
    |G_2(c+it)|
    = O\left(e^{-(\pi-\ve)|t|}\right),
\]
for large $|t|$ and $c>0$. Also the Mellin transform $G_2$ of
$\tilde{g}^\star(s)/Q(-2s)$ exists in the half-plane
$\Re(\omega)>0$. Consequently, by standard calculus of residues,
\[
    \bar{\mathscr{L}}[\tilde{V};s]
    = \frac{1}{\log 2}\sum_{k\in\mathbb{Z}}
    G_2(2+\chi_k)s^{-2-\chi_k}+O(|s|^{-\ve}),
\]
uniformly for $|s|\to 0$ and $|\arg(s)|\le\pi-\ve$. This in turn
yields the following expansion for $\mathscr{L}[\tilde{V};s]$
\[
    \mathscr{L}[\tilde{V};s]
    = \frac{1}{\log 2}\sum_{k\in\mathbb{Z}}G_2(2+\chi_k)s^{-2-\chi_k}
    +\frac{1}{\log 2}\sum_{k\in\mathbb{Z}}G_2(2+\chi_k)s^{-1-\chi_k}
    +O(|s|^{-\ve}),
\]
again uniformly for $|s|\to 0$ and $|\arg(s)|\le\pi-\ve$.

Finally, standard Laplace inversion gives
\begin{equation}\label{asym-variance}
    \tilde{V}(z)
    = \frac{z}{\log 2}\sum_{k\in\mathbb{Z}}\frac{G_2(2+\chi_k)}
    {\Gamma(2+\chi_k)}z^{\chi_k}+\frac{1}{\log 2}
    \sum_{k\in\mathbb{Z}}\frac{G_2(1+\chi_k)}{\Gamma(1+\chi_k)}
    z^{\chi_k}+O(|z|^{\ve-1}),
\end{equation}
uniformly for $|z|\to\infty$ and $|\arg(z)|\le\pi/2-\ve$.

Since $\tilde{f}_2(z) =\tilde{V}(z)+\tilde{f}_1(z)^2
+z\tilde{f}_1'(z)^2$, we see from (\ref{asym-variance}) and
(\ref{tf1z-asymp}) that
\[
    \tilde{f}_2(z) \asymp \tilde{f}_1(z)^2 \asymp |z|^2
    \log^2|z|\qquad(|\arg(z)|\le\pi/2-\ve).
\]
This proves Proposition~\ref{prop-PC-mean-var} and
Theorem~\ref{thm-KPS} by straightforward expansion. More refined
calculations give
\[
    \mathbb{V}(X_n)
    = \tilde{V}(n) -\frac{n}{2}\tilde{V}''(n)
    -\frac{n^2}{2}\tilde{f}''_1(n)^2 +O(n^{-1}),
\]
the two terms following $\tilde{V}(n)$ being both $O(1)$ and
periodic in nature. It is possible to further extend the same
idea and derive a full asymptotic expansion, which has also
its identity nature; details will be presented in a future paper.

\section{Bucket Digital Search Trees}

In this section, we extend the same approach to bucket digital
search trees ($b$-DSTs) in which each node can hold up to $b$ keys.
The construction rule is the same as DSTs, except that keys keep
staying in a node as long as its capacity remains less than $b$; see
Figure~\ref{fig-bdst} for a simple example with $b=2$. DSTs
correspond to $b=1$.

Note that when $b\ge2$ we can distinguish two different types of
total path-length: the total path-length of all keys (summing the
distance between each key to the root over all keys), which will be
referred to as the \emph{total key-wise path-length} (KPL) and the
total path-length of all nodes (summing the distance between each
node to the root over all nodes, regardless of the number of keys in
each node), referred to as the \emph{total node-wise path-length}
(NPL). When $b=1$ the two total path-lengths coincide. For
simplicity, we will use KPL and NPL, dropping the collective
adjective ``total''. While the expected values of both TPLs are of
order $n\log n$ under the same independent Bernoulli model, their
variances surprisingly turn out to exhibit very different behavior;
see Table~\ref{tb-all-pl}.

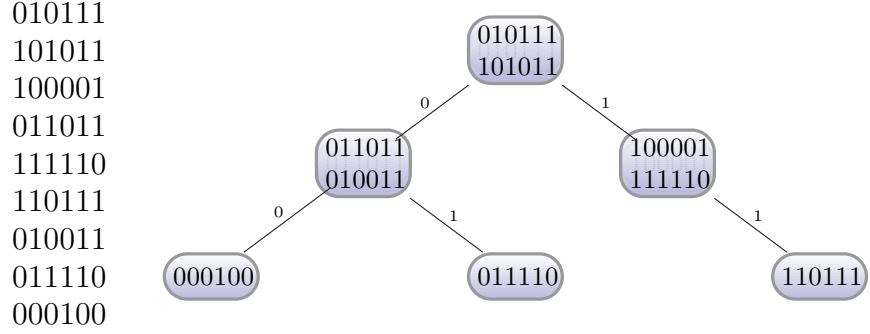
\begin{figure}[!h]
\begin{center}
\begin{tikzpicture}[
    s1/.style={
    rectangle,
    rounded corners=3mm,
    minimum size=6mm,
    very thick,
    draw=white!50!black!80,
    top color=white,
    bottom color=blue!50!black!30
}]
\path node at (1,0.5) {$010111$};%
\path node at (1,0) {$101011$};%
\path node at (1,-.5) {$100001$};%
\path node at (1,-1) {$011011$};%
\path node at(1,-1.5) {$111110$};%
\path node at (1,-2) {$110111$};%
\path node at (1,-2.5) {$010011$};%
\path node at (1,-3) {$011110$};%
\path node at (1,-3.5) {$000100$};%
\footnotesize%
\path node(a) at (7,0) [s1, text width=1cm, text centered]  {$010111$} ;%
\path node(b) at (7,0) [s1, text width=1cm, text centered]
{$010111$ $101011$};%
\path node(c) at (9,-1.5) [s1, text width=1cm, text centered] {$100001$};%
\path node(d) at (5,-1.5) [s1, text width=1cm, text centered] {$011011$};%
\path node(e) at (9,-1.5) [s1, text width=1cm, text centered]
{$100001$ $111110$};%
\path node(f) at (11,-3) [s1, text width=1cm, text centered] {$110111$};%
\path node(g) at (5,-1.5) [s1, text width=1cm, text centered]
{$011011$ $010011$};%
\path node(h) at (7,-3) [s1, text width=1cm, text centered] {$011110$};%
\path node(i) at (3,-3) [s1, text width=1cm, text centered] {$000100$};%
\path [draw,-,black!90] (b) -- (c) node[above,pos=.6,black]{{\tiny $1$}};%
\path [draw,-,black!90] (b) -- (d) node[above,pos=.6,black]{{\tiny $0$}};%
\path [draw,-,black!90] (e) -- (f) node[above,pos=.6,black]{{\tiny $1$}};%
\path [draw,-,black!90] (g) -- (h) node[above,pos=.6,black]{{\tiny $1$}};%
\path [draw,-,black!90] (d) -- (i) node[above,pos=.6,black]{{\tiny $0$}};%
\end{tikzpicture}
\end{center}
\caption{\emph{A $2$-DST with nine keys. The total key-wise
path-length is equal to $4\times 1+3\times 2=10$ and the total
node-wise path-length equals $2\times1 + 3\times 2 =
8$.}}\label{fig-bdst}
\end{figure}

\subsection{Key-wise path-length (KPL)}

We assume the same independent Bernoulli model for the input
strings. Let $X_n$ denote the KPL in a random $b$-DST built from $n$
random stings. Then by definition and the independence model
assumption
\begin{align}\label{bdst-Xn-rr}
    X_{n+b}
    \stackrel{d}{=}X_{B_n}+X_{n-B_n}^{*}+n,\qquad (n\ge 0)
\end{align}
with the initial conditions $X_0=\cdots=X_{b-1}=0$. Here
$B_n\sim\text{Binomial}(n,1/2), X_n\stackrel{d}{=}X_n^{*}$, and
$X_n,X_n^{*},B_n$ are independent.

\paragraph{Known and new results.} Hubalek \cite{hubalek00a} showed,
by the Flajolet-Richmond approach, that the mean satisfies
\[
    \mathbb{E}(X_n)
    = (n+b)\log_2n + n\left(c_2+\varpi_3(\log_2n)\right)
    +c_3 + \varpi_4(\log_2n) +O\left(n^{-1}\log n\right),
\]
where $c_2, c_3$ are effectively computable constants and $\varpi_3$
and $\varpi_4$ are very smooth periodic functions. He also proved
that the variance is asymptotically linear
\[
    \mathbb{V}(X_n)
    = n\left(C_h + \varpi_h(\log_2n)\right)+O((\log n)^2),
\]
where $C_h$ is expressed in terms of a very long, involved
expression and $\varpi_h$ is a periodic function.

We improve this estimate by deriving a much simpler expression for
the periodic function, including its average value $C_h$. To state
our result, we define the following functions. Let
\begin{align}\label{gz-var}
\begin{split}
    \tilde{g}(z)
    &:= \left(\sum_{0\le j\le b} \binom{b}{j}
    \tilde{f}_1^{(j)}(z) \right)^2 + z
    \left(\sum_{0\le j\le b} \binom{b}{j}
    \tilde{f}_1^{(j+1)}(z) \right)^2 \\ &\qquad
    - \sum_{0\le j\le b}
    \binom{b}{j} \left(\left(\tilde{f}_1^2(z)\right)^{(j)}
    +\left(z\tilde{f}_1'(z)^2\right)^{(j)}\right).
\end{split}
\end{align}
It is easily seen that $\tilde{g}(z)$ is of the form
\begin{equation}\label{form-g2}
    \tilde{g}(z)
    = \sum_{2\le i_1,i_2\le b}\tilde{g}_{i_1,i_2}
    \tilde{f}_1^{(i_1)}(z)\tilde{f}_1^{(i_2)}(z)
    +z\sum_{2\le i_1,i_2\le b+1}\tilde{g}'_{i_1,i_2}
    \tilde{f}_1^{(i_1)}(z)\tilde{f}_1^{(i_2)}(z),
\end{equation}
where $\tilde{g}_{i_1,i_2},\tilde{g}'_{i_1,i_2}\ge 0$ are given
explicitly by
\begin{align*}
    \tilde{g}_{i_1,i_2}
    &= \binom{b}{i_1}\binom{b}{i_2}
    -\binom{b}{i_1}\binom{b-i_1}{i_2}
    -(b-i_1+1)\binom{b}{i_1-1}\binom{b-i_1}{i_2-1},\\
    \tilde{g}'_{i_1,i_2}
    &= \binom{b}{i_1-1}\binom{b}{i_2-1}
    -\binom{b}{i_1-1}\binom{b-i_1+1}{i_2-1},
\end{align*}
both coefficients being symmetric in $i_1$ and $i_2$. Define
\[
    G_2(\omega)
    = \int_{0}^{\infty}\frac{s^{\omega-1}}{Q(-2s)^b}
    \int_{0}^{\infty}e^{-zs}\tilde{g}(z)\dd z\dd{s},
\]
which is well-defined for $\Re(\omega)>0$, as we will see later.

\begin{thm} The variance of the total key-wise path-length of random
$b$-DSTs of $n$ strings satisfies
\begin{align} \label{var-bdst-TKPL}
    \mathbb{V}(X_n)
    = n \left(C_h + \varpi_h(\log_2n)\right) + O(1),
\end{align}
where
\[
    C_h
    = \frac{G_2(2)}{\log 2}
    = \frac{1}{\log 2}\int_{0}^{\infty}\frac{s}{Q(-2s)^b}
    \int_{0}^{\infty}e^{-zs}\tilde{g}(z)\dd z\dd{s},
\]
and
\[
    \varpi_h(t)
    = \frac1{\log 2}\sum_{k\in \mathbb{Z}\setminus\{0\}}
    \frac{G_2(2+\chi_k)}{\Gamma(2+\chi_k)} e^{2k\pi i t}.
\]
\end{thm}

By straightforward truncations, expansions and approximations, we
obtain the following numerical values for $b=1,\ldots, 5$.
\begin{center}
\begin{tabular}{|c|c|c|c|c|c|}\hline
$b$ & 1 & 2 & 3 & 4 & 5 \\ \hline
$C_h$ & $0.26600$ & $0.13260$ & $0.09004$ & $0.06958$ & $0.05781$
\\ \hline
\end{tabular}
\end{center}
More powerful means are needed to be developed if more degree of
precision is required.

\paragraph{Generating functions.} From (\ref{bdst-Xn-rr}), it follows
that the moment generating function $M_n(y):=\mathbb{E}(e^{X_n y})$
can be recursively computed by the relation
\[
    M_{n+b}(y)
    = \frac{e^{ny}}{2^n}\sum_{0\le j\le n}\binom{n}{j}
    M_j(y)M_{n-j}(y)\qquad(n\ge0),
\]
with $M_n(y)=1$ for $0\le n<b$. The bivariate exponential generating
function $F(z,y)$ then satisfies the equation
\[
    \frac{\partial^b}{\partial z^b}\,F(z,y)
    = F\left(\frac{e^y z}2,y\right)^2,
\]
with $F^{(j)}(0,y)=1$ for $0\le j<b$, and we have the nonlinear
equation for the Poisson generating function $\tilde{F}(z,y) :=
e^{-z} F(z,y)$
\begin{align} \label{PGF-bdst-TKPL}
    \sum_{0\le j\le b}\binom{b}{j}\tilde{F}^{(j)}(z,y)
    = e^{(e^y-1)z}\tilde{F}\left(\frac{e^yz}{2},y\right)^2,
\end{align}
with $\tilde{F}(0,y)=1$.

From this form, the asymptotic analysis of the mean value and that
of the variance proceed along exactly the same line we developed in
the previous section. Thus we briefly sketch the principal steps of
the analysis, leaving the details to the interested reader.

\paragraph{The expected value of $X_n$.} From (\ref{PGF-bdst-TKPL}),
we derive the following differential-functional equation for the
Poisson generating function of the mean
\[
    \sum_{0\le j\le b}\binom{b}{j}\tilde{f}_1^{(j)}(z)
    =2\tilde{f}_1(z/2)+z,
\]
with the initial conditions $\tilde{f}_1^{(j)}(0)=0$ for $0\le j<
b$.

Before applying the Laplace-Mellin approach, we need first a
transfer-type result similar to Proposition~\ref{prop-dst-tr}.
\begin{prop} \label{prop-bdst-tr}
Let $\tilde{f}(z)$ and $\tilde{g}(z)$ be entire functions satisfying
\begin{align}\label{dfe-b-DST}
    \sum_{0\le j\le b}\binom{b}{j}\tilde{f}^{(j)}(z)
    = 2\tilde{f}(z/2)+\tilde{g}(z),
\end{align}
with $f(0)=0$. Then
\[
    \tilde{g}\in\JS \quad
    \Longleftrightarrow\quad\tilde{f}\in\JS.
\]
\end{prop}
\begin{proof} (Sketch)
The same proof as that for Proposition~\ref{prop-dst-tr} applies
\emph{mutatis mutandis} to (\ref{dfe-b-DST}). The only difference is
that we now have
\[
    f^{(b)}(z)
    = 2e^{z/2}f(z/2)+g(z),
\]
where $f(z) := e^z \tilde{f}(z)$ and $g(z) := e^z \tilde{g}(z)$, so
that (\ref{dst-ir-f}) has the extended representation
\begin{align*}
    f(z)
    &= \frac1{(b-1)!}\int_0^z (z-t)^{b-1}
    \left(2e^{t/2}f(t/2)+g(t)\right) \dd t\\
    &= \frac{z^b}{(b-1)!}\int_0^1 (1-t)^{b-1}
    \left(2e^{tz/2}f(tz/2)+g(tz)\right)\dd t,
\end{align*}
and
\[
    \tilde{f}(z)
    = \frac{z^b}{(b-1)!}\int_0^1 (1-t)^b e^{-(1-t)z}
    \left(2\tilde{f}(tz/2)+\tilde{g}(z)\right) \dd t.
\]
All required estimates can be derived by the same arguments used
there.
\end{proof}

The Laplace transform of $\tilde{f}_1$ now satisfies the functional
equation
\[
    (s+1)^b\mathscr{L}[\tilde{f}_1;s]
    = 4\mathscr{L}[\tilde{f}_1;2s]+s^{-2},
\]
for $\Re(s)>0$. From this equation, we obtain
\[
   \mathscr{L}[\tilde{f}_1;s]
    = \frac1{s^2}\sum_{j\ge0}\frac1{(s+1)^b\cdots
    (1+2^js)^b},
\]
which extends (\ref{Lf1s}). From this series and partial fraction
expansions, we can derive a close-form expression for
$\tilde{f}_1(z)$, which becomes messy especially for large $b$.
Define as before $\bar{\mathscr{L}}[\tilde{f}_1;s]
:=\mathscr{L}[\tilde{f}_1;s]/Q(-s)^b$. Then we obtain
\[
    \bar{\mathscr{L}}[\tilde{f}_1;s]
    =4\bar{\mathscr{L}}[\tilde{f}_1;2s]+\frac1{Q(-2s)^bs^2}.
\]
This relation is almost the same as (\ref{L-bar}). Thus the same
Mellin analysis given there carries over and we deduce that
\begin{equation*}
\begin{split}
    \mathscr{L}[\tilde{f}_1;s]
    &= \frac{1}{s^2}\log_2\frac{1}{s} + \frac{1}{s^2}
    \left(\frac12+\frac{c_4}{\log2}+
    \frac{1}{\log2}\sum_{k\in\mathbb{Z}\setminus\{0\}}
    G_1(2+\chi_k) s^{-\chi_k}\right) \\
    &\qquad +\frac{b}s\log\frac1s+O(|s|^{-1}),
\end{split}
\end{equation*}
uniformly for $|s|\to 0$ and $|\arg(s)|\le\pi-\ve$, where
\begin{align*}
    G_1(\omega)
    := \int_{0}^{\infty}\frac{s^{\omega-3}}{Q(-2s)^b}\dd s,
\end{align*}
and
\begin{align*}
    c_4
    &:= \lim_{\omega\to2}\left(G_1(\omega)
    -\frac{1}{\omega-2}\right) \\
    &= \int_0^1\frac1s\left(\frac1{Q(-2s)^b}-1\right) \dd s
    + \int_1^\infty \frac1{sQ(-2s)^b} \dd s.
\end{align*}
Consequently, by the Laplace inversion,
\begin{equation}\label{asym-f1-bucket}
    \tilde{f}_1(z)
    =(z+b)\log_2 z+z\left(\frac{1}{2}+
    \frac{G_1(2)+\gamma-1}{\log 2}+\frac{1}{\log 2}
    \sum_{k\in\mathbb{Z}\setminus\{0\}}\frac{G_1(2+\chi_k)}
    {\Gamma(2+\chi_k)} z^{\chi_k}\right)+ O(1),
\end{equation}
uniformly for $|z|\to\infty$ and $|\arg(z)|\le\pi/2-\ve$. From this
and Propositions~\ref{prop-bdst-tr} and \ref{prop-PC-asymp}, we
obtain
\[
    \mathbb{E}(X_n)
    = \sum_{0\le j<2k}\frac{\tilde{f}^{(j)}(n)}{j!}\,\tau_j(n)
    +O\left(n^{-1+k}\right),
\]
for any $k=1,2,\dots$. Finally,
\[
    \mathbb{E}(X_n)
    =(n+b)\log_2 n+n\left(\frac{1}{2}+\frac{G_1(2)+\gamma-1}
    {\log 2}+\frac{1}{\log 2}\sum_{k\in\mathbb{Z}\setminus\{0\}}
    \frac{G_1(2+\chi_k)}{\Gamma(2+\chi_k)}n^{\chi_k}\right)+O(1).
\]

\paragraph{Variance of $X_n$.} The analysis here is again similar to
that for the mean. Let $\tilde{f}_2(z)$ denote the Poisson
generating function of the second moment $\mathbb{E}(X_n^2)$. Then,
by (\ref{PGF-bdst-TKPL}),
\[
    \sum_{0\le j\le b}\binom{b}{j}\tilde{f}_2^{(j)}(z)
    = 2\tilde{f}_2(z/2)+2\tilde{f}_1(z/2)^2+4z\tilde{f}_1(z/2)
    +2z\tilde{f}'_1(z/2)+z+z^2,
\]
with the first $b$ Taylor coefficients zero. Define again
\[
    \tilde{V}(z)
    =\tilde{f}_2(z)-\tilde{f}_1(z)^2-z\tilde{f}'_1(z)^2.
\]
Then $\tilde{V}(z)$ satisfies
\[
    \sum_{0\le j\le b}\binom{b}{j}\tilde{V}^{(j)}(z)
    =2\tilde{V}(z/2)+\tilde{g}(z),
\]
where $\tilde{g}(z)$ is given in (\ref{gz-var}).

By the representations (\ref{form-g2}) and (\ref{asym-f1-bucket}),
we have
\begin{align*}
    \tilde{g}(z)
    =\begin{cases}
        O(|z|),&\text{as}\ |z|\to 0;\\
        O(|z|^{-1}),&\text{as}\ |z|\to\infty,
    \end{cases}
\end{align*}
uniformly in the sector $|\arg(z)|\le\pi/2-\ve$. This is similar to
the corresponding estimate (\ref{zf1d22}) in the analysis of the
variance in the previous section. The same procedure there applies
and we deduce (\ref{var-bdst-TKPL}).

\subsection{Node-wise path-length (NPL)}\label{tnpl}

We consider in this section the total node-wise path-length (NPL).
Under the same independent Bernoulli model, we still use $X_n$ to
denote the NPL in a random $b$-DST of $n$ binary strings with node
capacity $b\ge2$. Also let $N_n$ stand for the total number of nodes
(space requirement) in random $b$-DST of $n$ strings. Despite its
being one of the most natural shape measures for $b$-DSTs, the
consideration of $X_n$ here seems to be new. For $N_n$, it is known
that the distribution is asymptotically normal with the mean and the
variance both asymptotically $n$ times a different smooth periodic
function; see \cite{hubalek02a}. In contrast to
(\ref{var-bdst-TKPL}) for the variance of KPL, what is unexpected
and surprising here is that the variance of $X_n$ is of order
$n(\log n)^2$.

\begin{thm} Assume $b\ge2$. The mean of $N_n$ and that of $X_n$
satisfy the following asymptotic relations.
\begin{equation}\label{ENnXn-bdst}
\left\{\begin{split}
    \mathbb{E}(N_n)
    &= nP_{1,0}(\log_2n)+O(1),\\
    \mathbb{E}(X_n)
    &= n(\log_2n)P_{1,0}(\log_2n)+nP_{0,1}^{[2]}(\log_2n)
    +(\log_2n)P_{0,1}^{[3]}(\log_2n)+O(1);
\end{split}\right.
\end{equation}
and the variances of $N_n$ and $X_n$ satisfy
\begin{equation}\label{VNnXn-bdst}
\left\{\begin{split}
    \mathbb{V}(N_n)
    &=nP_{2,0}(\log_2 n)+O(1),\\
    \textnormal{Cov}(N_n,X_n)
    &=n(\log_2n)
    P_{2,0}(\log_2n)+nP^{[2]}_{1,1}(\log_2 n)+(\log n)
    P^{[3]}_{1,1}(\log_2n)+O(1),\\
    \mathbb{V}(X_n)
    &=n(\log_2 n)^2P_{2,0}(\log_2n)+n(\log_2n)
    P^{[2]}_{0,2}(\log_2 n)+n P^{[3]}_{0,2}(\log_2 n)\\
    &\qquad +(\log n)^2 P^{[4]}_{0,2}(\log_2n)
    +(\log_2n)P^{[5]}_{0,2}(\log_2n)+O(1),
\end{split}\right.
\end{equation}
where the $P_{\cdot,\cdot}$'s are all computable, smooth, $1$-periodic
functions.
\end{thm}
Intuitively, that the variance of NPL is larger than that of KPL can
be seen from the definition of NPL, which depends on the random
variable $N_n$ (see (\ref{bdst-NnXn})), while on the other hand, KPL
depends on $n$ only (in addition to on the two subtrees).  The
following figure shows the first few values of the variance of NPL
and that of KPL.
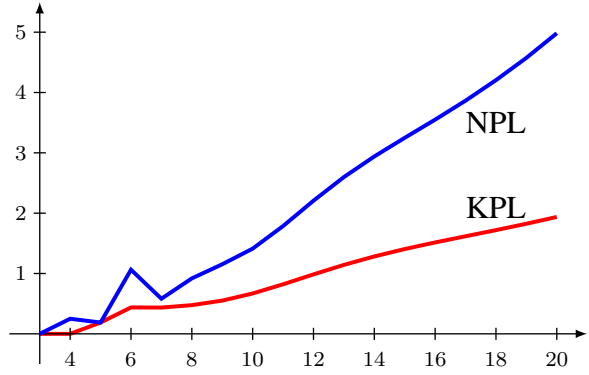
\begin{center}
\begin{tikzpicture}[xscale=0.4,yscale=0.8]
\draw[color=red, line width=1.5pt] plot
coordinates{(3,0.000000)(4,0.000000)(5,0.187500)(6,0.437500)
(7,0.435303)(8,0.475525)(9,0.550773)(10,0.668912)(11,0.821080)
(12,0.985353)(13,1.142456)(14,1.282745)(15,1.405386)(16,1.514793)
(17,1.617311)(18,1.719037)(19,1.824704)(20,1.937302)};
\foreach \x/\xtext in {4,6,...,20}
  \draw[shift={(\x,0)}, line width=.5pt]
  (0,-.1) -- (0,.1) node[below=5pt] {\scriptsize$\xtext$};%
\foreach \y/\ytext in {1,2,...,5}
  \draw[shift={(3,\y)}, line width=.5pt]
  (-.2,0) -- (.2,0) node[left=3pt] {\scriptsize$\ytext$};%
\draw[-latex, line width=.5pt] (2,0) -- (21,0) node[right] {};%
\draw[-latex, line width=.5pt] (3,-0.5) -- (3,5.5) node[above] {};%
\draw (18,3.5) node{NPL};%
\draw (18,2.1) node{KPL};%
\draw[color=blue!100, line width=1.5pt] plot
coordinates{(3,0)(4,0.25)(5,0.1875)(6,1.0625)(7,0.581787109)
(8,0.918972015)(9,1.152950227)(10,1.411997533)(11,1.786379786)
(12,2.20663314)(13,2.59814695)(14,2.940167301)(15,3.249019713)
(16,3.550402171)(17,3.864858061)(18,4.204876631)(19,4.574390419)
(20,4.982595028)};
\end{tikzpicture}
\end{center}
We see that the variance of NPL increases faster than that of KPL.

Note that the periodic functions of the dominant terms are all
equal, implying that the correlation coefficient of $N_n$ and
$X_n$ is asymptotically $1$.

On the other hand, the mean value $c_{1,0}$ of $P_{1,0}(t)$ is given
by
\[
    c_{1,0} = \frac1{\log 2}\int_0^\infty \frac{(s+1)^{b-1}}
    {Q(-2s)^b}\dd s ;
\]
numerical approximations to $c_{1,0}$ for the first few $b$ are
given as follows.
\begin{center}
\begin{tabular}{|c|c|c|c|c|c|c|}\hline
$b$ & 1 & 2 & 3 & 4 & 5 & 6\\ \hline
$c_{1,0}$ & $1$ & $0.57470$ & $0.40698$ & $0.31594$ & $0.25849$ &
$0.21885$ \\ \hline
\end{tabular}
\end{center}

Note that when $b=1$
\[
    c_{1,0}
    = \frac1{\log 2} \int_0^\infty \frac{\dd s}{Q(-2s)}
    = 1,
\]
by (\ref{G1w}), which is consistent with the fact that $N_n\equiv n$
in this case.

When $b=2$, we see that about $42.5\%$ of nodes on average contain
two keys and $14\%$ of nodes a single key. The storage utilization
is thus not very bad.

From (\ref{ENnXn-bdst}) and these numerical values, we see that, in
contrast to the expected KPL, which is asymptotic to $n\log_2n$ for
all $b$, the expected NPL provides a better indication of the
``shape variation" of random $b$-DSTs.

Our analysis is based on the following straightforward
distributional recurrences
\begin{equation}\label{bdst-NnXn}
    \left\{\begin{array}{l}
    N_{n+b}
    \stackrel{d}{=}N_{B_n}+N_{n-B_n}^{*}+1,\\
    X_{n+b}
    \stackrel{d}{=}X_{B_n}+X_{n-B_n}^{*}+N_{B_n}+N_{n-B_n}^{*},
    \end{array}
    \right.\qquad (n\ge 0),
\end{equation}
with the initial conditions $N_0=0,N_1=\cdots=N_{b-1}=1$ and
$X_0=\cdots=X_{b-1}=0$. Here again $B_n\sim\text{Binomial}(n,1/2),
N_n\stackrel{d}{=}N_n^{*}, X_{n}\stackrel{d}{=}X_n^{*}$ and
$X_n,X_n^{*},B_n$ as well as $N_n,N_n^{*},B_n$ are independent.

\paragraph{Generating functions.}
Define $M_{n}(u,v)=\mathbb{E}(e^{N_nu+X_nv})$. Then
(\ref{bdst-NnXn}) translates into the recurrence
\[
    M_{n+b}(u,v)
    =e^{u}2^{-n}\sum_{j=0}^{n}\binom{n}{j}
    M_j(u+v,v)M_{n-j}(u+v,v),\qquad (n\ge 0),
\]
with $M_0(u,v)=1,M_1(u,v)=\cdots=M_{b-1}(u,v)=e^u$. Next, let
\[
    F(z,u,v)
    :=\sum_{n\ge 0}\frac{M_n(u,v)}{n!}z^n.
\]
Then the recurrence relation gives
\[
    \frac{\partial^b}{\partial z^b}F(z,u,v)
    =e^uF\left(\frac z2,u+v,v\right)^2,
\]
and the Poisson generating function $\tilde{F}(z,u,v)
:=e^{-z}F(z,u,v)$ satisfies
\begin{equation}\label{equ-tildePzuv}
    \sum_{0\le j\le b}\binom{b}{j}
    \frac{\partial^j}{\partial z^j}\tilde{F}(z,u,v)
    = e^{u}\tilde{F}\left(\frac z2,u+v,v\right)^2,
\end{equation}
with the initial conditions $\tilde{F}(z,u,v)= 1+(e^u-1)\sum_{1\le
j<b} (-1)^{j-1}z^j/j!+\cdots$.

For the moments, if we expand $\tilde{F}(z,u,v)$ in terms of $u$ and
$v$,
\[
    \tilde{F}(z,u,v)
    =\sum_{m\ge 0}\frac{1}{m!}\sum_{0\le j\le m}
    \binom{m}{j}\tilde{f}_{j,m-j}(z)u^jv^{m-j},
\]
then $\tilde{f}_{j,m-j}(z)$ is the Poisson generating function of
$\mathbb{E}(N_{n}^{j}X_n^{m-j})$. Thus all moments of $X_n$ and
$N_n$ or their products can be computed by taking suitable
derivatives of (\ref{equ-tildePzuv}) with respect to $u$ and $v$ and
then substituting $u=v=0$.

\paragraph{Expected number of nodes and expected node-wise path
length.} By taking first derivatives of (\ref{equ-tildePzuv}), we
obtain
\begin{align} \label{bdst-tnpl-mean}
    \left\{\begin{array}{l} \displaystyle
    \sum_{0\le j\le b}\binom{b}{j}\tilde{f}_{1,0}^{(j)}(z)
    =2\tilde{f}_{1,0}(z/2)+1,\\  \displaystyle
    \sum_{0\le j\le b}\binom{b}{j}\tilde{f}_{0,1}^{(j)}(z)
    =2\tilde{f}_{0,1}(z/2)+2\tilde{f}_{1,0}(z/2),
    \end{array}\right.
\end{align}
the initial conditions being $\tilde{f}_{1,0}(0)=0$,
$\tilde{f}_{1,0}^{(j)}(0)=(-1)^{j-1}$ for $1\le j<b$ and
$\tilde{f}_{0,1}^{(j)}(0)=0$ for $0\le j< b$.

We can apply the Laplace-Mellin approach as before, starting from
the mean of $N_n$. Note that
\[
    \mathscr{L}[\tilde{f}^{(j)};s]
    = s^j\mathscr{L}[\tilde{f};s] - \sum_{0\le \ell<j} s^\ell
    \tilde{f}^{(j-1-\ell)}(0)\qquad(j=0,1,\dots),
\]
provided that the Laplace transform exists for $\Re(s)>0$. This
gives
\[
    (s+1)^b \mathscr{L}[\tilde{f}_{1,0};s]
    = 4 \mathscr{L}[\tilde{f}_{1,0};2s] +
    \tilde{g}_{1,0}^\star(s),
\]
where
\begin{align*}
    \tilde{g}_{1,0}^\star(s)
    &:= \frac1s +\sum_{0\le \ell \le b-2}s^\ell
    \sum_{\ell\le j\le b-2}
    \binom{b}{j+2}\tilde{f}^{(j+1-\ell)}_{1,0}(0)\\
    &= \frac1s+\sum_{1\le j<b}\binom{b-1}j s^{j-1}\\
    &= s^{-1}(s+1)^{b-1}.
\end{align*}
Unlike all previous cases, iterating this functional equation leads
to a divergent series. Although this problem can be solved by
subtracting a sufficient number of initial terms of
$\tilde{f}_{1,0}(z)$, the approach we use does not rely on this and
avoids completely such a consideration.

Let $\bar{\mathscr{L}}[\tilde{f}_{1,0};s]
 := \mathscr{L}[\tilde{f}_{1,0};s]/Q(-s)^b$.
Then
\[
    \mathscr{M}[\bar{\mathscr{L}}[\tilde{f}_{1,0};s];\omega]
    =\frac{G_{1,0}(\omega)}{1-2^{2-\omega}},\qquad (\Re(\omega)>2),
\]
where
\begin{align*}
    G_{1,0}(\omega)
    := \int_{0}^{\infty}\frac{s^{\omega-2}}{Q(-2s)^b}
    (s+1)^{b-1}\dd s,
\end{align*}
for $\Re(\omega)>1$.

From this, we deduce that
\begin{equation}\label{mean-1}
    \tilde{f}_{1,0}(z)
    = zP_{1,0}(\log_2z)+O(1),
\end{equation}
uniformly for $|z|\to\infty$ and $|\arg(z)|\le\pi/2-\ve$, where
$P_{1,0}(t)$ is a periodic function with the Fourier series
representation
\begin{align*}
    P_{1,0}(t)
    := \frac{1}{\log 2} \sum_{k\in\mathbb{Z}}
    \frac{G_{1,0}(2+\chi_k)}{\Gamma(2+\chi_k)} e^{2k\pi it},
\end{align*}
the series being absolutely convergent. From this we deduce the
first approximation of (\ref{ENnXn-bdst}).

We now turn to the expected NPL $\mathbb{E}(X_n)$. By
(\ref{bdst-tnpl-mean}), we have
\[
    (s+1)^b\mathscr{L}[\tilde{f}_{0,1};s]
    = 4 \mathscr{L}[\tilde{f}_{0,1};2s] +
    4\mathscr{L}[\tilde{f}_{1,0};2s].
\]
Let $\bar{\mathscr{L}}[\tilde{f}_{0,1};s]:=
 \mathscr{L}[\tilde{f}_{0,1};s]/Q(-s)^b$. Then
\[
    \mathscr{M}[\bar{f}_{0,1};\omega]
    =\frac{2^{2-\omega}G_{1,0}(\omega)}
    {\left(1-2^{2-\omega}\right)^2},\qquad (\Re(\omega)>2).
\]
From this we deduce that
\begin{equation}\label{mean-2}
    \tilde{f}_{0,1}(z)
    = z(\log_2 z)P_{0,1}^{[1]}(\log_2 z)
    +zP_{0,1}^{[2]}(\log_2 z)+(\log_2z)P_{0,1}^{[4]}(\log_2 z)
    +O(1),
\end{equation}
uniformly for $|z|\to\infty$ and $|\arg(z)|\le\pi/2-\ve$, where
$P_{0,1}^{[1]}(t),P_{0,1}^{[2]}(t),P_{0,1}^{[4]}(t)$ are smooth,
1-periodic functions whose Fourier coefficients are given by
\begin{align*}
    P_{0,1}^{[1]}(t)
    &= P_{1,0}(t)
    = \frac1{\log 2} \sum_{k\in\mathbb{Z}}
    \frac{G_{1,0}(2+\chi_k)}{\Gamma(2+\chi_k)}\,e^{2k\pi it},\\
    P_{0,1}^{[2]}(t)
    &= -\frac1{(\log 2)^2} \sum_{k\in\mathbb{Z}}
    \frac{G_{1,0}'(2+\chi_k)\psi(2+\chi_k)
    -G_{1,0}(2+\chi_k)}{\Gamma(2+\chi_k)}\,e^{2k\pi it},\\
    P_{0,1}^{[4]}(t)
    &= \frac b{\log 2} \sum_{k\in\mathbb{Z}}
    \frac{G_{1,0}(2+\chi_k)}{\Gamma(1+\chi_k)}\,e^{2k\pi it}.
\end{align*}
Here $\psi(z)$ denotes the derivative of $\log\Gamma(z)$ and all
series are absolutely convergent. This proves (\ref{ENnXn-bdst}).

\paragraph{Variance.} Taking second derivatives in
(\ref{equ-tildePzuv}) and substituting $u=v=0$ gives
\[
\left\{\begin{split}
    \sum_{0\le j\le b}\binom{b}{j}\tilde{f}^{(j)}_{2,0}(z)
    &=2\tilde{f}_{2,0}(z/2)+2\tilde{f}_{1,0}(z/2)^2
    +4\tilde{f}_{1,0}(z/2)+1,\\
    \sum_{0\le j\le b}\binom{b}{j}\tilde{f}^{(j)}_{1,1}(z)
    &=2\tilde{f}_{1,1}(z/2)+2\tilde{f}_{2,0}(z/2)
    +2(\tilde{f}_{1,0}(z/2)+\tilde{f}_{0,1}(z/2))
    (\tilde{f}_{1,0}(z/2)+1),\\
    \sum_{0\le j\le b}\binom{b}{j}\tilde{f}^{(j)}_{0,2}(z)
    &=2\tilde{f}_{0,2}(z/2)+4\tilde{f}_{1,1}(z/2)
    +2\tilde{f}_{0,2}(z/2)
    +2(\tilde{f}_{1,0}(z/2)+\tilde{f}_{0,1}(z/2))^2,
\end{split}\right.
\]
with the initial conditions $\tilde{f}_{2,0}^{(j)}(0)=(-1)^{j-1}$
for $1\le j< b$ and $\tilde{f}_{2,0}(0) = \tilde{f}_{1,1}^{(j)}(0)
= \tilde{f}_{0,2}^{(j)}(0) = 0$, for $0\le j< b$.

The remaining calculations follow the same pattern of proof we used
above but become much more involved. We begin with
\[
\left\{\begin{split}
    \tilde{V}(z)
    &=\tilde{f}_{2,0}(z)-\tilde{f}_{1,0}(z)^2
    -z\tilde{f}'_{1,0}(z)^2,\\
    \tilde{U}(z)
    &=\tilde{f}_{1,1}(z)-\tilde{f}_{1,0}(z)\tilde{f}_{0,1}(z)
    -z\tilde{f}'_{1,0}(z)\tilde{f}'_{0,1}(z),\\
    \tilde{W}(z)
    &=\tilde{f}_{0,2}(z)-\tilde{f}_{0,1}(z)^2
    -z\tilde{f}'_{0,1}(z)^2.
\end{split}\right.
\]
Then we deduce
\[
\left\{\begin{split}
    \sum_{0\le j\le b}\binom{b}{j}\tilde{V}^{(j)}(z)
    &=2\tilde{V}(z/2)+\tilde{g}_{2,0}(z),\\
    \sum_{0\le j\le b}\binom{b}{j}\tilde{U}^{(j)}(z)
    &=2\tilde{U}(z/2)+\tilde{g}_{1,1}(z),\\
    \sum_{0\le j\le b}\binom{b}{j}\tilde{W}^{(j)}(z)
    &=2\tilde{W}(z/2)+\tilde{g}_{0,2}(z),
\end{split}\right.
\]
where
\[
\left\{
\begin{split}
    \tilde{g}_{2,0}(z)
    &=\left(\sum_{0\le j\le b}\binom{b}{j}
    \tilde{f}^{(j)}_{1,0}(z)\right)^2
    +z\left(\sum_{0\le j\le b}\binom{b}{j}
    \tilde{f}^{(j+1)}_{1,0}(z)\right)^2\\
    &\qquad -\sum_{0\le j\le b}\binom{b}{j}
    \left(\tilde{f}_{1,0}(z)^2
    +z\tilde{f}'_{1,0}(z)^2\right)^{(j)},\\
    \tilde{g}_{1,1}(z)
    &=2\tilde{V}(z/2)+\left(\sum_{0\le j\le b}\binom{b}{j}
    \tilde{f}^{(j)}_{1,0}(z)\right)\left(\sum_{0\le j\le b}
    \binom{b}{j}\tilde{f}^{(j)}_{0,1}(z)\right)\\
    &\qquad +z\left(\sum_{0\le j\le b}\binom{b}{j}
    \tilde{f}^{(j+1)}_{1,0}(z)\right)
    \left(\sum_{0\le j\le b}\binom{b}{j}
    \tilde{f}^{(j+1)}_{0,1}(z)\right) \\
    &\qquad -\sum_{0\le j\le b}\binom{b}{j}
    \left(\tilde{f}_{1,0}(z)\tilde{f}_{0,1}(z)
    +z\tilde{f}'_{1,0}(z)\tilde{f}'_{0,1}(z)\right)^{(j)},\\
    \tilde{g}_{0,2}(z)
    &=4\tilde{U}(z/2)+2\tilde{V}(z/2)
    +\left(\sum_{0\le j\le b}\binom{b}{j}
    \tilde{f}^{(j)}_{0,1}(z)\right)^2\\
    &\qquad +z\left(\sum_{0\le j\le b}\binom{b}{j}
    \tilde{f}^{(j+1)}_{0,1}(z)\right)^2
    -\sum_{0\le j\le b}\binom{b}{j}\left(\tilde{f}_{0,1}(z)^2
    +z\tilde{f}'_{0,1}(z)^2\right)^{(j)}.
\end{split}\right.
\]
The initial conditions are $\tilde{V}(0) =\tilde{U}^{(j)}(0)=
\tilde{W}^{(j)}(0) = 0$ for $0\le j< b$ and
\[
    \tilde{V}^{(j)}(0)
    = (-1)^j\left(1+(j-2)2^{j-1}\right),
    \qquad(1\le j\le b).
\]
From (\ref{mean-1}), (\ref{mean-2}) and Ritt's theorem (see
\cite{olver1974a}), we have
\[
\left\{
\begin{split}
    \tilde{g}_{2,0}(z)
    &=O\left(|z|^{-1}\right),\\
    \tilde{g}_{1,1}(z)-2\tilde{V}(z/2)
    &=O\left(|z|^{-1}\right),\\
    \tilde{g}_{0,2}(z)-4\tilde{U}(z/2)-2\tilde{V}(z/2)
    &=O\left(|z|^{-1}\right),
\end{split}
\right.
\]
uniformly for $|z|\to\infty$ and $|\arg(z)|\le\pi/2-\ve$. Let
$\bar{\mathscr{L}}[\tilde{A};s] := \mathscr{L}
[\tilde{A};s]/Q(-s)^b$, where $\tilde{A}\in\{\tilde{V}, \tilde{U},
\tilde{W}\}$. Then we obtain, for $\Re(\omega)>2$,
\[
\left\{
    \begin{split}
    \mathscr{M}[\bar{\mathscr{L}}[\tilde{V};s];\omega]
    &=\frac{G_{2,0}(\omega)}{1-2^{2-\omega}},\\
    \mathscr{M}[\bar{\mathscr{L}}[\tilde{U};s];\omega]
    &=\frac{2^{2-\omega}G_{2,0}(\omega)}
    {\left(1-2^{2-\omega}\right)^2}+\frac{G_{1,1}(\omega)}
    {1-2^{2-\omega}},\\
    \mathscr{M}[\bar{\mathscr{L}}[\tilde{W};s];\omega]
    &=\frac{2^{5-2\omega}G_{2,0}(\omega)}
    {\left(1-2^{2-\omega}\right)^3}
    +\frac{2^{2-\omega}\left(2G_{1,1}(\omega)
    +G_{2,0}(\omega)\right)}{\left(1-2^{2-\omega}\right)^2}
    +\frac{G_{1,1}(\omega)+G_{0,2}(\omega)}{1-2^{2-\omega}},
\end{split}\right.
\]
where
\[
\left\{
\begin{split}
    G_{2,0}(\omega)
    &:=\int_{0}^{\infty}\frac{s^{\omega-1}}{Q(-2s)^b}
    \left(\mathscr{L}[\tilde{g}_{2,0};s]
    +\frac{(s+1)^{b-1}-(-1)^b\left(2b-3+(b-1)s\right)}
    {(s+2)^2}\right)\dd s,\\
    G_{1,1}(\omega)
    &:=\int_{0}^{\infty}\frac{s^{\omega-1}}{Q(-2s)^b}
    \int_{0}^{\infty}e^{-sz}
    \left(\tilde{g}_{1,1}(z)-2\tilde{V}(z/2)\right)
    \dd z\dd{s},\\
    G_{0,2}(\omega)
    &:=\int_{0}^{\infty}\frac{s^{\omega-1}}{Q(-2s)^b}
    \int_{0}^{\infty}e^{-sz}
    \left(\tilde{g}_{0,2}(z)-2\tilde{V}(z/2)-
    4\tilde{U}(z/2)\right)\dd z\dd{s},
\end{split}\right.
\]
with all functions analytic for $\Re(\omega)>0$. Consequently, we
deduce (\ref{VNnXn-bdst}).

\section{Digital search trees. II. More shape parameters.}
\label{DST-II}

We consider in this section four additional examples on DSTs whose
variances are essentially linear. The same tools we use readily
apply to $b$-DSTs, but we focus on DSTs because the results are
easier to state and the asymptotic behaviors do not differ in
essence with those for the more general $b$-DSTs the corresponding
expressions of which are however much messier.

The first parameter we consider is the so-called $w$-parameter (see
\cite{drmota09a}), which is the sum of the subtree-size of the
parent-node of each leaf (over all leaves)\footnote{The
\emph{leaves} or \emph{leaf-nodes} of a tree are nodes without any
descendants.}. Instead of $w$-parameter, we call it the \emph{total
peripheral path-length} (PPL), since it measures to some extent the
fringe ampleness of the trees. Also this is in consistency with the
two previous notions of path-length we distinguished.

Then we consider the number of leaves, which has previously
been studied in details in
\cite{flajolet86d,hubalek02a,kirschenhofer88d} and which is well
connected to PPL. Our expression for the variance simplifies known
ones.

Yet another notion of path-length we consider here is the so-called
\emph{Colless index} in phylogenetics, which is the sum of the
absolute difference of the two subtree-sizes of each node (over all
nodes). We call this index the \emph{total differential path-length}
(DPL) as it clearly indicates the balance or symmetry of the tree.
Another widely used measure of imbalance in phylogenetics is the
\emph{Sackin index}, which is nothing but the external path-length.

The last example we consider is the \emph{weighted path-length}
(WPL), which often arises in coding, optimization and many related
problems.

The orders of the means and the variances exhibited by all the shape
parameters we study in this paper are listed in Table~\ref{tb-all-pl}.

\subsection{Peripheral path-length (PPL)}

The PPL (or $w$-parameter) was introduced in \cite{drmota09a}, the
motivations arising from the analysis of compression algorithms. We
start from the \emph{fringe-size} of a leaf node $\lambda$, which is
defined to be the size of the subtree rooted at its parent-node; see
Figure~\ref{fg-fpl}. The PPL of a tree is then defined to be the sum
of the fringe-sizes of all leaf-nodes. Let $X_n$ denote the PPL in a
DST built from $n$ random binary strings under our usual independent
Bernoulli model.

\begin{figure}[!h]
\begin{center}
\begin{tikzpicture}[scale=0.4]
\draw[line width=0.8pt] (2,-2) -- (0,-1);%
\draw[line width=0.8pt] (-2,-2) -- (0,-1);%
\node[draw,fill,shape=circle, inner sep=1.5pt]  at (0,-1) { };%
\node[draw,fill,shape=circle, inner sep=1.5pt]  at (-2,-2) { };%
\draw[line width=0.8pt] (2,-2) -- (3.5,-4) -- (.5,-4) -- cycle;%
\draw[line width=0.8pt] (2,-3.5) node[text width=3cm,text centered]
{$T$};
\draw[line width=0.8pt] (0,.5) node[text width=3cm,text centered]
{\Large $\vdots$};
\end{tikzpicture}
\begin{tikzpicture}[scale=0.4]
\draw[line width=0.8pt] (2,-2) -- (0,-1);%
\draw[line width=0.8pt] (-2,-2) -- (0,-1);%
\node[draw,fill,shape=circle, inner sep=1.5pt]  at (0,-1) { };%
\node[draw,fill,shape=circle, inner sep=1.5pt]  at (2,-2) { };%
\draw[line width=0.8pt] (-2,-2) -- (-3.5,-4) --  (-.6,-4)  -- cycle;
\draw[line width=0.8pt] (-2.05,-3.5) node[text width=3cm,text centered]
{$T$};
\draw[line width=0.8pt] (0,.5) node[text width=3cm,text centered]
{\Large $\vdots$};
\end{tikzpicture}
\end{center}
\caption{\emph{The two possible configurations of the fringe of a
leaf: the fringe-size (or $w$-parameter) equals $|T|+2$. Note that
$T$ may be empty.}}
\label{fg-fpl}
\end{figure}
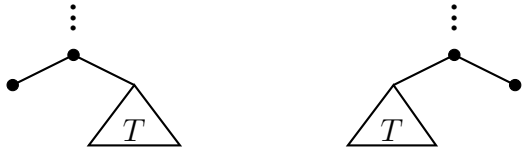

Drmota et al.\ showed in \cite{drmota09a} that
\begin{align}\label{EPPL}
    \mathbb{E}(X_n)
    = n\left(C_w + \varpi_w(\log_2n)\right) + o(n),
\end{align}
where
\begin{align*}
    C_w
    :=\sum_{\ell\ge0}\frac{(\ell+1)(\ell-2)}{Q_\ell 2^\ell}
    \left(\sum_{k\ge1} \frac1{2^{\ell+k}-1} -1 \right)
    +\frac{1}{\log2}\sum_{\ell\ge 0}\frac{2\ell-1}{Q_\ell 2^\ell}.
\end{align*}
Note that by (\ref{inv-Qz}), we have the identities
\begin{align*}
    \sum_{\ell\ge0}\frac{(\ell+1)(\ell-2)}{Q_\ell 2^\ell}
    &= \frac1{Q_\infty}\left(\sum_{j\ge1}\frac1{(2^j-1)^2}
    + \left(\sum_{j\ge1}\frac1{2^j+1}\right)^2-2\right),\\
    \sum_{\ell\ge0}\frac{2\ell-1}{Q_\ell 2^\ell}
    &= \frac1{Q_\infty}\left(\sum_{j\ge1}\frac2{2^j-1}
    -1\right).
\end{align*}
The asymptotic behavior (\ref{EPPL}) is to be compared with the
$n\log n$-order exhibited by most other log-trees such as binary
search trees and recursive trees; see \cite{drmota09a}. It reflects
that most fringes of random DSTs are small in size; see
Figure~\ref{fg-rnd-dst}. Indeed, since the expected number of leaves
is also asymptotic to $n$ times a periodic function, the result
(\ref{EPPL}) implies that the average size of a fringe in random
DSTs is bounded. We show that the standard deviation is also small.

Define
\begin{align}\label{PPL-g2z}
\begin{split}
    \tilde{g}_2(z)
    &:= z\tilde{f}_1''(z)^2 -\frac{z}{16}\,e^{-z}
    \left(z^4+4z^3+16z^2-8z+64\right) \\
    &\qquad -\frac{z}4\,e^{-z/2}\left(
    4(z+4)\tilde{f}_1(z/2)-2(z^2+2z+8)\tilde{f}_1'(z/2)
    -(z+2)(z+8)\right),
\end{split}
\end{align}
where $\tilde{f}_1(z)$ represents as usual the Poisson generating
function of $\mathbb{E}(X_n)$. Let $G_2(\omega)$ denote the Mellin
transform of $\mathscr{L}[\tilde{g}_2;s]/Q(-2s)$.
\begin{thm} The mean and the variance of the total PPL $X_n$ of
random DSTs of $n$ strings satisfy
\begin{align}
    \mathbb{E}(X_n)
    &= n\left(C_w + \varpi_w(\log_2n)\right) + O(1),\nonumber \\
    \mathbb{V}(X_n)
    &= n P_w(\log_2n) + O(1), \label{VPPL}
\end{align}
where $P_w(t)$ is a smooth, $1$-periodic function with the Fourier
series expansion
\[
    P_w(t)
    = \frac1{\log 2}\sum_{k\in\mathbb{Z}} \frac{G_2(2+\chi_k)}
    {\Gamma(2+\chi_k)}\,e^{2k\pi it},
\]
the series being absolutely convergent.
\end{thm}
We provide only the major steps of the proof since it follows the
same approach we developed above.

\paragraph{Recurrence and generating functions.}
By definition and by conditioning on the size of one of the subtrees
of the root, we have the following different configurations
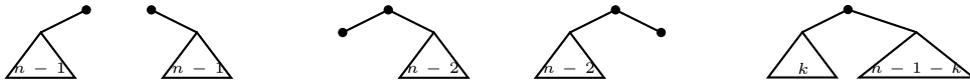
\begin{center}
\begin{tikzpicture}[scale=0.3, line width=.5pt]
\draw[line width=0.8pt] (-2,-2) -- (0,-1);%
\draw[fill] (0,-1) circle (5pt);%
\draw[line width=0.8pt] (-2,-2) --
(-3.5,-4) --  (-.5,-4)  -- cycle;%
\draw[line width=0.8pt](-2.05,-3.6) node[text width=3cm, text
centered] {\tiny $n-1$};%
\end{tikzpicture}\hspace*{-1.3cm}
\begin{tikzpicture}[scale=0.3, line width=.5pt]
\draw[line width=0.8pt] (2,-2) -- (0,-1);%
\draw[fill] (0,-1) circle (5pt);%
\draw[line width=0.8pt] (2,-2) --
(3.5,-4) --  (.5,-4)  -- cycle;%
\draw[line width=0.8pt] (2.05,-3.6) node[text width=3cm, text
centered] {\tiny $n-1$};%
\end{tikzpicture}\hspace*{-.3cm}
\begin{tikzpicture}[scale=0.3, line width=.5pt]
\draw[line width=0.8pt] (2,-2) -- (0,-1);%
\draw[line width=0.8pt] (-2,-2) -- (0,-1);%
\draw[fill] (0,-1) circle (5pt);%
\draw[fill] (-2,-2) circle (5pt);%
\draw[line width=0.8pt] (2,-2) -- (3.5,-4) --  (.5,-4)  -- cycle;%
\draw[line width=0.8pt] (1.95,-3.6) node[text width=3cm, text
centered] {\tiny $n-2$};%
\end{tikzpicture}\hspace*{-1.6cm}
\begin{tikzpicture}[scale=0.3, line width=.5pt]
\draw[line width=0.8pt] (-2,-2) -- (0,-1);%
\draw[line width=0.8pt] (2,-2) -- (0,-1);%
\draw[fill] (0,-1) circle (5pt);%
\draw[fill] (2,-2) circle (5pt);%
\draw[line width=0.8pt] (-2,-2) -- (-3.5,-4) --  (-.5,-4)  -- cycle;%
\draw[line width=0.8pt] (-2.05,-3.6) node[text width=3cm, text
centered] {\tiny $n-2$};%
\end{tikzpicture}\hspace*{-.3cm}
\begin{tikzpicture}[scale=0.3, line
width=.5pt]
\draw[line width=0.8pt] (-2,-2) -- (0,-1);%
\draw[line width=0.8pt] (3,-2) -- (0,-1);%
\draw[fill] (0,-1) circle (5pt);%
\draw[line width=0.8pt] (-2,-2) --
(-3.5,-4) --  (-.5,-4)  -- cycle;%
\draw[line width=0.8pt] (3,-2) --
(5.5,-4) --  (0.5,-4)  -- cycle;%
\draw[line width=0.8pt] (-1.95,-3.6)
node[text width=3cm, text centered] {\tiny $k$};%
\draw[line width=0.8pt] (3,-3.6) node[text width=3cm, text
centered] {\tiny $n-1-k$};%
\end{tikzpicture}
\end{center}

\noindent from which we derive the recurrence for the PPL
\[
    X_n
    \stackrel{d}{=}\begin{cases}
        X_{n-1},&\text{with probability}\ 2^{2-n};\\
        n+X_{n-2},&\text{with probability}\ (n-1)2^{2-n};\\
        X_k+X_{n-1-k}^{*},&\text{with probability}\
        2^{1-n}\binom{n-1}{k},\ 2\le k\le n-3,
    \end{cases}
\]
where $X_0=X_1=0, X_2=2$ and $X_3$ has the distribution
\[
    X_3
    =\begin{cases}
        6,&\text{with probability $1/2$};\\
        2,&\text{with probability $1/2$}.
    \end{cases}
\]

From this recurrence, it follows that the bivariate Poisson
generating function
\[
    \tilde{F}(z,y)
    :=e^{-z}\sum_{n\ge 0}\frac{\mathbb{E}(e^{X_ny})}{n!}\,z^n
\]
satisfies the nonlinear equation
\begin{align}\label{PPL-PGF}
\begin{split}
    \tilde{F}(z,y)+\frac{\partial}{\partial z}\tilde{F}(z,y)
    &=\tilde{F}\left(\frac z2,y\right)^2+ze^{2y+e^yz/2-z}
    \tilde{F}\left(\frac{e^yz}2,y\right)\\
    &\qquad -ze^{-z/2}\tilde{F}\left(\frac z2,y\right)
    + \frac{z^2}4\,e^{-z}\left(e^{3y}-1\right)^2,
\end{split}
\end{align}
with the initial condition $\tilde{F}(0,y)=1$.

\paragraph{The expected PPL.}
By (\ref{PPL-PGF}), we obtain the differential-functional equation
for $\tilde{f}_1(z)$ by taking derivative with respect to $y$ and
then substituting $y=1$, giving
\begin{align} \label{PPL-f1z}
    \tilde{f}_1(z)+\tilde{f}'_1(z)
    =2\tilde{f}_1(z/2)+z(2+z/2)e^{-z/2},
\end{align}
with $f_1(0)=0$. The Laplace transform of $\tilde{f}_1$ satisfies
\begin{align*}
    \mathscr{L}[\tilde{f}_1;s]
    &= \frac{4}{s+1}\,\mathscr{L}[\tilde{f}_1;s]
    + \frac{16}{(1+2s)^3} \\
    &= 16\sum_{k\ge0} \frac{4^k}{(s+1)\cdots(2^{k-1}s+1)
    (2^{k+1}s+1)^3}.
\end{align*}
Then a straightforward application of the
Laplace-Mellin-de-Poissonization approach yields
\[
    \mathbb{E}(X_n)
    =\frac{n}{\log 2}\sum_{k\in\mathbb{Z}}\frac{G_1(2+\chi_k)}
    {\Gamma(2+\chi_k)}\, n^{\chi_k}+O(1),
\]
where
\[
    G_1(\omega)
    :=16\int_{0}^\infty\frac{s^{\omega-1}}
    {Q(-s)(2s+1)^3}\dd s\qquad(\Re(\omega)>0).
\]
The $O(1)$-term can be further refined by the same analysis. In
particular, we get an alternative expression for $C_w$
\[
    C_w
    = \frac{G_1(2)}{\log 2}
    =\frac{16}{\log2}\int_{0}^{\infty}\frac{s}
    {Q(-s)(2s+1)^3}\dd s
    \approx1.10302\,66959\cdots.
\]
That the two expressions of $C_w$ are identical can be proved by
standard calculus of residues; see \cite{flajolet92a} for similar
details.

\paragraph{The variance of the PPL.}
Again from (\ref{PPL-PGF}), we derive the equation for the Poisson
generating function $\tilde{f}_2(z)$ of the second moment of $X_n$
\begin{align} \label{PPL-f2z}
\begin{split}
    \tilde{f}_2(z)+\tilde{f}'_2(z)
    &=2\tilde{f}_2(z/2)+2\tilde{f}_1(z/2)^2
    +\frac92\,z^2e^{-z} \\
    &\qquad +ze^{-z/2} \left((z+4)\tilde{f}_1(z/2)
    +z\tilde{f}'_1(z/2)
    +\frac{z^2+10z+16}4\right),
\end{split}
\end{align}
with $\tilde{f}_2(0)=0$.

Let $\tilde{V}(z)=\tilde{f}_2(z)-\tilde{f}_1(z)^2-z\tilde{f}'_1(z)^2$.
Then, by (\ref{PPL-f1z}), (\ref{PPL-f2z}) and Lemma~\ref{lmm-Vz-dfe},
\[
    \tilde{V}(z)+\tilde{V}'(z)
    =2\tilde{V}(z/2)+\tilde{g}_2(z),
\]
with $\tilde{V}(0)=0$, where $\tilde{g}_2$ is defined in
(\ref{PPL-g2z}).

Applying again the Laplace-Mellin-de-Poissonization approach, we
deduce (\ref{VPPL}). In particular, the mean value of the periodic
function $P_w$ is given by
\[
    \frac{G_2(2)}{\log2}
    = \frac1{\log 2}\int_0^\infty\frac{s}{Q(-2s)}
    \int_0^\infty e^{-zs}\tilde{g}_2(z)\dd z \dd s.
\]

\subsection{The number of leaves}

The leaves of a tree are the locations where the nodes holding
new-coming keys will be connected; thus different types of data
fields can be used to save memory, notably for $b$-DSTs. The number
of leaves then provides a quick and simpler look at the ``fringes''
of a tree. Such nodes are sometimes referred to as the
external-internal nodes or internal endnodes in the literature; see
\cite{drmota09a,flajolet86d,kirschenhofer91a,prodinger92a}.

Let $X_n$ denote the number of leaves in a random DST of $n$ keys.
Then $X_n$ satisfies the recurrence
\begin{align}\label{leaves-rr}
    X_{n+1}
    \stackrel{d}{=} X_{B_n}+ X_{n-B_n}^* \qquad(n\ge1),
\end{align}
with $X_0=0$ and $X_1=1$, where $B_n \sim \text{Binomial}(n;1/2)$.

Flajolet and Sedgewick \cite{flajolet86d}, solving an open question
raised by Knuth, showed that
\[
    \mathbb{E}(X_n)
    = n\left(C_{\textit{fs}}+\varpi_{\textit{fs}}(\log_2n)\right)
    +O(n^{1/2}),
\]
where $\varpi_{\textit{fs}}(t)$ is a smooth, $1$-periodic function and
\begin{align*}
    C_{\textit{fs}}
    &= 1+\sum_{k\ge1} \frac{k}{Q_k2^k}\sum_{1\le j\le k}
    \frac1{2^j-1}-\frac{1}{Q_\infty}\left(\frac{1}{\log
    2}+\left(\sum_{k\ge1}\frac1{2^k-1}\right)^2-
    \sum_{k\ge1}\frac1{2^k-1}\right) \\
    &\approx 0.37204\,86812\cdots.
\end{align*}
A finer approximation, together with the alternative (and
numerically better) expression
\[
    C_{\textit{fs}}
    = 1+\sum_{k\ge1}\frac1{2^k-1}
    -\frac1{Q_\infty} \left(\frac1{\log 2}
    + \sum_{k\ge1}\frac{(-1)^kk}{Q_k(2^k-1)2^{k(k+1)/2}}\right),
\]
was derived by Kirschenhofer and Prodinger \cite{kirschenhofer88d};
see also \cite{prodinger92a}. They proved additionally the
asymptotic linearity of the variance
\[
    \mathbb{V}(X_n)
    \sim n\left(C_{\textit{kp}}+\varpi_{\textit{kp}}(\log_2n)\right),
\]
where $\varpi_{\textit{kp}}$ is a smooth, $1$-periodic function with
mean zero and a long, complicated expression is given for the
leading constant $C_{\textit{kp}}$. We derive different forms for
these two asymptotic approximations.

Define
\begin{align} \label{leaves-g2}
    \tilde{g}_2(z)
    =z\tilde{f}''_1(z)^2+e^{-z}\left(1-e^{-z}(1+z)
    +2z\tilde{f}'_1(z/2)-4\tilde{f}_1(z/2)\right),
\end{align}
where $\tilde{f}_1(z) := e^{-z}\sum_{n\ge0}\mathbb{E}(X_n)z^n/n!$.
\begin{thm} \label{thm-mv-leaves}
The mean and the variance of the number of leaves are both
asymptotically linear with the approximations
\begin{align*}
    \mathbb{E}(X_n)
    &= \frac n{\log 2}\sum_{k\in\mathbb{Z}} \frac{G_1(2+\chi_k)}
    {\Gamma(1+\chi_k)}\,n^{\chi_k} + O(1), \\
    \mathbb{V}(X_n)
    &= \frac n{\log 2}\sum_{k\in\mathbb{Z}} \frac{G_2(2+\chi_k)}
    {\Gamma(2+\chi_k)}\,n^{\chi_k}  + O(1),
\end{align*}
where the two series are absolutely convergent with $G_1, G_2$
defined by
\begin{align*}
    G_1(\omega)
    &=\int_0^\infty\frac{s^{\omega-1}}{(s+1)Q(-2s)}\dd s,\\
    G_2(\omega)
    &=\int_0^\infty\frac{s^{\omega-1}}{Q(-2s)}
    \int_{0}^{\infty}e^{-zs}\tilde{g}_2(z)\dd z\dd s,
\end{align*}
for $\Re(\omega)>0$.
\end{thm}

We see in particular that
\begin{align}
    C_{\textit{fs}}
    &= \frac1{\log2}\int_0^\infty
    \frac{s}{(s+1)Q(-2s)}\dd s,\nonumber \\
    C_{\textit{kp}}
    &= \frac1{\log2}\int_0^\infty\frac{s}{Q(-2s)}
    \int_{0}^{\infty}e^{-zs}\tilde{g}_2(z)\dd z\dd s.\label{C-kp}
\end{align}

\paragraph{Sketch of proof.} From (\ref{leaves-rr}), we derive the
equation for the bivariate generating function $\tilde{F}(z,y) :=
e^{-z}\sum_{n\ge0} \mathbb{E}(e^{X_ny}) z^n/n!$
\[
    \tilde{F}(z,y) + \frac{\partial}{\partial z}\tilde{F}(z,y)
    = \tilde{F}\left(\frac z2,y\right)^2 + \left(e^y-1\right)e^{-z},
\]
with $\tilde{F}(0,y)=1$. Then the Poisson generating functions of
the first two moments satisfy
\begin{align}
    \tilde{f}_1(z) + \tilde{f}_1'(z)
    &= 2\tilde{f}_1(z/2) + e^{-z},\label{leaves-f1-dfe}\\
    \tilde{f}_2(z) + \tilde{f}_2'(z)
    &= 2\tilde{f}_2(z/2) + 2\tilde{f}_1(z/2)^2+e^{-z},
    \nonumber
\end{align}
with $\tilde{f}_1(0)=\tilde{f}_2(0)$. Consequently, the function
$\tilde{V}(z) := \tilde{f}_2(z) - \tilde{f}_1(z)^2
-z\tilde{f}_1'(z)^2$ satisfies
\[
    \tilde{V}(z) + \tilde{V}'(z)
    = 2\tilde{V}(z/2) + \tilde{g}_2(z),
\]
with $\tilde{V}(0)=0$, where $\tilde{g}_2$ is given in
(\ref{leaves-g2}). The remaining analysis follows the same pattern
as above and is omitted.

We provide instead some details for the numerical evaluation of the
constant $C_{\textit{kp}}$ as defined in (\ref{C-kp}), which is
similar to the case of internal path-length of DSTs.

By applying the Laplace transform to both sides of (\ref{leaves-f1-dfe})
and by iteration, we get
\[
    \mathscr{L}[\tilde{f}_1;s]
    =\sum_{k\ge 0}\frac{4^k}
    {(s+1)(2s+1)\cdots(2^{k-1}s+1)(2^ks+1)^2}.
\]
Since the inverse Laplace transform derived from the partial
fraction expansion of this series is divergent, we consider the
function $\hat{f}_1(z) := \tilde{f}_1(z)-z+z^2/2$ for which the
equation (\ref{leaves-f1-dfe}) becomes
\[
    \hat{f}_1(z)+\hat{f}'_1(z)
    = 2\hat{f}_1(z/2)-1+z+\frac{z^2}{4}+e^{-z},
\]
with $\hat{f}_1(0)=0$, and we have
\[
    \mathscr{L}[\hat{f}_1;s]
    =\frac{1}{2s^3}\sum_{k\ge 0}\frac{3\cdot 2^ks+1}{2^k(s+1)
    \cdots (2^{k-1}s+1)(2^ks+1)^2}.
\]

Then by the partial fraction expansion
\begin{align*}
    \frac{3\cdot 2^ks+1}{(s+1)\cdots(2^{k-1}s+1)(2^ks+1)^2}
    =&\sum_{0\le \ell<k}\frac{(-1)^{k-\ell}(3\cdot 2^{k-\ell}-1)
    2^{-\binom{k-\ell+1}{2}}}{(2^{k-\ell}-1)Q_\ell Q_{k-\ell}}
    \cdot\frac{1}{2^{\ell}s+1}\\
    &\ +\frac1{Q_k}\left(3+2\sum_{1\le j\le k}
    \frac{1}{2^{j}-1}\right)
    \frac{1}{2^ks+1}-\frac{2}{Q_k(2^ks+1)^2},
\end{align*}
we obtain
\[
    \mathscr{L}[\hat{f}_1;s]
    =\frac{1}{2s^3}\sum_{\ell\ge 0}\frac{1}{2^\ell Q_\ell}
    \left(\frac{\delta_\ell}{2^\ell s+1}
    -\frac{2}{(2^\ell s+1)^2}\right),
\]
where
\[
    \delta_\ell
    =3+2\sum_{1\le j\le \ell}\frac{1}{2^{j}-1}
    +\sum_{j\ge1}\frac{(-1)^{j}(3\cdot2^{j}-1)
    2^{-\binom{j+1}{2}}}{(2^{j}-1)2^jQ_{j}}.
\]
Obviously, $\lim_{\ell\rightarrow\infty}\delta_\ell=4$. Now, by
the inverse Laplace transform,
\begin{align*}
    \hat{f}_1(z)
    =\frac{1}{2}\sum_{\ell\ge0}\frac{1}{Q_\ell}\Biggl(
    &2^\ell\delta_\ell\left(1-\frac{z}{2^\ell}
    +\frac{z^2}{2^{2\ell+1}}-e^{-z/2^\ell}\right)\\
    &-2^{\ell+1}\left(3-\frac{z}{2^{\ell-1}}+\frac{z^2}{2^{2\ell+1}}
    -3e^{-z/2^\ell}\right)+2ze^{-z/2^\ell}\Biggr),
\end{align*}
which converges for all $z$; also from \cite{flajolet86d} we have
\[
    \hat{f}_1(z)
    = \sum_{n\ge3}\frac{(-1)^{n-1} z^n}{n!}
    Q(n-2)\sum_{0\le j\le n-2}\frac1{Q(j)}.
\]
Then the first and the second derivatives are given by
\begin{align*}
    \hat{f}'_1(z)
    &=\frac{1}{2}\sum_{\ell\ge0}\frac{1}{Q_\ell}
    \left(\delta_\ell\left(-1+z/2^\ell+e^{-z/2^\ell}\right)
    +4-\frac{z}{2^{\ell-1}}-4e^{-z/2^\ell}-
    \frac{z}{2^{\ell-1}}\,e^{-z/2^\ell}\right),\\
    \hat{f}''_1(z)
    &=\frac{1}{2}\sum_{\ell\ge0}\frac{1}{2^\ell Q_\ell}
    \left(\delta_\ell\left(1-e^{-z/2^\ell}\right)-2
    +2e^{-z/2^\ell}+\frac{z}{2^{\ell-1}}\,e^{-z/2^\ell}\right).
\end{align*}

Now the constant $C_{\textit{kp}}$ can be expressed in terms of the
integrals of $\hat{f}_1$ as follows.
\begin{align*}
    (\log 2)C_{kp}
    &=\int_{0}^{\infty}\frac{s}{Q(-2s)(s+1)(s+2)^2}
    \dd s+\int_{0}^{\infty}\frac{s}{Q(-2s)}
    \int_{0}^{\infty}e^{-zs}z(\hat{f}''_1(z)-1)^2\dd z\dd s\\
    &\ +2\int_{0}^{\infty}\frac{s}{Q(-2s)}
    \int_{0}^{\infty}e^{-z(s+1)}\left(z-\frac{1}{s+1}\right)
    (\hat{f}'_1(z/2)-z)\dd z\dd s.
\end{align*}
And we get $C_{kp}\approx0.034203\cdots$.

\paragraph{A general weighted sum of node-types for $b$-DSTs.}
For $b\ge2$, we can consider $X_{n}^{[j]},1\le j\le b$, the number
of leaves containing $j$ records in a random $b$-DST with bucket
capacity $b$ built from $n$ records. Let also $X_{n}^{[b+1]}$ be the
number of internal (non-leaf) nodes. Define
\[
    X_{n}
    =\sum_{1\le j\le b+1}a_jX_{n}^{[j]},
\]
where $a_1,\ldots,a_{b+1}$ are arbitrary real numbers. By a
straightforward computation
\begin{equation*}
    \sum_{0\le j\le b}\binom{b}{j}
    \frac{\partial^j}{\partial z^j}\tilde{F}(z,y)
    =e^{a_{b+1}y}\tilde{F}\left(\frac{z}{2},y\right)^2
    +e^{-z}\left(e^{a_by}-e^{a_{b+1}y}\right),
\end{equation*}
with $\tilde{F}(0,y)=1$. Then our approach can be applied and leads
to the same type of results as Theorem~\ref{thm-mv-leaves} with
different $G_1$ and $G_2$; the resulting expressions for the
variance are more explicit and simpler than those given in
\cite{hubalek02a}.

\subsection{Colless index: the differential path-length (DPL)}

The DPL of a tree is defined to be the sum over all nodes of the
absolute difference of the two subtree-sizes of each node as
depicted below.
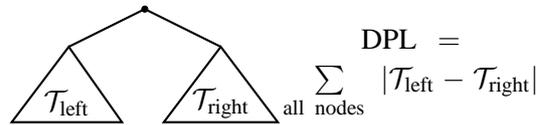
\begin{center}
\begin{tikzpicture}[scale=0.5]
\draw[line width=0.8pt] (2,-2) -- (0,-1);
\draw[line width=0.8pt] (-2,-2) -- (0,-1);
\node[draw,fill,shape=circle, inner sep=0.8pt]  at (0,-1) { };
\draw[line width=0.8pt] (-2,-2) -- (-3.5,-4) --  (-.6,-4)  -- cycle;
\draw[line width=0.8pt] (-2.05,-3.5) node[text width=3cm,text centered]
{\small $\mathcal{T}_{\text{left}}$};
\draw[line width=0.8pt] (2,-2) -- (3.5,-4) -- (.5,-4) -- cycle;
\draw[line width=0.8pt] (2,-3.5) node[text width=3cm, text centered]
{\small $\mathcal{T}_{\text{right}}$};
\draw[line width=0.8pt] (7,-2.7) node[text width=4cm, text centered]
{\small $\text{DPL}=\sum\limits_{\text{all nodes }}
|\mathcal{T}_{\text{left}} -\mathcal{T}_{\text{right}}|$};
\end{tikzpicture}
\end{center}

Properties of such a path length in random binary search trees have
long been investigated in the systematic biology literature; see
\cite{blum06a} and the references therein.

Let $X_n$ denote the DPL of a random DST of $n$ input-strings. Then
by definition and by our independence assumption, we have the
recurrence for the moment generating function
\begin{align}\label{DPL-MGF}
    M_{n+1}(y)
    = 2^{-n}\sum_{0\le j\le n} \binom{n}{j}
    M_j(y) M_{n-j}(y) e^{|n-2j|y}\qquad(n\ge0),
\end{align}
with $M_0(y)=1$.

Let also
\begin{align*}
    \tilde{g}_2(z)
    &:= z\tilde{f}_1''(z)^2 +z -\tilde{h}_1(z)^2
    - z\tilde{h}_1'(z)^2
    -4\tilde{h}_1(z)\tilde{f}(z/2)
    -2z\tilde{h}_1'(z)\tilde{f}_1'(z/2)
    +4\tilde{h}_c(z),
\end{align*}
where $\tilde{f}_1(z)$ is the Poisson generating function of
$\mathbb{E}(X_n)$ and $\tilde{h}_c(z)$ is defined by
\begin{align} \label{hcz-def}
    \tilde{h}_c(z)
    := e^{-z}\sum_{n\ge0} \frac{(z/2)^n}{n!}\sum_{0\le k\le n}
    \binom{n}{k}\mathbb{E}(X_k)|n-2k|.
\end{align}

\begin{thm} The mean and the variance of the DPL of random DSTs
satisfy the asymptotic relations
\begin{align}
    \mathbb{E}(X_n)
    &= n P_{d,\mu}(\log_2n) - \frac{\sqrt{2n}}{\sqrt{\pi}
    (\sqrt{2}-1)} +O(1),\label{Mean-DPL}\\
    \mathbb{V}(X_n)
    &= \left(1-\frac2\pi\right)n\log_2n + nP_{d,\sigma}(\log_2n)
    +O(n^{1/2}), \label{Var-DPL}
\end{align}
where $P_{d,\mu}$ and $P_{d,\sigma}$ are explicitly computable,
smooth, $1$-periodic functions.
\end{thm}
These results are to be compared with the known results for random
binary search trees for which the DPL has mean of order $n\log n$
and variance of order $n^2$; see \cite{blum06a}.

\paragraph{Expected DPL.}
The approach we follow here for deriving the differential-functional
equations satisfied by the Poisson generating functions of the first
two moments is slightly different from the one we used since the
corresponding nonlinear equation for the bivariate generating
function $F(z,y) := \sum_{n\ge0} M_n(y)z^n/n!$ is very involved as
given below.
\begin{align*}
    \frac{\partial}{\partial z}\,F(z,y)-1
    &= F\left(\frac{e^yz}{2},y\right)
    F\left(\frac{e^{-y}z}{2},y\right) \\
    &\quad +\frac1{2\pi i}\oint_{|w|=r>0}
    F\left(\frac{wz}2,y\right)
    \left(\frac{F(e^yz/2,y)-w^{-1}e^{-y}F(z/(2w),y)}
    {w-e^{-y}}\right.\\
    &\qquad\hspace*{3.6cm} \left.-\frac{F(e^{-y}z/2,y)
    -w^{-1}e^{y}F(z/(2w),y)}{w-e^{y}} \right)\dd w,
\end{align*}
with $F(0,y)=1$.

We use instead a more elementary argument. From the recurrence
(\ref{DPL-MGF}), we obtain, with $\mu_n := \mathbb{E}(X_n)$,
\[
    \mu_{n+1}
    = 2^{1-n}\sum_{0\le k\le n}\binom{n}{k}\mu_k
    + 2^{-n}\sum_{0\le k\le n}\binom{n}{k}|n-2k|
    \qquad(n\ge1),
\]
the initial condition being $\mu_0=0$. Then the Poisson generating
function of $X_n$ satisfies the equation
\[
    \tilde{f}_1(z) + \tilde{f}_1'(z)
    = 2\tilde{f}_1(z/2)+ \tilde{h}_1(z),
\]
with $\tilde{f}_1(0)=0$, where $\tilde{h}_1$ is given by
\begin{align*}
    \tilde{h}_1(z)
    &= e^{-z}\sum_{n\ge0}\frac{(z/2)^n}{n!}
    \sum_{0\le k\le n} \binom{n}{k}|n-2k|\\
    &= ze^{-z}\left(I_0(z)+I_1(z)\right),
\end{align*}
where we used the identity
\[
    \sum_{0\le k\le n}\binom{n}{k}|n-2k|
    = \frac{2n!}{\lfloor n/2\rfloor!(\lceil n/2\rceil-1)!}
    \qquad(n\ge1),
\]
and $I_\alpha(z)$ denotes the modified Bessel functions
\[
    I_\alpha(z)
    := \sum_{n\ge 0} \frac{(z/2)^{2n+\alpha}}
    {n!\Gamma(n+\alpha+1)}.
\]

It is known (see \cite{whittaker1927a}) that, as $|z|\to\infty$,
\begin{align} \label{bessel}
    I_\alpha(z)
    = \begin{cases}
        \displaystyle \frac{e^z}{\sqrt{2\pi z}}\left(
        1+O(|z|^{-1})\right),&\mbox{if }|\arg(z)|\le \pi/2-\ve,\\
        \displaystyle O\left(|z|^{-1/2}\left(
        e^{\Re(z)}+e^{-\Re(z)}\right)\right),
        &\mbox{if }|\arg(z)|\le \pi,
    \end{cases},
\end{align}
the $O$-term holding uniformly in $z$ in each case. Thus, by
(\ref{bessel}), $\tilde{h}_1 \in\JS$ and
\[
    \tilde{h}_1(z)
    = \sqrt{\frac{2z}{\pi}}\left(1+O(|z|^{-1})\right) ,
\]
for $|z|\to\infty$ in $|\arg(z)|\le \pi/2-\ve$. Also
\[
    \mathscr{L}[\tilde{h}_1;s]
    = (s+2)^{-1/2}s^{-3/2} \qquad(\Re(s)>0).
\]
Thus we can apply the same approach and deduce that
\[
    \mathbb{E}(X_n)
    = \frac{n}{\log 2}\sum_{k\in\mathbb{Z}}
    \frac{G_1(2+\chi_k)}{\Gamma(2+\chi_k)}\,
    n^{\chi_k} - \frac{\sqrt{2n}}{\sqrt{2\pi}
    (\sqrt{2}-1)} +O(1),
\]
where $G_1(\omega)$ is the Mellin transform of
$\mathscr{L}[\tilde{h}_1;s]/Q(-2s)$
\[
    G_1(\omega)
    = \int_0^\infty \frac{s^{\omega-5/2}}{Q(-2s)\sqrt{s+2}}
    \dd s\qquad(\Re(\omega)>3/2).
\]
This proves (\ref{Mean-DPL}). Numerically, the mean value of the
dominant periodic function is $G_1(2)/\log 2\approx 1.33907\,46494$.

\paragraph{The variance of DPL.} Again from (\ref{DPL-MGF}), we have
the recurrence for the second moment $s_n := \mathbb{E}(X_n^2)$
\begin{align*}
    s_{n+1}
    = 2^{-n}\sum_{0\le k\le n}\binom{n}{k}\left(
    s_k + s_{n-k}+ (n-2k)^2 + 2\mu_k\mu_{n-k}
    + 4\mu_k|n-2k|\right),
\end{align*}
for $n\ge1$ with $s_0=s_1=0$. Since
\begin{align} \label{diff-2}
    2^{-n}\sum_{0\le k\le n}\binom{n}{k}(n-2k)^2
    = n,
\end{align}
we see that the Poisson generating function of $s_n$ satisfies the
nonlinear equation
\[
    \tilde{f}_2(z) + \tilde{f}'_2(z)
    = 2\tilde{f}_2(z/2) +2\tilde{f}_1(z/2)^2
    +z + 4\tilde{h}_c(z),
\]
with $\tilde{f}_2(0)=0$, where $\tilde{h}_c(z)$ is defined in
(\ref{hcz-def}).

\begin{lmm} The function $\tilde{h}_c$ is JS-admissible and satisfies
\begin{align} \label{hcz}
    \tilde{h}_c(z)
    = \tilde{h}_1(z)\tilde{f}_1(z/2) + O(|z|^{1/2}),
\end{align}
in the sector $|\arg(z)|\le \pi/2-\ve$.
\end{lmm}
\begin{proof}
Observe first that
\begin{align*}
    h_1(z)
    &=\sum_{k\ge0} \frac{1}{k!} \left(\frac z2
    \right)^k \sum_{n\ge0} \frac{|n-k|}{n!}
    \left(\frac z2\right)^n\\
    &= 2\sum_{k\ge0}\frac{1}{k!}\left(\frac z2
    \right)^k\sum_{0\le j\le k} \frac{k-j}{j!}
    \left(\frac z2 \right)^j \\
    &= \frac{2}{2\pi i}\oint_{|w|=r}
    \frac{e^{z(w+1/w)/2}}{(w-1)^2} \dd w
    \qquad(r<1),
\end{align*}
since
\[
    \sum_{0\le j\le k} \frac{k-j}{j!}
    \left(\frac z2 \right)^j
    = [w^k]\frac{we^{zw/2}}{(w-1)^2}.
\]

On the other hand, since $f_1(z) = \sum_{n\ge0} \mu_n z^n/n!$, we
have, by the same argument,
\begin{align*}
    h_c(z) &:= e^z \tilde{h}_c(z)\\
    &= \sum_{k\ge0} \frac{\mu_k}{k!} \left(\frac z2
    \right)^k \sum_{n\ge0} \frac{|n-k|}{n!}
    \left(\frac z2\right)^n\\
    &= \sum_{k\ge0} \frac{\mu_k}{k!} \left(\frac z2
    \right)^k \left(\sum_{0\le n\le k} \frac{k-n}{n!}
    \left(\frac z2\right)^n+\sum_{n\ge k} \frac{n-k}{n!}
    \left(\frac z2\right)^n\right)\\
    &= \sum_{k\ge0} \frac{\mu_k}{k!} \left(\frac z2
    \right)^k \sum_{0\le n\le k} \frac{k-n}{n!}
    \left(\frac z2\right)^n +
    \sum_{n\ge0}\frac{1}{n!}\left(\frac z2\right)^n
    \sum_{0\le k\le n} (n-k) \frac{\mu_k}{k!} \left(\frac z2
    \right)^k\\
    &= \frac{1}{2\pi i}\oint_{|w|=r<1}
    f_1\left(\frac{z}{2w}\right) \frac{e^{wz/2}}{(w-1)^2} \dd w
    + \frac{1}{2\pi i}\oint_{|w|=r<1}f_1\left(\frac{wz}{2}\right)
    \frac{e^{z/(2w)}}{(w-1)^2} \dd w.
\end{align*}

To prove condition \textbf{(O)}, we start with changes of variables,
giving
\[
    h_c(z)
    = \frac{z}{2\pi i}\;\;\circ\hspace*{-.52cm}\oint_{|w|=|z|}
    f_1\left(\frac{w}{2}\right) \frac{e^{z^2/(2w)}}{(w-z)^2} \dd w
    + \frac{z}{2\pi i}\;\circ\hspace*{-.4cm}\oint_{|w|=|z|}
    f_1\left(\frac{w}{2}\right) \frac{e^{z^2/(2w)}}{(w-z)^2} \dd w,
\]
where the first integration circle is indented to the right to avoid
the polar singularity $w=z$, and the second to the left. By
splitting each integration contour into two parts, we obtain
\begin{align*}
    h_c(z)
    &= \frac{z}{2\pi i}\left(\supset\hspace*{-.4cm}\int
    +\quad \subset\hspace*{-.4cm}\int \right)
    f_1\left(\frac{w}{2}\right) \frac{e^{z^2/(2w)}}{(w-z)^2} \dd w
    +O\left(\ve^{-2}\int_{\ve\le|\theta|\le\pi}
    \left|f_1\left(\frac{|z|e^{i\theta}}{2}\right)\right|
    e^{|z|(\cos\theta)/2}\dd\theta\right),
\end{align*}
where the integration contour \raisebox{.06cm}{\tiny $\supset$}
\hspace*{-8.6pt}$\int$ is any path connecting the two endpoints
$|z|e^{\pm i\ve}$ and indented to the right, and
\raisebox{.07cm}{\tiny $\subset$} \hspace{-8.6pt}$\int$ denotes the
corresponding symmetric contour with respect to $w=z$ (and indented
to the left). Since $\tilde{f}_1\in\JS$, condition \textbf{(O)} for
$\tilde{h}_c(z)$ is readily checked.

For condition \textbf{(I)}, it suffices to prove (\ref{hcz}). For
that purpose, we use the representation
\begin{align*}
    \tilde{h}_c(z)
    &= \frac{e^{-z}}{2\pi i}\oint_{|w|=r<1}
    \frac{e^{z(w+1/w)/2}}{(w-1)^2}\left(
    \tilde{f}_1\left(\frac{z}{2w}\right)
    +\tilde{f}_1\left(\frac{wz}{2}\right) \right) \dd w \\
    &= \tilde{f}_1\left(\frac{z}{2}\right)
    \frac{e^{-z}}{2\pi i}\oint_{|w|=r<1}
    \frac{e^{z(w+1/w)/2}}{(w-1)^2}\dd w
    + \frac{e^{-z}}{2\pi i}\oint_{|w|=r<1}
    e^{z(w+1/w)/2} R_z(w) \dd w\\
    &= \tilde{f}_1(z/2)\tilde{h}_1(z) +
    \frac{e^{-z}}{2\pi i}\oint_{|w|=1}
    e^{z(w+1/w)/2} R_z(w) \dd w,
\end{align*}
where
\begin{align*}
    R_z(w)
    := \frac1{(w-z)^2}& \Biggl\{\left(\tilde{f}_1
    \left(\frac{z}{2w}\right)
    - \tilde{f}_1\left(\frac{z}{2}\right) -
    \tilde{f}_1'\left(\frac{z}{2}\right)\frac{z(w-1)}2\right)\\
    &+\left(\tilde{f}_1\left(\frac{wz}{2}\right)
    -\tilde{f}_1\left(\frac{z}{2}\right)+
    \tilde{f}_1'\left(\frac{z}{2}\right)\frac{z(w-1)}2\right)\Biggr\}
\end{align*}
is analytic at $w=z$. The error term of
$\tilde{h}_c(z)-\tilde{h}_1(z) \tilde{f}_1(z/2)$ can be estimated by
a similar argument as that used for checking condition \textbf{(O)}.
This completes the proof of the Lemma.
\end{proof}

The remaining analysis is now routine. Let $\tilde{V}(z) :=
\tilde{f}_2(z) -\tilde{f}_1(z)^2-z\tilde{f}_1'(z)^2$. Then
\[
    \tilde{V}(z) + \tilde{V}'(z)
    = 2\tilde{V}(z/2) + \tilde{g}_2(z),
\]
where, by Lemma~\ref{lmm-Vz-dfe},
\begin{align*}
    \tilde{g}_2(z)
    &= z-\tilde{h}_1(z)^2 +4\left(\tilde{h}_c(z)-
    \tilde{h}_1(z)\tilde{f}_1(z/2)\right)-z\tilde{h}_1'(z)^2
    -2z\tilde{h}_1'(z)\tilde{f}_1'(z/2)+z\tilde{f}_1''(z)^2\\
    &= \left(1-\frac2\pi\right)z + O(|z|^{1/2}),
\end{align*}
for $|\arg(z)|\le\pi/2-\ve$. From this and the analytic properties
of the functions involved, we deduce (\ref{Var-DPL}).

\paragraph{Remark.} The same approach can be extended to more
general differential path-length of the form $\sum _{\text{all
nodes}}|\mathcal{T}_{\text{left}} -\mathcal{T}_{\text{right}}|^m$
with $m\ge2$. Interestingly, when $m=2$, the \emph{mean is identical
to the total internal path-length} in view of (\ref{diff-2}) and the
variance is asymptotic to $4n^2$. For $m>2$, the mean and the
variance are asymptotic to
\[
    \frac{2^{m/2}\Gamma((m+1)/2)}{\sqrt{\pi}(1-2^{1-m})}
    \,n^{m/2}, \quad \frac{2^m(\Gamma(m+1/2)-
    \pi^{-1/2}\Gamma((m+1)/2)^2)}{\sqrt{\pi}( 1-2^{1-m} )}
    n^m,
\]
respectively.

\subsection{A weighted path-length (WPL)}

Weighted path-lengths of the form $W_n := \sum_{1\le j\le n} w_j
\ell_j$ appear often in applications, where $\ell_j$ denotes the
distance of the $j$-th node (arranged in an appropriate manner, say
first level-wise and then left-to-right or in their incoming order)
to the root and $w_j$ the weight attached to the $j$-th node. The
calculation of $W_n$ in the case of random DSTs can be carried out
recursively by
\[
    W_{n+1}
    \stackrel{d}{=} W_{B_n} + W_{n-B_n}^* + \sum_{2\le j\le n+1}w_j,
\]
assuming that the root is labelled $1$. We consider in this section
the case when $w_j = (\log j)^m$, $m\ge1$. From a technical point of
view, it suffices to consider the random variables
\begin{align*}
    X_{n+1}
    \stackrel{d}{=} X_{B_n} + X_{n-B_n}^* + (n+1)(\log (n+1))^m
    \qquad(n\ge0),
\end{align*}
with $X_0=0$, since the partial sum $\sum_{2\le j\le n}(\log j)^m$
is nothing but
\[
    \sum_{2\le j\le n}(\log j)^m
    = [z^n] \frac{L_{0,m}(z)}{1-z},
\]
where
\[
    L_{a,m}(z)
    := \sum_{k\ge1} n^{-a}(\log k)^m z^m\qquad(a\not=1,2,\dots),
\]
on whose analytic properties our analytic approach heavily relies.

The random variables $X_n$ represent the sole example on DSTs we
discuss in this paper with non-integral values; they also exhibit an
interesting phenomenon in that the mean is of order $n(\log
n)^{m+1}$ but the variance is asymptotic to $n$ times a periodic
function, in contrast to the orders of DPL.

\begin{thm} The mean and the variance of the weighted path-length
$X_n$ are asymptotic to
\begin{align*}
    \mathbb{E}(X_n)
    &= \frac{n(\log n)^{m+1}}{(m+1)\log 2} +
    n\sum_{1\le j\le m} c_{m,j} (\log n)^j + nP_{w,\mu}(\log_2n)
    +O\left((\log n)^{m+1}\right),\\
    \mathbb{V}(X_n)
    &= nP_{w,\sigma}(\log_2n) +O\left((\log n)^{2m+2}\right),
\end{align*}
respectively, where the $c_{m,j}$'s are constants depending on $m$,
and $P_{w,\mu}$ and $P_{w,\sigma}$ are $1$-periodic, smooth
functions.
\end{thm}
That the variance is linear is well-predicted by the deep theorem of
Schachinger derived in \cite{schachinger95b} since the second
difference of the sequence $n(\log n)^m$ is $o(n^{-1/2-\ve})$. Our
approach has the advantage of providing more precise approximations.

The new ingredient we need is incorporated in the following lemma.
\begin{lmm}[\cite{flajolet99a}] The function $L_{a,m}(z)$ can
analytically be continued into the cut-plane
$\mathbb{C}\setminus[1,\infty)$ with a sole singularity at $z=1$
near which it admits the asymptotic approximation
\[
    L_{a,m}(e^{-s})
    = \Gamma(1-a) s^{a-1}(-\log s)^{m} +O(1),
\]
the $O$-term holding uniformly for $|\arg(s)|\le\pi-\ve$.
\end{lmm}

Indeed, the tools developed in \cite{flajolet99a} can also be easily
extended to similar ``toll-functions'' such as $nH_n^m$. Details are
left for the interested readers.

\section{Conclusions and extensions}

We showed in this paper, through many shape parameters on random
DSTs that the crucial use of the normalization $\tilde{V}(z) :=
\tilde{f}_2(z) - \tilde{f}_1(z)^2-z\tilde{f}'(z)^2$ at the level of
Poisson generating function is extremely helpful in simplifying the
asymptotic analysis of the variance as well as the resulting
expressions. The same idea can be applied to a large number of
concrete problems with a binomial splitting procedure. These and
some related topics and extensions will be pursued elsewhere. We
briefly mention in this final section some extensions and related
properties.

\paragraph{Central limit theorems.}

All shape parameters we considered in this paper are asymptotically
normally distributed in the sense of convergence in distribution. We
describe the results in this section and merely indicate the methods
of proofs. The only case that requires a separate study is NPL of
random $b$-DSTs with $b\ge2$ (a bivariate consideration of the limit
laws is needed), details being given in a future paper.

\begin{thm} The internal path-length, the peripheral path-length, the
number of leaves, the differential path-length, the weighted
path-length of random DSTs, and the key-wise path-length of random
$b$-DSTs with $b\ge2$ are all asymptotically normally distributed
\[
    \frac{X_n - \mathbb{E}(X_n)}{\sqrt{\mathbb{V}(X_n)}}
    \stackrel{d}{\longrightarrow} \mathscr{N}(0,1),
\]
where $X_n$ denotes any of these shape parameters,
$\stackrel{d}{\longrightarrow}$ stands for convergence in
distribution, and $\mathscr{N}(0,1)$ is a standard normal
distribution with zero mean and unit variance.
\end{thm}
See Figure~\ref{fig-DPL-histo} for a plot of the histograms of DPL.

The method of moments applies to all these cases and establishes the
central limit theorems; similar details are given as in
\cite{hubalek02a} (the asymptotic normality of the number of leaves
being already proved there as a special case).

In a parallel way, contraction method also works well for all these
shape parameters; see \cite{neininger01a,neininger04a,neininger06a}.

On the other hand, Schachinger's asymptotic normality results cover
the IPL, PPL, number of leaves and WPL, but not PPL and KPL on
$b$-DSTs, although his approach may be modified for that purpose.

Finally, the complex-analytic approach used in \cite{jacquet95a} for
internal path-length may be extended to prove some of these cases,
but the proofs are messy, although the results established are often
stronger (for example, with convergence rate).

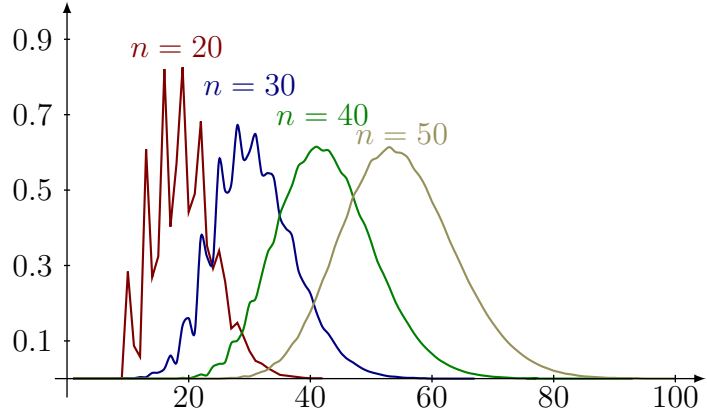
\begin{figure}[!h]
\begin{center}
\begin{tikzpicture}[xscale=.08,yscale=5]
\draw[color=red!50!black, line width=.8pt] plot coordinates{(1,
0.000) (2,  0.000) (3,  0.000) (4,  0.000) (5, 0.000) (6,  0.000)
(7,  0.000) (8,  0.000) (9,  0.000) (10, 0.284) (11, 0.087) (12,
0.058) (13, 0.608) (14, 0.269) (15, 0.324) (16, 0.821) (17, 0.403)
(18, 0.571) (19, 0.826) (20, 0.442) (21, 0.490) (22, 0.682) (23,
0.353) (24, 0.292) (25, 0.339) (26, 0.257) (27, 0.133) (28, 0.148)
(29, 0.112) (30, 0.073) (31, 0.047) (32, 0.040) (33, 0.028) (34,
0.019) (35, 0.010) (36, 0.008) (37, 0.006) (38, 0.004) (39, 0.002)
(40, 0.001) (41, 0.001) (42, 0.001)};%
\draw[color=blue!50!black, line width=.8pt] plot[smooth]
coordinates{(1, 0.000)(2, 0.000)(3, 0.000)(4, 0.000)(5, 0.000)(6,
0.000)(7, 0.000)(8, 0.000)(9,0.001)(10,0.001)(11,0.001)(12,0.004)
(13,0.004)(14,0.015)(15,0.016)(16,0.026)(17,0.061)(18,0.042)
(19,0.141)(20,0.160)(21,0.123)(22,0.375)(23,0.324)(24,0.314)
(25,0.580)(26,0.498)(27,0.515)(28,0.672)(29,0.582)(30,0.602)
(31,0.648)(32,0.544)(33,0.545)(34,0.532)(35,0.429)(36,0.398)
(37,0.378)(38,0.289)(39,0.252)(40,0.226)(41,0.176)(42,0.137)
(43,0.118)(44,0.093)(45,0.069)(46,0.053)(47,0.043)(48,0.032)
(49,0.023)(50,0.017)(51,0.013)(52,0.009)(53,0.006)(54,0.005)
(55,0.003)(56,0.002)(57,0.002)(58,0.001)(59,0.001)(60,0.000)
(67,0.000)}; \draw[color=green!50!black, line width=.8pt]
plot[smooth] coordinates{(1, 0.000)(2, 0.000)(3, 0.000)(4, 0.000)(5,
0.000)(6, 0.000)(7, 0.000)(8,
0.000)(9,0.000)(10,0.000)(11,0.000)(12,0.000)
(13,0.000)(14,0.000)(15,0.000)(16,0.000)(17,0.000)(18,0.000)
(19,0.002)(20,0.001)(21,0.005)(22,0.011)(23,0.009)(24,0.032)
(25,0.039)(26,0.044)(27,0.094)(28,0.101)(29,0.131)(30,0.199)
(31,0.212)(32,0.274)(33,0.335)(34,0.364)(35,0.435)(36,0.476)
(37,0.513)(38,0.563)(39,0.577)(40,0.603)(41,0.615)(42,0.605)
(43,0.605)(44,0.580)(45,0.548)(46,0.529)(47,0.484)(48,0.437)
(49,0.406)(50,0.362)(51,0.312)(52,0.278)(53,0.242)(54,0.203)
(55,0.173)(56,0.147)(57,0.121)(58,0.099)(59,0.081)(60,0.066)
(61,0.052)(62,0.042)(63,0.033)(64,0.026)(65,0.020)(66,0.015)
(67,0.012)(68,0.009)(69,0.007)(70,0.005)(71,0.004)(72,0.003)
(73,0.002)(74,0.001)(75,0.001)(76,0.001)(77,0.001)(78,0.000)
(93,0.000)}; \draw[color=yellow!50!black, line width=.8pt]
plot[smooth]
coordinates{(1,0.000)(2,0.000)(3,0.000)(4,0.000)(5,0.000)
(6,0.000)(7,0.000)(8,0.000)(9,0.000)(10,0.000)(11,0.000)
(12,0.000)(13,0.000)(14,0.000)(15,0.000)(16,0.000)(17,0.000)
(18,0.000)(19,0.000)(20,0.000)(21,0.000)(22,0.000)(23,0.000)
(24,0.000)(25,0.000)(26,0.001)(27,0.002)(28,0.002)(29,0.006)
(30,0.006)(31,0.012)(32,0.018)(33,0.024)(34,0.039)(35,0.048)
(36,0.067)(37,0.091)(38,0.111)(39,0.146)(40,0.177)(41,0.216)
(42,0.260)(43,0.297)(44,0.349)(45,0.391)(46,0.429)(47,0.481)
(48,0.510)(49,0.542)(50,0.577)(51,0.586)(52,0.602)(53,0.614)
(54,0.601)(55,0.598)(56,0.586)(57,0.557)(58,0.536)(59,0.507)
(60,0.469)(61,0.438)(62,0.402)(63,0.362)(64,0.328)(65,0.293)
(66,0.258)(67,0.228)(68,0.199)(69,0.171)(70,0.147)(71,0.125)
(72,0.105)(73,0.088)(74,0.074)(75,0.061)(76,0.050)(77,0.041)
(78,0.033)(79,0.027)(80,0.021)(81,0.017)(82,0.013)(83,0.011)
(84,0.008)(85,0.006)(86,0.005)(87,0.004)(88,0.003)(89,0.002)
(90,0.002)(91,0.001)(92,0.001)(93,0.001)(94,0.001)(95,0.000)
(100,0.000)}; \foreach \x/\xtext in {20,40,...,100}
   \draw[shift={(\x,0)}, line width=.5pt]
   (0,-.01) -- (0,.01) node[below] {$\xtext$};%
\foreach \y/\ytext in {0.1,0.3,...,0.9}
   \draw[shift={(0,\y)}, line width=.5pt]
   (.5,0) -- (-.5,0) node[left]{$\ytext$};%
\draw[-latex, line width=.5pt] (-2,0) -- (105,0) node[right] {};%
\draw[-latex, line width=.5pt] (0,-.05) -- (0,1) node[above] {};%
\draw (18,.88)node[color=red!50!black]  {$n=20$};%
\draw (30,.78)node[color=blue!50!black]  {$n=30$};%
\draw (42,.70)node[color=green!50!black]  {$n=40$};%
\draw (55,.65)node[color=yellow!50!black]  {$n=50$};%
\end{tikzpicture}
\end{center}
\caption{\emph{The histograms of DPL for $n=20, 30,
40$ and $50$, normalized by their standard deviations.}}
\label{fig-DPL-histo}
\end{figure}

\paragraph{The depth.} The asymptotic analysis we used in this paper
can also be extended to the depth (the distance between a randomly
chosen internal node and the root) although it is of logarithmic
order. Let $X_n$ denote the depth of a random DST of $n$ nodes. The
starting point is to consider the expected profile polynomial
\[
    P_n(y)
    := \sum_{0\le k<n} n\mathbb{P}(X_n=k) y^k,
\]
where $n\mathbb{P}(X_n=k)$ is nothing but the expected number of
internal nodes at distance $k$ to the root. Then we have the
recurrence
\[
    P_{n+1}(y)
    = 1+y2^{-n}\sum_{0\le k\le n} \binom{n}{k} \left(P_k(y)
    +P_{n-k}(y)\right) \qquad(n\ge0),
\]
with $P_0(y)=0$. From this relation, we obtain the equation for the
Poisson generating function $\tilde{F}(z,y)$ of $P_n(y)$
\[
   \tilde{F}(z,y) + \frac{\partial}{\partial z} \tilde{F}(z,y)
   = 2y \tilde{F}\left(\frac z2,y\right) + 1,
\]
with $\tilde{F}(0,y)=0$.  It follows, by taking coefficients of
$z^n$ on both sides and by solving the resulting recurrence, that
\[
    P_n(y)
    = \sum_{1\le k\le n} \binom{n}{k}(-1)^{k-1}
    \prod_{0\le j\le k-2} \left(1-\frac{y}{2^j}\right)
    \qquad(n\ge1);
\]
see \cite[p.\ 504]{knuth1998a} for a different proof. Asymptotic
approximation to $P_n(y)$ can then be obtained by Rice's integral
formula
\[
    P_n(y)
    = n-\frac{y-1}{2\pi i} \int_{(3/2)}
    \frac{\Gamma(n+1)\Gamma(-s) Q(y)}
    {\Gamma(n+1-s)(1-2^{1-s}y) Q(2^{1-s}y)} \dd s,
\]
for $|y-1|\le\ve$. More precisely, if $t\in\mathbb{C}$ lies in a
small neighborhood of the origin, then
\begin{align}
    \mathbb{E}(e^{X_nt})
    &= \frac{P_n(e^t)}{n} \nonumber\\
    &= \frac{(e^t-1)Q(e^t)}{Q(1)\log 2}\sum_{k\in\mathbb{Z}}
    \Gamma\left(-1-\frac{t}{\log 2}-\chi_k\right)
    n^{t/\log_2+\chi_k}\left(1+O\left(n^{-1}\right)\right)
    +O(n^{-1}),\label{depth-ae}
\end{align}
uniformly for $|t|\le\ve$. Alternatively, one can also apply the
Laplace-Mellin-de-Poissonization approach and obtain the same type
of result for not only DSTs but also for more general $b$-DSTs. See
\cite{louchard95a,louchard99a} for a more general and detailed
treatment (by a different approach).

The estimate (\ref{depth-ae}) leads to effective asymptotic
estimates for all moments of $X_n - \log_2n$ by standard arguments;
see \cite{hwang98d}. In particular, we obtain
\begin{align*}
    \mathbb{E}(X_n)
    &= \log_2n + \frac{\gamma-1}{\log 2} + \frac12-\sum_{k\ge1}
    \frac1{2^k-1}+\varpi_1(\log_2n) +O\left(n^{-1}\log n\right),\\
    \mathbb{V}(X_n)
    &= \frac1{12}+\frac1{(\log 2)^2} \left(1+\frac{\pi^2}{6}\right)-
    \sum_{k\ge1} \frac{2^k}{(2^k-1)^2}+\varpi_5(\log_2n) +
    O\left(n^{-1}\log^2n\right),
\end{align*}
where the estimate for the mean is exactly (\ref{mean-asymp}) with
$\varpi_1$ given in (\ref{F1u}) and $\varpi_5$ is a smooth periodic
function.

\paragraph{An analytic extension.}
From a purely analytic viewpoint, the underlying
differential-functional equation (\ref{dfe-DST}) for the moments can
be extended to an equation of the form
\[
    \sum_{0\le j\le b} \binom{b}{j} \tilde{f}^{(j)}(z)
    = \alpha\tilde{f}\left(\frac{z}{\beta}\right) +\tilde{g}(z)
    \qquad(\alpha>0;\beta>1),
\]
for which our approach still applies, leading to the functional
equation for the Laplace transform
\[
    (s+1)^b\mathscr{L}[\tilde{f};s]
    = \alpha \beta \mathscr{L}[\tilde{f};\beta s]
    + \mathscr{L}[\tilde{g};s].
\]
The natural normalizing function is then provided by
\[
    Q_\beta(-s) := \prod_{j\ge1} \left(1+\frac{s}{\beta^j}\right)^b,
\]
and the corresponding Laplace-Mellin asymptotic analysis is similar.

In particular, the case when $\alpha=\beta=m$ corresponds to a
straightforward extension of binary DSTs to $m$-ary DSTS (and the
binary unbiased Bernoulli random variable to the uniform
distribution over $\{0,1,\dots,m-1\}$). The stochastic behaviors of
all shape parameters on such trees follow the same patterns as
showed in this paper.

Yet another concrete instance arises in the so-called Eden model
studied by Dean and Majumdar \cite{dean06a}, which corresponds to
$\alpha=m$ and $\beta>1$. The model is constructed in the following
way. We start at time $t=0$ at which we have an empty node. Then at
time $t=T$, where $T\sim \text{Exponential}(1)$, we fill the empty
node and  attach to it $m$ different empty nodes. The process then
continues independently for each empty node by the following
recursive rule. Once an empty node of depth $j$ is attached to a
tree at time $t=t'$, it is then filled at time point $t'+T$, where
$T\sim E(\beta^j)$, and $m$ new empty nodes are attached to it.

The mean and the variance of the number of filled nodes at a large
time of such trees are studied in details in \cite{dean06a}. Since
the model is continuous, there is no need to de-Poissonize to derive
the asymptotics of the coefficient; as a consequence, no correction
term as we used in this paper is required for the asymptotics of the
variance.

\paragraph{Other DST-type recurrences.}
While the technique of Poissonized variance with correction remains
useful for  the natural case when the Bernoulli random variable is
no longer symmetric, the Laplace-Mellin approach does not apply
directly. Other asymptotic ingredients are needed such as a direct
manipulation of the Mellin transforms; see \cite{louchard99a} and the
references therein.

DST-type structures and recurrences also arise in other statistical
physical models such as the diffusion-limited aggregation; see
\cite{aldous88a,bradley85a}.

\section*{Acknowledgement}
We thank the referee for opportune helpful comments and the more
precise title.

\section*{Appendix. An Elementary Approach to the Asymptotic
Linearity of the Variance.}

We describe briefly here a direct elementary approach
to the variance of random variables satisfying the recurrence
\[
    X_{n+1}
    \stackrel{d}{=} X_{B_n} + X_{n-B_n}^* + T_n,
\]
where
\[
    \pi_{n,k}
    := \mathbb{P}(B_n=k)
    = \binom{n}{k} 2^{-n}\qquad(0\le k\le n).
\]
The starting point is to consider the recurrence satisfied by the
variance $v_n := \mathbb{V}(X_n)$
\[
    v_{n+1}
    = \sum_{0\le k\le n}\pi_{n,k}\left(v_k + v_{n-k}\right) +u_n
    + \mathbb{V}(T_n),
\]
where $\mu_k := \mathbb{E}(X_n)$ and
\[
    u_n
    := \sum_{0\le k\le n} \pi_{n,k}
    \left(\mu_k + \mu_{n-k}-\mu_{n+1}
    +\mathbb{E}(T_n)\right)^2.
\]
In most cases, we have the estimate $\mu_k = \tilde{f}_1(k)
+O(k^\ve)$. This, together with the Gaussian approximation of the
binomial distribution, implies that
\begin{align*}
    u_n
    &\approx  \sum_{\substack{|k-n/2|=o(n^{2/3})\\
    k=n/2+x\sqrt{n}/2}} \pi_{n,k}\left(\tilde{f}_1\left(
    \frac n2+\frac{x}2\sqrt{n}\right)+\tilde{f}_1\left(
    \frac n2-\frac{x}2\sqrt{n}\right)- \tilde{f}_1(n+1)
    + \mathbb{E}(T_n)\right)^2\\
    &\approx \sum_{\substack{|k-n/2|=o(n^{2/3})\\
    k=n/2+x\sqrt{n}/2}} \pi_{n,k}\left(2\tilde{f}_1\left(
    \frac n2\right)- \tilde{f}_1(n)-\tilde{f}_1'(n)
    + \mathbb{E}(T_n)\right)^2\\
    &\approx \left(2\tilde{f}_1\left(
    \frac n2\right)- \tilde{f}_1(n)-\tilde{f}_1'(n)
    + \mathbb{E}(T_n)\right)^2.
\end{align*}
But then (see (\ref{dfe-DST}) below)
\[
    2\tilde{f}_1\left(\frac n2\right)- \tilde{f}_1(n)
    -\tilde{f}_1'(n)+ \mathbb{E}(T_n)
    = \mathbb{E}(T_n) - \tilde{h}_1(n) ,
\]
where
\[
    \tilde{h}_1(z)
    := e^{-z} \sum_{j\ge0}\frac{\mathbb{E}(T_j)}{j!}\,z^j.
\]
The order of the difference $\mathbb{E}(T_n)- \tilde{h}_1(n)\approx
n|\tilde{h}_1''(n)|$ are expected to be small, roughly $O(n^\ve)$ in
all cases we consider here. Consequently, the variance is
asymptotically linear; see \cite{hubalek02a,schachinger95b} for more
precise details.

We see clearly that the smallness of the variance results naturally
from the high concentration of the binomial distribution near its
mean.


\begin{thebibliography}{10}

\bibitem{aldous88a}
D.~Aldous and P.~Shields.
\newblock A diffusion limit for a class of randomly-growing binary trees.
\newblock {\em Probab. Theory Related Fields}, 79(4):509--542, 1988.

\bibitem{bai01a}
Z.-D. Bai, H.-K. Hwang, W.-Q. Liang, and T.-H. Tsai.
\newblock Limit theorems for the number of maxima in random samples from planar
  regions.
\newblock {\em Electron. J. Probab.}, 6:no. 3, 41 pp. (electronic), 2001.

\bibitem{berndt1985a}
B.~C. Berndt.
\newblock {\em Ramanujan's notebooks. Part I}.
\newblock Springer-Verlag, New York, 1985.
\newblock With a foreword by S. Chandrasekhar.

\bibitem{blum06a}
M.~G.~B. Blum, O.~Fran{\c{c}}ois, and S.~Janson.
\newblock The mean, variance and limiting distribution of two statistics
  sensitive to phylogenetic tree balance.
\newblock {\em Ann. Appl. Probab.}, 16(4):2195--2214, 2006.

\bibitem{bradley85a}
R.~M. Bradley and P.~N. Strenski.
\newblock Directed aggregation on the bethe lattice: Scaling, mappings, and
  universality.
\newblock {\em Phys. Rev. B}, 31(7):4319--4328, Apr 1985.

\bibitem{chen03b}
W.-M. Chen and H.-K. Hwang.
\newblock Analysis in distribution of two randomized algorithms for finding the
  maximum in a broadcast communication model.
\newblock {\em J. Algorithms}, 46(2):140--177, 2003.

\bibitem{chern07a}
H.-H. Chern, M.~Fuchs, and H.-K. Hwang.
\newblock Phase changes in random point quadtrees.
\newblock {\em ACM Trans. Algorithms}, 3(2):Art. 12, 51, 2007.

\bibitem{chern02a}
H.-H. Chern, H.-K. Hwang, and T.-H. Tsai.
\newblock An asymptotic theory for {C}auchy-{E}uler differential equations with
  applications to the analysis of algorithms.
\newblock {\em J. Algorithms}, 44(1):177--225, 2002.
\newblock Analysis of algorithms.

\bibitem{coffman70a}
E.~G. Coffman, Jr. and J.~Eve.
\newblock File structures using hashing functions.
\newblock {\em Commun. ACM}, 13(7):427--432, 1970.

\bibitem{dean06a}
D.~S. Dean and S.~N. Majumdar.
\newblock Phase transition in a generalized {E}den growth model on a tree.
\newblock {\em J. Stat. Phys.}, 124(6):1351--1376, 2006.

\bibitem{dennert07a}
F.~Dennert and R.~Gr{\"u}bel.
\newblock Renewals for exponentially increasing lifetimes, with an application
  to digital search trees.
\newblock {\em Ann. Appl. Probab.}, 17(2):676--687, 2007.

\bibitem{devroye92a}
L.~Devroye.
\newblock A study of trie-like structures under the density model.
\newblock {\em Ann. Appl. Probab.}, 2(2):402--434, 1992.

\bibitem{devroye99a}
L.~Devroye.
\newblock Universal limit laws for depths in random trees.
\newblock {\em SIAM J. Comput.}, 28(2):409--432 (electronic), 1999.

\bibitem{drmota02a}
M.~Drmota.
\newblock The variance of the height of digital search trees.
\newblock {\em Acta Inform.}, 38(4):261--276, 2002.

\bibitem{drmota2009a}
M.~Drmota.
\newblock {\em Random trees}.
\newblock SpringerWienNewYork, Vienna, 2009.
\newblock An interplay between combinatorics and probability.

\bibitem{drmota09a}
M.~Drmota, B.~Gittenberger, A.~Panholzer, H.~Prodinger, and M.~D.
Ward.
\newblock On the shape of the fringe of various types of random trees.
\newblock {\em Math. Methods Appl. Sci.}, 32(10):1207--1245, 2009.

\bibitem{drmota09b}
M.~Drmota and W.~Szpankowski.
\newblock ({U}n)expected behavior of digital search tree profile.
\newblock In {\em SODA}, pages 130--138, 2009.

\bibitem{erdelyi1953a}
A.~Erd{\'e}lyi, W.~Magnus, F.~Oberhettinger, and F.~Tricomi.
\newblock {\em Higher transcendental functions. Vol. I}.
\newblock McGraw-Hill(New York), 1953.

\bibitem{fayolle86a}
G.~Fayolle, P.~Flajolet, and M.~Hofri.
\newblock On a functional equation arising in the analysis of a protocol for a
  multi-access broadcast channel.
\newblock {\em Adv. in Appl. Probab.}, 18(2):441--472, 1986.

\bibitem{fisher18a}
R.~A. Fisher.
\newblock The correlation between relatives on the supposition of mendelian
  inheritance.
\newblock {\em Philosophical Transactions of the Royal Society of Edinburgh},
  52:399--433, 1918.

\bibitem{flajolet99a}
P.~Flajolet.
\newblock Singularity analysis and asymptotics of {B}ernoulli sums.
\newblock {\em Theoret. Comput. Sci.}, 215(1-2):371--381, 1999.

\bibitem{flajolet95a}
P.~Flajolet, X.~Gourdon, and P.~Dumas.
\newblock Mellin transforms and asymptotics: harmonic sums.
\newblock {\em Theoret. Comput. Sci.}, 144(1-2):3--58, 1995.
\newblock Special volume on mathematical analysis of algorithms.

\bibitem{flajolet90a}
P.~Flajolet and A.~Odlyzko.
\newblock Singularity analysis of generating functions.
\newblock {\em SIAM J. Discrete Math.}, 3(2):216--240, 1990.

\bibitem{flajolet92a}
P.~Flajolet and B.~Richmond.
\newblock Generalized digital trees and their difference-differential
  equations.
\newblock {\em Random Structures Algorithms}, 3(3):305--320, 1992.

\bibitem{flajolet86e}
P.~Flajolet and N.~Saheb.
\newblock The complexity of generating an exponentially distributed variate.
\newblock {\em J. Algorithms}, 7(4):463--488, 1986.

\bibitem{flajolet86d}
P.~Flajolet and R.~Sedgewick.
\newblock Digital search trees revisited.
\newblock {\em SIAM J. Comput.}, 15(3):748--767, 1986.

\bibitem{flajolet95b}
P.~Flajolet and R.~Sedgewick.
\newblock Mellin transforms and asymptotics: finite differences and {R}ice's
  integrals.
\newblock {\em Theoret. Comput. Sci.}, 144(1-2):101--124, 1995.
\newblock Special volume on mathematical analysis of algorithms.

\bibitem{flajolet2009a}
P.~Flajolet and R.~Sedgewick.
\newblock {\em Analytic combinatorics}.
\newblock Cambridge University Press, Cambridge, 2009.

\bibitem{hald2002a}
A.~Hald.
\newblock {\em On the history of series expansions of frequency functions and
  sampling distributions, 1873--1944}.
\newblock Matematisk-Fysiske Meddelelser. 49. Copenhagen: The Royal Danish
  Academy of Sciences and Letters., 2002.

\bibitem{hubalek00a}
F.~Hubalek.
\newblock On the variance of the internal path length of generalized digital
  trees -- the {M}ellin convolution approach.
\newblock {\em Theoret. Comput. Sci.}, 242(1-2):143--168, 2000.

\bibitem{hubalek02a}
F.~Hubalek, H.-K. Hwang, W.~Lew, H.~Mahmoud, and H.~Prodinger.
\newblock A multivariate view of random bucket digital search trees.
\newblock {\em J. Algorithms}, 44(1):121--158, 2002.
\newblock Analysis of algorithms.

\bibitem{hwang98d}
H.-K. Hwang.
\newblock On convergence rates in the central limit theorems for combinatorial
  structures.
\newblock {\em European J. Combin.}, 19(3):329--343, 1998.

\bibitem{jacquet90a}
P.~Jacquet and E.~Merle.
\newblock Analysis of a stack algorithm for csma-cd random length packet
  communication.
\newblock {\em IEEE Transactions on Information Theory}, 36(2):420--426, 1990.

\bibitem{jacquet88a}
P.~Jacquet and M.~R{\'e}gnier.
\newblock Normal limiting distribution of the size of tries.
\newblock In {\em Performance'87 ({B}russels, 1987)}, pages 209--223.
  North-Holland, Amsterdam, 1988.

\bibitem{jacquet95a}
P.~Jacquet and W.~Szpankowski.
\newblock Asymptotic behavior of the {L}empel-{Z}iv parsing scheme and [in]
  digital search trees.
\newblock {\em Theoret. Comput. Sci.}, 144(1-2):161--197, 1995.
\newblock Special volume on mathematical analysis of algorithms.

\bibitem{jacquet98a}
P.~Jacquet and W.~Szpankowski.
\newblock Analytical de-{P}oissonization and its applications.
\newblock {\em Theoret. Comput. Sci.}, 201(1-2):1--62, 1998.

\bibitem{jacquet01a}
P.~Jacquet, W.~Szpankowski, and J.~Tang.
\newblock Average profile of the {L}empel-{Z}iv parsing scheme for a
  {M}arkovian source.
\newblock {\em Algorithmica}, 31(3):318--360, 2001.
\newblock Mathematical analysis of algorithms.

\bibitem{janson06a}
S.~Janson.
\newblock Rounding of continuous random variables and oscillatory asymptotics.
\newblock {\em Ann. Probab.}, 34(5):1807--1826, 2006.

\bibitem{kirschenhofer88d}
P.~Kirschenhofer and H.~Prodinger.
\newblock Eine {A}nwendung der {T}heorie der {M}odulfunktionen in der
  {I}nformatik.
\newblock {\em \"Osterreich. Akad. Wiss. Math.-Natur. Kl. Sitzungsber. II},
  197(4-7):339--366, 1988.

\bibitem{kirschenhofer88c}
P.~Kirschenhofer and H.~Prodinger.
\newblock Further results on digital search trees.
\newblock {\em Theoret. Comput. Sci.}, 58(1-3):143--154, 1988.
\newblock Thirteenth International Colloquium on Automata, Languages and
  Programming (Rennes, 1986).

\bibitem{kirschenhofer91a}
P.~Kirschenhofer and H.~Prodinger.
\newblock On some applications of formulae of {R}amanujan in the analysis of
  algorithms.
\newblock {\em Mathematika}, 38(1):14--33, 1991.

\bibitem{kirschenhofer94a}
P.~Kirschenhofer, H.~Prodinger, and W.~Szpankowski.
\newblock Digital search trees again revisited: the internal path length
  perspective.
\newblock {\em SIAM J. Comput.}, 23(3):598--616, 1994.

\bibitem{knessl00b}
C.~Knessl and W.~Szpankowski.
\newblock Asymptotic behavior of the height in a digital search tree and the
  longest phrase of the {L}empel-{Z}iv scheme.
\newblock {\em SIAM J. Comput.}, 30(3):923--964 (electronic), 2000.

\bibitem{knuth1998a}
D.~E. Knuth.
\newblock {\em The art of computer programming. {V}olume 3: Sorting and
  searching}.
\newblock Addison-Wesley Publishing Co., Reading, Mass., second edition, 1998.

\bibitem{konheim73a}
A.~G. Konheim and D.~J. Newman.
\newblock A note on growing binary trees.
\newblock {\em Discrete Math.}, 4:57--63, 1973.

\bibitem{louchard87a}
G.~Louchard.
\newblock Exact and asymptotic distributions in digital and binary search
  trees.
\newblock {\em RAIRO Inform. Th\'eor. Appl.}, 21(4):479--495, 1987.

\bibitem{louchard94a}
G.~Louchard.
\newblock Digital search trees revisited.
\newblock {\em Cahiers Centre \'Etudes Rech. Op\'er.}, 36:259--278, 1994.
\newblock Hommage {\`a} Simone Huyberechts.

\bibitem{louchard95a}
G.~Louchard and W.~Szpankowski.
\newblock Average profile and limiting distribution for a phrase size in the
  {L}empel-{Z}iv parsing algorithm.
\newblock {\em IEEE Trans. Inform. Theory}, 41(2):478--488, 1995.

\bibitem{louchard99a}
G.~Louchard, W.~Szpankowski, and J.~Tang.
\newblock Average profile of the generalized digital search tree and the
  generalized {L}empel-{Z}iv algorithm.
\newblock {\em SIAM J. Comput.}, 28(3):904--934 (electronic), 1999.

\bibitem{mahmoud1992a}
H.~M. Mahmoud.
\newblock {\em Evolution of random search trees}.
\newblock Wiley-Interscience Series in Discrete Mathematics and Optimization.
  John Wiley \& Sons Inc., New York, 1992.
\newblock A Wiley-Interscience Publication.

\bibitem{neininger01a}
R.~Neininger.
\newblock On a multivariate contraction method for random recursive structures
  with applications to {Q}uicksort.
\newblock {\em Random Structures Algorithms}, 19(3-4):498--524, 2001.
\newblock Analysis of algorithms (Krynica Morska, 2000).

\bibitem{neininger04a}
R.~Neininger and L.~R{\"u}schendorf.
\newblock A general limit theorem for recursive algorithms and combinatorial
  structures.
\newblock {\em Ann. Appl. Probab.}, 14(1):378--418, 2004.

\bibitem{neininger06a}
R.~Neininger and L.~R{\"u}schendorf.
\newblock A survey of multivariate aspects of the contraction method.
\newblock {\em Discrete Math. Theor. Comput. Sci.}, 8(1):31--56 (electronic),
  2006.

\bibitem{olver1974a}
F.~W.~J. Olver.
\newblock {\em Asymptotics and special functions}.
\newblock Academic Press, 1974.
\newblock Computer Science and Applied Mathematics.

\bibitem{pittel86a}
B.~Pittel.
\newblock Paths in a random digital tree: limiting distributions.
\newblock {\em Adv. in Appl. Probab.}, 18(1):139--155, 1986.

\bibitem{prodinger92a}
H.~Prodinger.
\newblock External internal nodes in digital search trees via {M}ellin
  transforms.
\newblock {\em SIAM J. Comput.}, 21(6):1180--1183, 1992.

\bibitem{prodinger92b}
H.~Prodinger.
\newblock Hypothetical analyses: approximate counting in the style of {K}nuth,
  path length in the style of {F}lajolet.
\newblock {\em Theoret. Comput. Sci.}, 100(1):243--251, 1992.

\bibitem{schachinger95b}
W.~Schachinger.
\newblock On the variance of a class of inductive valuations of data structures
  for digital search.
\newblock {\em Theoret. Comput. Sci.}, 144(1-2):251--275, 1995.
\newblock Special volume on mathematical analysis of algorithms.

\bibitem{schachinger01a}
W.~Schachinger.
\newblock Asymptotic normality of recursive algorithms via martingale
  difference arrays.
\newblock {\em Discrete Math. Theor. Comput. Sci.}, 4(2):363--397 (electronic),
  2001.

\bibitem{szpankowski88b}
W.~Szpankowski.
\newblock The evaluation of an alternative sum with applications to the
  analysis of some data structures.
\newblock {\em Inform. Process. Lett.}, 28(1):13--19, 1988.

\bibitem{szpankowski91a}
W.~Szpankowski.
\newblock A characterization of digital search trees from the successful search
  viewpoint.
\newblock {\em Theoret. Comput. Sci.}, 85(1, Algorithms Automat. Complexity
  Games):117--134, 1991.

\bibitem{szpankowski2001a}
W.~Szpankowski.
\newblock {\em Average case analysis of algorithms on sequences}.
\newblock Wiley-Interscience Series in Discrete Mathematics and Optimization.
  Wiley-Interscience, New York, 2001.
\newblock With a foreword by Philippe Flajolet.

\bibitem{whittaker1927a}
E.~T. Whittaker and G.~N. Watson.
\newblock {\em A course of modern analysis. An introduction to the general
  theory of infinite processes and of analytic functions; with an account of
  the principal transcendental functions}.
\newblock Cambridge Mathematical Library. Cambridge University Press,
  Cambridge, fourth edition, 1927.

\end{thebibliography}
\end{document}